\documentclass[12pt,a4paper,twoside]{amsart} 
\usepackage{a4}

\usepackage{fancyhdr}
\usepackage[active]{srcltx}
\usepackage{amsmath,amssymb}
\usepackage{longtable}
\usepackage[arrow,matrix,curve]{xy}
\usepackage{amstext}
\usepackage{amssymb}
\usepackage{amsfonts}
\usepackage{amssymb}
\usepackage{amsmath}
\usepackage{graphicx}
\usepackage{afterpage,float,amsmath}
\usepackage{graphicx}
\usepackage{enumitem}
\usepackage{upgreek}
\usepackage{tikz}
\usepackage{multirow}
\usepackage{caption} 
\usepackage{mathrsfs}
\usepackage{amsthm}
\usepackage{hyperref} 

\usepackage{tikz-cd}
\usepackage{tikz}
\usepackage{tkz-graph}

\theoremstyle{plain}
 \newtheorem{thm}{Theorem}[section]
 \newtheorem{cor}{Corollary}[section]
 \newtheorem{lem}{Lemma}[section]
 \newtheorem{prop}{Proposition}[section]
 
 \newtheorem{rem}{Remark}[section]
 \newtheorem{ex}{Example}[section]
 
 \newtheorem*{con*}{Conjecture}
  \numberwithin{equation}{section}

\DeclareMathOperator{\diag}{diag}
\DeclareMathOperator{\spa}{span}

\newcommand{\Hom}{\mathrm{Hom}}
\newcommand{\End}{\mathrm{End}}

\newcommand{\ses}[3]{0\rightarrow #1\rightarrow #2\rightarrow#3\rightarrow 0}
\newcommand{\Fu}{\mathfrak{F}}
\newcommand{\Fp}{\mathrm{F}^+}
\newcommand{\Fm}{\mathrm{F}^-}
\newcommand{\Ft}{\mathfrak{U}}
\newcommand\Po{\mathcal{S}}
\newcommand\Pop{{\mathcal{S}^{op}}}
\newcommand\Poe{\widehat{\Po}}
\newcommand\Q[1]{Q({#1})}
\newcommand\Gr{\mathrm{Gr}}
\newcommand\Mi[1]{C_{#1}}
\newcommand\Mir[1]{{C^{\circ}_{#1}}}
\newcommand\Ti[1]{\widehat{C_{#1}}}
\newcommand\Mp[1]{\mathbb M_{#1}}
\newcommand\Ma{\mathbb M_n}
\newcommand\dimv{\mathbf{dim}\,}
\newcommand\cdn{\mathbf{cdn}\,}

\newcommand\V{\mathbf{V}}
\newcommand\W{\mathbf{W}}
\newcommand\U{\mathbf{U}}
\newcommand\Cs{S}
\newcommand\Ct{T}

\newcommand\alv{\mathbf{\boldsymbol \alpha}}
\newcommand\bev{\mathbf{\boldsymbol \beta}}
\newcommand\chv{\mathbf{\boldsymbol \chi}}
\newcommand\thv{\mathbf{\boldsymbol \theta}}

\newcommand\Df[1]{d_{#1}}
\newcommand\Bf[1]{b_{#1}}

\newcommand\cats[1]{\mathfrak{sp}_{#1}}
\newcommand\core[1]{\mathrm{core}(#1)}
\newcommand\ucats[2]{\mathfrak{usp}_{{#1},{#2}}}

\newcommand\R[2]{\mathbf{R}_{#2,#1}}
\newcommand\M[2]{\mathcal{M}_{#1,#2}}

\newcommand\Su[1]{K_{#1}}
\newcommand\GL{\mathbf{Gl}}
\newcommand\SL{\mathbf{Sl}}
\newcommand\Ima{\mathrm{Im}\,}

\newcommand\supp{\mathbf{supp}\,}

\newcommand\Zn{\mathbb{Z}}

\newcommand\Cn{\mathbb{C}}
\newcommand\F{\mathbb{F}}

\newcommand\Cox[1]{\mathrm{Cox}_{#1}}
\newcommand\Coxs[1]{\widehat{\mathrm{Cox}}_{#1}}
\newcommand\Refo[1]{r^{0}_{#1}}
\newcommand\Refom[1]{\widehat{r^{0}_{#1}}}
\newcommand\Refl[1]{r^{*}_{#1}}
\newcommand\Reflm[1]{\widehat{r^{*}_{#1}}}
\newcommand\Refu[1]{r_{#1}}
\newcommand\Split[1]{\mathcal L(#1)}
\newcommand\SSum[2]{\Delta_{#1}(#2)}
\newcommand\Gale[1]{\Gamma(#1)}

\begin{document}

\title{Stable representations of posets}

\author{Vyacheslav Futorny}
\address[Vyacheslav Futorny]{Departamento de Matem\'{a}tica Univ. de S\~ao Paulo, Caixa Postal 66281, S\~ao Paulo, SP 05315-970 -- Brazil}
\email{futorny@ime.usp.br}

\author{Kostiantyn Iusenko}
\address[Kostiantyn Iusenko]{Departamento de Matem\'{a}tica Univ. de S\~ao Paulo, Caixa Postal 66281, S\~ao Paulo, SP 05315-970 -- Brazil}
\email{iusenko@ime.usp.br}

\begin{abstract} 
The purpose of this paper is to study stable representations of partially ordered sets (posets) and compare it to the well known theory for quivers. In particular, we prove that every indecomposable representation of a poset of finite representation type is stable with respect to some weight and construct that weight explicitly in terms of the dimension vector. We show that if a poset is primitive then Coxeter transformations preserve stable representations. When the base field is the field of complex numbers we establish the connection between the polystable representations and the unitary $\chv$-representations of posets. This connection explains the similarity of the results obtained in the series of papers.
\end{abstract}

\maketitle

\section*{Introduction}

Representation theory of finite dimensional algebras turned into a vast field of study in the last 40-50 years. It was observed that the subject can be approached combinatorially via representations of posets (due to  L.A.~Nazarova and A.V.~Roiter) and representations of quivers (due to P. Gabriel).
Despite of certain similarities representations of quivers and posets have significant differences. For instance: the category of representation of given quiver is abelian, while the category of representations of given poset is additive; the global dimension of the category of representations of a given quiver is at most one while it can be arbitrary for the posets; the variety of representations of a fixed dimension of a quiver   is affine  while it is projective in the case of posets; etc.

The problem of classifying  representations of ``most'' algebras is wild in a sense that it  is as difficult as the problem of classifying  representations of free algebras, or of any wild quiver (or poset). Nevertheless, one can use geometrical approach (following the ideas of D.~Mumford, e.g. \cite{mfk}) by considering the spaces whose points correspond naturally to isomorphism classes of representations. This is how A.~King in \cite{King} defined the moduli spaces of finite dimensional algebras and quivers (we refer to \cite{reineke} for exhaustive survey of this subject). 

%We would like to mention several works which are strongly related with the objects studied in our paper. 
The concept of stability for an arbitrary abelian category was developed by Rudakov in \cite{Rud}, where the author formalized the stability conditions for representations theory of quivers (due to Schofield \cite{SC91} and King \cite{King}). In particular, the existence of  Harder-Narasimhan and Jordan-H\"{o}lder filtrations was proved in a  general setting and different definitions of stability were compared. The result were applied in the case of algebraic coherent sheaves on a projective variety. Later on Futorny, Jardim and Moura introduced in \cite {FJM08} the concept of a moduli space of (semi)stable objects in an abelian category which has the structure of a projective algebraic variety, and applied this new concept to several important abelian categories in representation theory, like highest weight categories. Motivated by Dirichlet branes in string theory, Bridgeland defined in \cite{Bridgeland07} stability for an arbitrary triangulated category. This concept turned out to be of great importance for various applications (e.g., \cite{MacriShmidt} and references therein). Bridgeland stability conditions are now well established in algebraic geometry, while the computations with stability conditions, particularly in the geometric setting, are quite difficult. We would like to mention Macri's approach to stability conditions on curves in \cite{Macri07} and Bayer and Macri's study of the space of stability conditions for the local projective plane in \cite{BM11}. The approach in these papers has led some authors, for example Dimitrov, Haiden, Katzarkov and Kontsevich, to look at examples in the representation theory of finite dimensional algebras \cite{DHKK}. In this context much less has been done on spaces of Bridgeland stability conditions, although we would like to mention the work of Broomhead, Pauksztello and Ploog \cite{BPP16}, where the authors showed that for the family of finite dimensional algebras (that have discrete derived category) the stability manifold is contractible. In \cite{KS14} Kontsevich and Soibelman studied the wall crossing in the context of Donaldson-Thomas invariants in integrable systems and mirror symmetry. 
The language of scattering diagram turned out to be very useful. Later on in \cite{BST17}, joining the concept of scattering diagrams and their wall-and-chamber structure  described by Bridgeland in \cite{BR17}, the authors give a description of the wall and chamber structure of any finite dimensional algebra (using the theory of $\tau$-tilting). 

%In \cite{KQ15} King and Qiu observed that wall crossing in the stability manifold of derived category associated to Dynkin quiver is closely related to the cluster exchange graph. 
%Triangulated surfaces with no punctures give rise to gentle algebras [4] and the structure of (part of) their stability space, i.e. the wall and chamber structure, can be studied from the viewpoint of support $\tau$-tilting theory [1, 15]. 
  
The purpose of this paper is to study stable objects in the category of subspace representations of a given poset. We would like to mention that in \cite{WY13} the authors already tried to define moduli spaces of posets via moduli spaces of corresponding bound quivers. This rose certain technical problems as, for instance, the category of bound representations of corresponding quiver is ``bigger'' then the category of representation of underlying poset.
One of the goals of the current paper is to develop a general framework to define and study the moduli spaces of posets intrinsically.  To be more precise, let $\F$ be a field and $\Po=\{s_1,\dots,s_t\}$ be a finite poset. A subspace representation of  $\Po$ is a tuple $\V=(V_0;V_s)_{s\in \Po}$, in which $V_0$ is a finite dimensional vector space over $\F$ and $V_i$ are its subspaces such that $V_s\subseteq V_t$ if $s\prec t$ (that is, each representation is a homomorphism from $\Po$ to the poset of all subspaces of $V_0$). All subspace representations of $\Po$ form an additive category denoted by $\cats{\Po}$ (see Section 1 for more details). Considering  (semi)-stable representations  in $\cats{\Po}$ we adopt the definitions and properties from \cite{Rud}, where (semi)-stable objects in an arbitrary abelian category were studied. Following the ideas of \cite{King} for quivers we approach  the classification of representations of posets geometrically. We show that (semi)-stable orbits are connected to the algebraic definition of (semi)-stability in $\cats{\Po}$  and relate unitary $\chv$-representations of posets (see \cite{KruglyakNazarovaRoiter,KruglyakRoiter,SamYus}) to polystable representations of $\Po$. 

Note that in our case the application of ideas and technique from the cited papers requires a special care in our case, since the category $\cats{\Po}$ is not abelian and the variety of all representations of $\Po$ of fixed dimension is projective.

The paper is organized as follows. In Section \ref{secPrem}  we establish the notation and terminology and prove some preliminary statements. In Section \ref{secDefn} we define an algebraic stability (and costability) in $\cats{\Po}$, prove the existence of Harder-Narasimhan and Jordan-H\"{o}lder filtrations (in $\cats{\Po}$) and relate stability with costability (under certain assumptions). Section \ref{secCoxeter} is devoted to the reflection transformations of posets (introduced in \cite{Drozd}). We prove that the corresponding Coxeter transformations preserve stability in the case of primitive posets. Section \ref{secFinite} is devoted to the posets of finite type. We prove (Theorem \ref{TheoremSection3}) that $\Po$ has a finite type if and only if any indecomposable representation of $\Po$ is positively costable, equivalently if and only if any indecomposable representation of $\Po$ is positively stable. This theorem is a consequence of Propositions \ref{thmDScho} and \ref{propPosit} which are analogues of the Schofield's characterization of Schurian roots for quivers (see \cite[Theorem 6.1]{sch}).  In Section \ref{secGeometric} we relate the introduced concept of stability   with the geometric notion and define moduli space of polystable representations of $\Po$ with fixed admissible dimension vector. Namely, we consider the embeddings of the projective variety $\R{\alv}{\Po}$ of all representations of $\Po$ having the admissible dimension $\alv$ into a projective space and prove that the set of (semi)-stable points of the $\SL(\alpha_0)$-action  coincides with the set of (semi)-stable representations in the sense of Section 2. 
In Section \ref{secMoment} we study the moment map of the $\SL(\alpha_0)$-action on $\R{\alv}{\Po}$ when $\F=\mathbb C$. As a consequence of the theorem of Kempf-Ness  we obtain that the symplectic quotient of  $\R{\alv}{\Po}$ can be identified with the moduli space defined in Section \ref{secGeometric}. Also we show the that preimage of $0$ of the moment map is the set of $\chv$-representations (defined in \cite{KruglyakNazarovaRoiter, KruglyakRoiter, SamYus}). In \ref{appA} we prove some additional statements. 
For the convinience of reading the details of proofs of Proposition \ref{thmDScho} and Proposition \ref{propPosit}) are left in Appendix B and C in ArXiv version of this manuscript.   
\iffalse
In Appendix B we describe all exact representations of non-primitive posets of finite type, describe their maximal subcoordinate vectors and state costability condition for each exact representation (this completes the proof of Proposition \ref{thmDScho}).   In Appendix B we describe all quite sincere representations of non-primitive posets of finite type, describe their maximal subdimension vectors and state stability condition for each quite sincere representation (this completes the proof of Proposition \ref{propPosit}).   
\fi

\section*{Acknowledgements} This work was initiated in February of 2015 during the visit of the second author to Centre de Recerca Matematica (Spain, Barcelona) to participate in Research program IRTATCA: ``Interactions between Representation Theory, Algebraic Topology and Commutative Algebra'', the hospitality and working atmosphere is gratefully acknowledged. The second author would like to thank Mark Kleiner and Vladimir Sergeichuk for very helpful discussions and comments.   V.F. is
supported in part by  CNPq grant (304467/2017-0) and by 
Fapesp grant (2014/09310-5). K.I. is
supported in part by Fapesp grants 2014/09310-5, 2015/00116-4 and by CNPq grant 456698/2014-0.

The authors are grateful to the referee whose suggestions were very helpful in improving the presentation.

\section{Notations and Terminology} \label{secPrem}

We fix a field $\F$ which is assumed algebraically closed in Section 4 and $\F=\mathbb C$ in Section 5.
A finite poset $\Po\equiv(\Po,\preceq)$ is given by the set of elements $\{s_1,\ldots,s_n\}$ and a partial order $\preceq$. We assume that elements $s_1,\dots,s_n$ of $\Po$ are enumerated so that $s_i \prec s_j$ implies $i<j$. A poset $\Po$ is said to be \textit{primitive} if it is a disjoint union of finite number of linearly ordered posets. Denote by $\Pop$ the dual poset $\Pop\equiv(\Pop,\preceq^\circ)$, in which $a\preceq^\circ b$ if and only if $b\succeq a$ in $\Po$. The relation $\preceq$ is uniquely defined by the \textit{incidence matrix} $C_{\Po}$ of $\Po$; that is, the integral square $n\times n$ matrix 
$$
	\Mi{\Po} = [c_{st}]_{s,t \in \Po}\in \Mp{\Po}(\Zn)=\Ma(\Zn), \quad \mbox{with}\ c_{st}=\left \{ \begin{array}{c}
	   1,\quad \mbox{for}\ s\preceq t,\\
	   0,\quad \mbox{for}\ s\npreceq t.\\ 
	\end{array} \right.
$$
It is easy to see that $\Mi{\Po}$ is invertible, $\Mi{\Po}^{-1} \in \Ma(\Zn)$ and that $\Mi{\Pop}=\Mi{\Po}^{tr}$ (the transpose of $\Mi{\Po}$). Given a poset $\Po$, by $\Poe$ we understand its enlargement by unique maximal element $0$; that is, $\Poe\equiv(\Poe,\preceq^0)$ with $\Poe \setminus \{0\}=\Po$ and the order $\preceq^0$ is  obvious. The \textit{Tits matrix} $\Ti{\Po}$ and the \textit{reduced incidence matrix} $\Mir{\Po}$ of $\Poe$ are defined as the following bipartite matrices (we use the notation and terminology from \cite{simson}):
$$
	\Ti{\Po}=\left[
\begin{array}{c|c}
  1 & 0 \\ \hline
  -E & \Mi{\Po}^{tr}
\end{array}
\right]\in \Mp{\Poe}(\Zn), \qquad
	\Mir{\Po}=\left[
\begin{array}{c|c}
  1 & 0 \\ \hline
  0 & \Mi{\Po}
\end{array}
\right]\in \Mp{\Poe}(\Zn),
$$
in which $E$ is a $1\times n$ unit matrix.

A \textit{subspace representation} of $\Po$ is a system $\V=(V_0;V_s)_{s\in \Po}$ of subspaces $V_s$ of a finite dimensional vector space $V_0$ such that $V_s\subset V_t$ if $s\prec t$. The vector space $V_0$ will be called the \textit{ambient} space of $\V$. A \textit{morphism} between two representations $\V$ and $\V'$ is a $\F$-linear map $g:V_0\rightarrow V_0'$ such that $g(V_s)\subset V_s'$ for all $s$. Denote by $\cats{\Po}$ the corresponding additive category of all subspace representations of $\Po$.  Interested reader is refereed to \cite{SimsonB} where the systematic (up-to-date) exposition of the representation theory of finite posets is given.

The \textit{dimension vector} of $\V$ is a $\mathbb Z$-function on $\Poe$ given by $\dimv\, \V(s)=\dim V_s$, that is the dimension vector of $\V$ is an element of $\Zn^{\Poe}$. 
We say that $\alv=(\alpha_0;\alpha_s)_{s\in \Po}\in \Zn^{\Poe}$ is \textit{admissible}
$\alv \cdot \Mir{\Po}^{-1}$ is non-negative vector and $\alpha_0 \geq \alpha_s$ for all $s \in \Po$. Clearly, if $\alv$ is an admissible dimension vector then $\alv$ is a dimension vector of some representation of $\Po$. Fixing an admissible dimension vector $\alv=(\alpha_0;\alpha_s)_{s\in \Po} \in \Zn^{\Poe}$ we consider the following projective variety (see Proposition \ref{posvar}),
$$
	\R{\alv}{\Po}=\Big\{ (V_s)_{s\in \Po} \in \prod_{s\in \Po} \Gr(\alpha_s,\alpha_0) \ \Big |\ V_s \subset V_t\, \mbox{ if } \, s\prec t \ \Big\}.
$$
The group $\GL(\alpha_0)$ acts on $\R{\alv}{\Po}$ (diagonally) via the base change so that the orbits of this action are in a bijection with the isomorphisms classes of subspace representations of $\Po$ with the dimension 
$\alv$. 
%In what follows the variety $\R{\alv}{\Po}$ is called \textit{poset variety}.

The \textit{coordinate vector} of $\V$ is a function on $\Poe$, given by 
$$
	\cdn\, V(s)=
\left\{
	\begin{array}{l}
		\dim (V_s/\sum_{t\prec s} V_t),\quad s\neq 0, \\
		\dim V_0, \quad s=0.
	\end{array} \right.	
$$

Two elements $s_1,s_2\in \Po$ form an \textit{arrow} (denoted by $s_1\rightarrow s_2$) if $s_1\prec s_2$ and there is no $t\in \Po$ such that $s_1 \prec t \prec s_2$. We say that a representation $\V$ is \textit{coordinate} if 
$\dimv \V= \cdn \V \cdot \Mir{\Po}$. There are various important examples of coordinate representation. Mention that if for any point $s\in \Po$ the sum $\sum_{t\to s} V_t$ is direct then  
$$
	\cdn\, V(s)=\dim\Big(V_s/\sum_{t\prec s} V_t\Big)=\dim V_s - \sum_{t\rightarrow s} \dim V_t, 
$$  
and hence $\dimv \V= \cdn \V \cdot \Mir{\Po}$. One easily checks that any representation of a primitive poset is coordinate and that any subrepresentation of a coordinate representation is coordinate.

Given $\alv \in \Zn^{\Po}$, we define the support of $\alv$, $\supp\alv$, to be the full subposet of $\Po$ of the elements $\{s: \alv(s)\neq 0\}$. An indecomposable representation $\V$ is called \textit{sincere} (resp. \textit{exact}) if $\supp\dimv \V=\Poe$ (resp. if $\supp\cdn \V=\Poe$).
 
The following two bilinear forms play a fundamental role in studying the category 
$\cats{\Po}$ (cf. \cite{simson})
$$\Df{\Po},\Bf{\Po}:\Zn^{\Poe}\times \Zn^{\Poe} \rightarrow \Zn,$$ 
\begin{align*}
	\Df{\Po}(\alv,\bev)&=\alv \cdot \Ti{\Po}\cdot \bev^{tr}=\sum_{s\in \Po} \alpha_s\beta_s+\sum_{t\prec s \in \Po}\alpha_s\beta_t-\alpha_0\sum_{s\in \Po} \beta_s, \\
	\Bf{\Po}(\alv,\bev)&=\alv \cdot \Mi{\Poe}^{-1} \cdot \bev^{tr}=\sum_{s\in \Po} \alpha_s\beta_s+\sum_{t\prec s \in \Po}c^{-}_{st}\alpha_s\beta_t,	
\end{align*}
where $c^{-}_{st}$ is the $(s,t)$ entry of the matrix $\Mi{\Poe}^{-1}\in \Mp{\Poe}(\Zn)$ inverse to $\Mi{\Poe}$. Both forms have certain homological  (see \cite{simson}, for the details) and geometric interpretations (see \cite{CI,SimsonB}). For instance, $\Bf{\Po}(\alv,\alv)$ satisfies the following Tits-type equality
(see \ref{appA} for details)
\begin{equation}\label{dimPosetVar}
	\Bf{\Po}(\alv,\alv)=\dim \GL(\alpha_0)-\dim \R{\alv}{\Po},
\end{equation}
for any admissible dimension vector $\alv\in \mathbb Z^{\Poe}$. 

\iffalse
\begin{prop} Let $\Po$ be any poset. We have $\Mir{\Po}^{tr}\cdot\Mi{\Poe}^{-1}\cdot\Mir{\Po}=\Ti{\Po}$, that is the matrices $\Mi{\Poe}^{-1}$ and $\Ti{\Po^{op}}$ are $\Zn$-congruent.
\end{prop}
\begin{proof}
The matrix $\Mi{\Poe}^{-1}$ can be written as the following bipartite matrix
$$
	\Mi{\Poe}^{-1}=\left[\begin{array}{c|c}
  1 &  0 \\ \hline
  -\Mi{\Po}^{-1}\cdot E & \Mi{\Po}^{-1}
\end{array}\right].
$$
Therefore
\begin{align*}
	\Mir{\Po}\cdot\Mi{\Poe}^{-1}\cdot\Mir{\Po}^{tr}&=
\left[
\begin{array}{c|c}
  1 & 0 \\ \hline
  0 & \Mi{\Po}
\end{array}
\right]\cdot	
	\left[\begin{array}{c|c}
  1 &  0 \\ \hline
  -\Mi{\Po}^{-1}\cdot E & \Mi{\Po}^{-1}
\end{array}\right]\cdot \left[
\begin{array}{c|c}
  1 & 0 \\ \hline
  0 & \Mi{\Po}
\end{array}
\right]^{tr}\\
&=\left[
\begin{array}{c|c}
  1 & 0 \\ \hline
  -E & I_n
\end{array}
\right]\cdot\left[
\begin{array}{c|c}
  1 & 0 \\ \hline
  0 & \Mi{\Po}^{tr}
\end{array}
\right]=\left[
\begin{array}{c|c}
  1 & 0 \\ \hline
  -E & \Mi{\Po^{tr}}
\end{array}
\right]=\Ti{\Po}.
\end{align*}
\end{proof}
\fi

\begin{prop} \label{corEqualityofForms}
	If $\V,\W\in \cats{\Po}$ are two coordinate representations hence
	$$
		\Bf{\Po}(\dimv \V,\dimv \W)=\Df{\Po}(\cdn \V,\cdn \W).
	$$
\end{prop}
\begin{proof}
As $\V$ and $\W$ are coordinate representations we have 
	$\dimv \V= \cdn \V \cdot \Mir{\Po}$ and $\dimv \W= \cdn \W \cdot \Mir{\Po}$. Also it is easy to check (see also \cite[Proposition 3.13(a)]{SimsonB}) the matrices $\Mi{\Poe}^{-1}$ and $\Ti{\Po}$ are $\Zn$-congruent via
$$
\Mir{\Po}\cdot\Mi{\Poe}^{-1}\cdot\Mir{\Po}^{tr}=\Ti{\Po}.
$$
Thus, 
\begin{align*}
	\Bf{\Po}(\dimv \V,\dimv \W)&=\dimv \V\cdot \Mi{\Poe}^{-1}\cdot (\dimv \W)^{tr}\\
	&=(\cdn \V \cdot \Mir{\Po})\cdot \Mi{\Poe}^{-1}\cdot (\cdn \W \cdot \Mir{\Po})^{tr}\\
	&=\cdn \V\cdot (\Mir{\Po}\cdot \Mi{\Poe}^{-1}\cdot \Mir{\Po}^{tr})\cdot (\cdn \W)^{tr}\\
	&=\cdn \V\cdot\Ti{\Po}\cdot (\cdn \W)^{tr}\\
	&=\Df{\Po}(\cdn \V,\cdn \W).
\end{align*}\end{proof}

Recall (see \cite[Section 2.1]{ladkani} and \cite[Section 2]{simson}) that a matrix $A\in \Ma(\Zn)$ is called \textit{$\mathbb Z$-regular} if $\det A
\neq 0$ and its Coxeter matrix $\Cox{A}$ defined as $\Cox{A}=-A^{-1}\cdot A^{tr}$ lies in $\Ma(\Zn)$.  
Given a poset $\Po$ with $n$ elements, define the following reflection matrices
\begin{equation*} %\label{reflectionMatrices}
	\Refo{\Po}=\left[
\begin{array}{c|c}
  -1 & 0 \\ \hline
  E & I_n
\end{array}
\right],
\qquad 
	\Refl{\Po}=\left[
\begin{array}{c|c}
  1 & E^{tr} \\ \hline
  0 & -I_n
\end{array}
\right], \qquad
	\Refu{\Po}=\left[
\begin{array}{c|c}
  -1 & 0 \\ \hline
  0 & I_n
\end{array}
\right],
\end{equation*}
and
$$
	\Refom{\Po}=\Refu{\Po}\cdot \Refo{\Po}\cdot \Refu{\Po}, \qquad \Reflm{\Po}=\Refu{\Po}\cdot \Refl{\Po}\cdot \Refu{\Po}.
$$
By $\Cox{\Po}$ and $\Coxs{\Po}$ we denote the Coxeter matrix of $\mathbb Z$-regular matrices $\Mi{\Poe}$ and $(\Mir{\Po^{op}}\cdot \Reflm{\Po})^{-1}$ respectively. One checks that  
$$
	\Cox{\Po}=\left[
\begin{array}{c|c}
  -1 & -E^{tr} \\ \hline
  \Mi{\Po}^{-1}\cdot E & \Mi{\Po}^{-1}\cdot (E\cdot E^{tr}-\Mi{\Po}^{tr})
\end{array}
\right],
$$
$$
	\Coxs{\Po}=\left[
\begin{array}{c|c}
  -1 + E^{tr}\cdot \Mi{\Po}^{-1}\cdot E & -E^{tr}\cdot \Mi{\Po}^{-1}\cdot E \\ \hline
  -\Mi{\Po}^{tr}\cdot \Mi{\Po}^{-1}\cdot E & -\Mi{\Po}^{tr}\cdot \Mi{\Po}^{-1}
\end{array}
\right],
$$
and that these matrices have the following factorization:
\begin{equation}\label{Coxeterfactorization}
\begin{split}
	\Cox{\Po}&=\Mir{\Po}^{-1}\cdot \Refo{\Po}\cdot \Mir{\Po^{op}}\cdot \Refl{\Po},\\
	\Cox{\Po}^{-1}&=\Refl{\Po}\cdot (\Mir{\Po^{op}})^{-1}\cdot \Refo{\Po}\cdot \Mir{\Po},\\
	\Coxs{\Po}&=\Mir{\Po^{op}}\cdot \Reflm{\Po}\cdot \Mir{\Po}^{-1}\cdot \Refom{\Po},\\	
	\Coxs{\Po}^{-1}&=\Refom{\Po}\cdot \Mir{\Po}\cdot \Reflm{\Po}\cdot (\Mir{\Po^{op}})^{-1}.\\	
\end{split}
\end{equation}

\section{Stable representations of posets.} \label{secDefn}

\subsection{Definitions and properties}

The notion of \textit{stability} in an abelian category was defined in \cite{Rud}. 
Given an abelian category $\mathfrak A$ and a function $\theta:K_0(\mathfrak A)\rightarrow \Zn$, an object $X\in \mathfrak A$ is called \textit{stable} if $\theta(X)=0$ and $\theta(Y)<0$ for any proper subobject $Y$ of $X$. Our first aim is to define stable objects in $\cats{\Po}$. We adopt the definition above as well as the proof of some results from \cite{Rud} to additive case.

First we define \textit{proper} subobjects in $\cats{\Po}$. 
A morphism $g: \U \to \V$ in $\cats{\Po}$ is said to be \textit{proper} if, 
for all $s\in \Po$, $g(U_s)=V_s\cap g(U_0)$. Given a representation $\V=(V_0;V_s)_{s\in \Po}$ and a subspace $K\subset V_0$, one checks that 
$\V_K=(K;V_s\cap K)_{s\in \Po}$ is the unique subrepresentation of $\V$ with the ambient space $K$ for which the inclusion $\V_K \hookrightarrow \V$ gives a proper monomorphism $\V_K \to \V$. In what follows by \textit{proper subrepresentation} of $\V$ we mean a representation of the form $\V_K$ where $K$ is a proper subspace of the ambient space of $\V$.

\begin{rem}
Generally, given a representation $\V$ and its subrepresentation $\W$ in $\cats{\Po}$, the quotient $\V/\W$ does not need to belong to $\cats{\Po}$. Nevertheless, in the case when $\W=\V_K$ is a proper subrepresentation we have  $\V/\V_K \in \cats{\Po}$ (see Proposition \ref{quotientincats}). 
\end{rem}

The map $\V \longmapsto \dimv \V$ gives rise to an isomorphism between the Grothendieck group $K_0(\cats{\Po})$ and $\Zn^{\Poe}$. Fixing a form $\thv \in \Hom(\Zn^{\Poe},\Zn)$ we say that $\V\in \cats{\Po}$ is \textit{$\thv$-stable} (resp. \textit{$\thv$-semistable}) if $\thv(\dimv(\V))=0$ and 
$$\thv(\dimv(\W))<0\qquad (\mbox{resp.}\ \leq),$$ for any proper subrepresentation 
$\W$ of $\V$. 

This definition is equivalent to the following. Fixing a basis in $\Hom(\Zn^{\Poe},\Zn)$, we will regard $\thv$ as a vector $\thv=(\theta_0;\theta_s)_{s\in \Po}\in \Zn^{\Poe}$, so that $\thv(\dimv \V)$ simply means $\thv\cdot \dimv \V^{tr}$. Define the 
\textit{$\mu_\thv$-slope} of $\V\in \cats{\Po}$ as 
$$
	\mu_\thv(\V)=\dfrac{1}{\dim V_0}\sum_{s\in \Po} \theta_s \dim V_s.
$$
We say that $\V\in \cats{\Po}$ is \textit{$\mu_\thv$-stable} (resp. \textit{$\mu_\thv$-semistable}) if 
$$
	\mu_\thv(\W)<\mu_\thv(\V)\qquad (\mbox{resp.}\ \leq)
$$
for any proper subrepresentation $\W$ of $\V$. Below we establish several standard properties of semistable objects adopting the proofs of some results from \cite{Hille,Rud} to the case of additive category $\cats{\Po}$.

\begin{prop} \label{seesaw}
Let $\ses{\W}{\V}{\U}$ be an exact sequence of representations and let $\thv$ be a weight. Then the following conditions are equivalent:
\begin{enumerate}
	\item[(1)]	$\mu_\thv(\W)\leq \mu_\thv(\V)$,
	\item[(2)]  $\mu_\thv(\W)\leq \mu_\thv(\U)$,
	\item[(3)]  $\mu_\thv(\V)\leq \mu_\thv(\U)$.
\end{enumerate}
\end{prop}
\begin{proof}
	The proof is similar to the proof of \cite[Lemma 2.1]{Hille} and \cite[Lemma 2.6]{Hu}.
\end{proof}

\begin{prop} \label{maximal}
Let $\thv$ be any weight. Each representation $\V$ has a unique subspace $K\in V_0$ such that the subrepresentation $\W=\V_K$ satisfies:
\begin{itemize}
	\item[(1)] the value of $\mu_\thv(\W)$ is maximal among all proper subrepresentations of $\V$, and
	\item[(2)] $\W$ is maximal among all proper subrepresentations which have the maximal value $\mu_\thv(\W)$.
\end{itemize} 
\end{prop}

\begin{proof} 
	Since $\cats{\Po}$ is noetherian, the existence of a representation $\W$ with $(1)$ and $(2)$ follows. We prove the uniqueness.  
	Let $\W_1$ and $\W_2$ be two non-isomorphic representations satisfying $(1)$ and $(2)$. Consider the following short exact sequence:
$$
	\ses{\W_1\cap \W_2}{\W_1\oplus \W_2}{\W_1+\W_2}.
$$
By  $(1)$ we get $\mu_\thv(\W_1\cap \W_2)\leq \mu_\thv(\W_1)=\mu_\thv(\W_2)$ and $\mu_\thv(\W_1+\W_2)=\mu_\thv(\W_1)=\mu_\thv(\W_2)$. Therefore $\W_1=\W_1+\W_2=\W_2$ by $(2)$.
\end{proof}
Obviously the unique subrepresentation $\W$ from Proposition \ref{maximal} is $\mu_\thv$-semistable. 

The existence of Harder-Narasimhan and Jordan-H\"{o}lder filtrations with (semi)stable factors were proved in \cite[Theorem 2,3]{Rud} for quite general abelian category and in \cite[Theorem 2.5]{Hille} for representations of quivers. Adopting the ideas from cited papers, we provide similar statements in our case. 

\begin{prop}[Harder-Narasimhan filtration] For any $\V=(V_0;V_s)_{s \in \Po}\in \cats{\Po}$ there is a unique filtration (of vector subspaces)
$$
	0=K^0\subset K^1\subset \dots \subset K^{h}=V_0,	
$$
which induces a filtration of $\V$
\begin{equation*} %\label{HNfiltration}
	0=\V^{0}\subset \V^{1}\subset \dots \subset \V^{h}=\V,	
\end{equation*}
in which $\V^{i}=\V_{K^i}=(K^i;V_s\cap K^i)_{s\in \Po}$, such that:

\begin{itemize}
\item[(1)] $\V^{i}/\V^{i-1}$ are $\mu_\thv$-semistable, and
\item[(2)] $\mu_\thv(\V^{i}/\V^{i-1})>\mu_\thv(\V^{i+1}/\V^{i})$ for all $i=1,\dots,h-1$.
\end{itemize} 
\end{prop}

\begin{proof}
 By Proposition \ref{maximal} there exists a unique $\V^{1}\subset \V$ with maximal slope $\mu_\thv$. If $\V^{1}=\V$ then $\V$ is $\mu_\thv$-semistable and we are done. Otherwise we get the first step of the filtration
 $$
 	0\subset \V^1\subset \V.
 $$
Now consider $\V/\V^{1}$. If it is $\mu_\thv$-semistable then Proposition \ref{seesaw} implies  $\mu_\thv(\V^1)>\mu_\thv(\V/{\V^1})$. If $\V/\V^{1}$ is not $\mu_\thv$-semistable then we apply the above procedure to produce a unique linear subspace $K^2$ ($K^{1}\subset K^{2}\subset V_0$) with $\V^2/\V^1$  $\mu_\thv$-semistable. As $\mu_\thv(\V^1)>\mu_\thv(\V^2)$,  Proposition \ref{seesaw} implies that $\mu_\thv(\V^1)>\mu_\thv(\V^2/{\V^1})$. Then by induction we get the desired filtration. The uniqueness of the filtration is clear from the proof.
\end{proof}

\begin{prop}[Jordan-H\"{o}lder filtration] %\label{JordanHolder}
	For any $\mu_\thv$-semistable $\V=(V_0;V_s)_{s \in \Po}\in \cats{\Po}$ there is a  filtration (of vector subspaces)
$$
	0=K^0\subset K^1\subset \dots \subset K^{h}=V_0,	
$$
which induces a filtration of $\V$
\begin{equation*} 
	0=\V^{0}\subset \V^{1}\subset \dots \subset \V^{h}=\V,	
\end{equation*}
in which $\V^{i}=\V_{K^i}=(K^i;V_s\cap K^i)_{s\in \Po}$, such that:
\begin{itemize}
\item[(1)] $\V^{i}/\V^{i-1}$ are $\mu_\thv$-stable, and
\item[(2)] $\mu_\thv(\V^{i}/\V^{i-1})=\mu_\thv(\V^{i+1}/\V^{i})$ for all $i=1,\dots,h-1$.
\end{itemize} 
\end{prop}

\begin{proof}

If $\V$ is stable we are done. Otherwise, let $\W$ be a maximal subspace such that $\mu_\thv(\V_{W})=\mu_\thv(\V)$. Then $\V_{W}$ is $\mu_\thv$-semistable. Using Proposition \ref{seesaw} we have that $\V/\V_{W}$ is $\mu_\thv$-stable and $\mu_\thv(\V/\V_{W})=\mu_\thv(\V)$. Repeating the same procedure for $\V_W$ we get a desired filtration.
\end{proof}

\begin{prop} \label{PropSchur}
Suppose that both $\V,\V' \in \cats{\Po}$ are $\mu_\thv$-semistable
and $g:\V\rightarrow \V'$ is a non zero morphism. Then $\mu_\thv(\V)\leq \mu_\thv(\V')$. 
\end{prop}

\begin{proof}
Consider the proper induced morphism $g':\V\rightarrow \Ima g$. The kernel $\ker g'$ is a subrepresentation of $\V$ and we have the following short exact sequence 
	$$
		\ses{\ker g'}{\V}{\Ima g}.
	$$
 As $\Ima g$ is a subrepresentation of $\V'$ we have $\mu_\thv(\Ima g)\leq \mu_\thv(\V')$ (since $\V'$ is semistable). Assuming $\mu_\thv(\V)>\mu_\thv(\V')$ we also have that $\mu_\thv(\Ima g)< \mu_\thv(\V)$. Therefore, by Proposition \ref{seesaw},   $\mu_\thv(\ker g')>\mu_\thv(\V)$. But this contradicts the $\mu_\thv$-semistability of $\V$.
\end{proof}

\begin{cor} \label{CorSchur}
If $\V$ is $\mu_\thv$-stable then $\End(\V)$ is a division algebra over $\F$. In particular any stable object is indecomposable. Also if  $\F$ is algebraically closed then $\End(\V)\simeq \F$ and any stable object is Schurian.
\end{cor}

\subsection{Costability}
In what follows we relate $\thv$-stability with the following notion. 
Let $\thv \in \Zn^{\Poe}$. We say that $\V\in \cats{\Po}$ is \textit{$\thv$-costable} (resp. \textit{$\thv$-cosemistable}) if $\thv(\cdn(\V))=0$ and 
$$\thv(\cdn(\W))<0\qquad (\mbox{resp.}\ \leq)$$ 
for any proper subrepresentation $\W$ of $\V$. 

Note that in general the function $\cdn$ is not additive (unless $\Po$ is primitive), therefore the notion of costability does not posses the properties proven in Section 2.1. Nevertheless, if $\V$ is coordinate then costability is related to stability as the following proposition shows.

\begin{prop} \label{coordStable}
Let $\V\in \cats{\Po}$ be a coordinate representation. Then $\V$ is $\thv$-stable if and only if $\V$ is $\thv'$-costable with $\thv'=\thv\cdot \Mir{\Po^{op}}$.
\end{prop}

\begin{proof}
	If $\V$ is coordinate then for any subrepresentation $\W$ of $\V$ (not necessarily proper) we have 
	  $$\dimv \W=\cdn \W \cdot \Mir{\Po},$$ and hence 
\begin{align*}
		\thv(\dimv \W)&=\thv\cdot (\dimv \W)^{tr}=\thv\cdot (\cdn \W \cdot \Mir{\Po})^{tr}\\
					   &=\thv\cdot \Mir{\Po}^{tr} \cdot (\cdn \W)^{tr}=\thv\cdot 			\Mir{\Po^{op}}\cdot (\cdn \W)^{tr}\\ 
					   &=\thv'\cdot (\cdn \W)^{tr}=\thv'(\cdn \W).
\end{align*}
Therefore 
$$\thv(\dimv \V)=0 \quad \mbox{if and only if} \quad \thv'(\cdn \V)=0,$$ 
and for any proper subrepresentation $W$
$$\thv(\dimv \W)<0 \quad \mbox{if and only if} \quad \thv'(\cdn \W)<0.$$ 
The claim follows.
\end{proof}

\begin{cor}\label{NonSchurCostab} Let  $\V\in \cats{\Po}$ be an indecomposable coordinate representation whose endomorphism ring is not a division algebra. Then $\V$ can not be costable with respect to  some form. 
\end{cor}

\begin{proof}
	If $\V$ is costable then $\V$ is stable by Proposition \ref{coordStable} and, therefore its endomorphism ring is a division algebra by Corollary \ref{CorSchur}. This is a contradiction.
\end{proof}

\subsection{Positive stability} 
We say that representation $\V$ is \textit{positively stable} (respectively \textit{costable}) if there exists a form $\thv\in \Zn^{\Poe}$ with $\theta_s>0, s\in \Po$ such that $\V$ is $\thv$-stable (respectively costable). 

Note that if a representation is $\thv$-stable this does not imply in general that $\theta_s>0, s\in \Po$. For instance, if $\V$ is a general representation of a poset $\Po$ with $4$ incomparable elements in dimension $\alv=(2;1,1,1,1)$, then $\V$ is $(-5;4,4,4,-2)$-stable but the form is not positive.

We need the following extension of stability. Given a $\thv$-stable representation $\V$ of a poset $\Po$ and any representation $\widetilde \V \in \cats{\widetilde \Po}$ such that $\Po$ is a subposet $\widetilde \Po$ and $\widetilde \V\big\vert_{\Po}=\V$, define
 $\widetilde \theta_0=\theta_0$, $\widetilde \theta_s=\theta_s$ if $s\in \Po$ and $\widetilde \theta_s=0$ if $s\notin \Po$. Obviously $\widetilde \V$ is $\widetilde \thv$-stable. We prove even a stronger connection.

\begin{prop} \label{stLift}
	Let $\V\in \cats{\Po}$ be a positively stable representation with form $\thv$.
	Any representation $\widetilde \V \in \cats{\widetilde \Po}$, such that $\Po$ is a subposet $\widetilde \Po$ and $\widetilde \V\big\vert_{\Po}=\V$ is positively stable with some form $\widetilde \thv$.
\end{prop}

\begin{proof}
	We prove the statement for the case when 
	$\Po=\widetilde \Po \setminus  \{\tilde s\}$ (the remaining part follows by induction). If $\V=(V_0;V_s)_{s\in \Po}$, we view the representation $\widetilde \V$ as $\widetilde \V=(V_0;V_{\tilde s},V_s)_{s\in \Po}$. Let $\U$ be a proper subrepresentation of $\V$ such that $\thv(\dimv \U)$ is maximal. As $\thv(\dimv \U)<0$ we have that $\thv(\dimv \U)\leq -1$. Hence defining $\thv'=(\dim V_{\tilde s}\dim V_0+1)\cdot \thv$, we have that $\V$ is $\thv'$-stable and for any proper subrepresentation $\W$, 
	$$\thv'(\dimv \W)=(\dim V_{\tilde s}\dim V_0+1)\cdot \thv(\dimv \W)\leq-\dim V_{\tilde s} \dim V_0-1.$$
Set $\nu=-\dim V_{\tilde s} \dim V_0-1$ and define a form $\widetilde \theta$ as  follows: 
$$
  \widetilde \theta_0=\nu \theta_0-\dim V_s, \qquad
  \widetilde \theta_t=\left\{\begin{array}{c}
  		\theta'_t,\quad  t\neq \tilde s\\ 
  		\dim V_0,\quad  t=\tilde s.
  \end{array}\right.
$$
Then
$$\widetilde \thv(\dimv \widetilde \V)=\nu \thv'(\dimv \V)=0,$$
and for any proper subrepresentation $\widetilde \W=(W_0;W_{\tilde s},W_s)_{s\in \Po}$ we have 
\begin{align*}
	\widetilde \thv(\dimv \widetilde \W)&=\thv'(\dimv \W)+\dim V_0\dim{\widetilde W_{\tilde s}}-\dim V_{\tilde s}\dim{\widetilde W_0}\\
	&\leq -\dim V_{\tilde s}\dim V_0-1+\dim V_0\dim{\widetilde W_{\tilde s}}-\dim V_{\tilde s}\dim{\widetilde W_0}\\
	&\leq -\dim V_{\tilde s}\dim V_0-1+\dim V_0\dim V_{\tilde s}<0.
\end{align*}
As  $\widetilde \thv$ has all positive components except $\widetilde \theta_0$, the claim follows.
\end{proof}

Similarly one proves the following

\begin{prop} \label{cstLift}
Let $\V\in \cats{\Po}$ be a positively costable representation with form $\thv$.
Any representation $\widetilde \V \in \cats{\widetilde \Po}$ such that $\Po$ is a subposet $\widetilde \Po$, $\widetilde \V\big\vert_{\Po}=\V$ and $\cdn \widetilde \V \big\vert_{\Po}=\cdn \V$ is positively costable with some form $\widetilde \thv$.
\end{prop}

\section{Reflections and stability} \label{secCoxeter}

In this section we discuss how the Coxeter transformations for posets defined in \cite{Drozd} act on (semi)stable representations.

\subsection{Reflections and Coxeter transformations} Reflections for posets were defined in \cite{Drozd} using the bimodule language of poset representations. Below we recall this construction.
Given $\V=(V_0;V_s)_{s\in \Po}$ define $\Cs(\V)=(V_0^{*};V_s^{\bot})$, where $V_s^{\bot}=\{\varphi \in V_0^{*}\ | \varphi(V_s)=0 \}.$ Obviously, $\Cs(\V)\in \cats{\Po^{op}}$, $\Cs^2(\V)\cong\V$ and 
$$\dimv \Cs(\V)=\dimv \V \cdot \Refl{\Po}=(\dim V_0;\dim V_0-\dim V_s)_{s\in \Po}.$$

The second reflection is defined as follows. Given a representation $\V=(V_0;V_s)_{s\in \Po} \in \cats{\Po}$ define by 
$\Split{\V}$ the following family of systems of subspaces:
$$
	\Split{\V}=\Big\{(V_0;V_s')_{s\in \Po} \ | \  V_s'\subset V_0,\ V_s'\cong V_s/\sum_{t\prec s} V_t,\ \sum_{t\preceq s} V_t'=V_s  \Big\}.
$$
Note that having any system of subspaces $(V_0;V_s)_{s\in \Po}$ indexed by a poset $\Po$, one can form a representation $\SSum{\Po}{(V_0;V_s)}=(V_0;\tilde V_s)_{s\in \Po}$ of $\Po$ setting $V_s=\sum_{t\preceq s}V_t$. In particular, it  follows that for any  
$(V_0;V_s')_{s\in \Po} \in \Split{\V}$ we recover $\V$ as $\V=\SSum{\Po}{(V_0;V_i')}$.

Now let $(V_0;V_i)_{i\in I}$ be any system of subspaces in $V_0$ indexed by a finite set $I$ (considered as a poset with trivial partial order), such that the map 
\begin{align*}
	\varphi:\bigoplus_{i \in I} V_i& \to V_0\\
	          (v_i)_{i\in I} & \mapsto \sum_{i\in I} v_i, \qquad v_i\in V_i
\end{align*}
is surjective. 
Consider the following short exact sequence
$$
\xymatrix@R2pt@C25pt{
	0 \ar@{->}[r]& \ker \varphi \ar@{->}[r]^>>>>>{\psi} & \bigoplus_{i \in I} V_i \ar[r]^<<<<{\varphi} & V_0 \ar@{->}[r]& 0
	}
$$
 where $\psi(y)=(p_i(y))_{i\in I}$, $p_i:\ker \varphi \to V_i$.  	
Dualizing we get the sequence 
$$
\xymatrix@R2pt@C25pt{
	0 \ar@{->}[r]& V_0^* \ar@{->}[r]^{\varphi^*\quad } & \bigoplus_{i \in I} V_i^* \ar[r]^{\psi^*} &(\ker \varphi)^* \ar@{->}[r]& 0.
	}
$$
 Denote by $\Gale{(V_0;V_i)_{i\in I}}$ the system of subspaces $((\ker \varphi)^*;\textrm{Im}(p_i^*))_{i\in I}$. The action of $\Gamma$ on dimensions is the following: $$\dimv\Gamma(V_0;V_i)_{i\in I}=\dimv(V_0;V_i)_{i\in I}\cdot \Refo{I}=(\sum_{i\in I} \dim V_i - \dim V_0; \dim V_i)_{i\in I}.$$

Denote by $E_0$ a simple representation of the form $(\F;0)_{s\in \Po}$ in $\cats{\Po}$.  Let $\V \in \cats{\Po}$ be any representation. Choosing $(V_0;V_s')_{s\in \Po} \in \Split{\V}$ we define $\Ct(\V)\in \cats{\Po^{op}}$ by
\begin{equation*}
	\Ct(\V)=\SSum{\Po^{op}}{\Gamma{(V_0;V_s')_{s\in \Po}}}.
\end{equation*}
One easily checks that if $\V$ does not contain $E_0$ as a direct summand then the map $\varphi$ above is surjective, and therefore $\Ct(\V)$ is well-define (in particular, it does not depend on the choice of representative in $\Split{\V}$) and we have $\Ct^2(\V)\cong\V$.
 If  $\V$ and $\Ct(\V)$ are coordinate (for instance, this is always the case when $\Po$ is primitive) we have that
$$
	\dimv\Ct(\V)=\dim \V \cdot (\Mir{\Po})^{-1} \cdot \Refo{\Po} \cdot \Mir{\Po^{op}}.
$$
The compositions (when it makes sense) of reflections $\Cs$, $\Ct$ will be denoted by
$$
	\Fp=S \circ T, \qquad \Fm=T \circ S
$$ 
and called \textit{Coxeter transformations}.

\begin{rem}
We refer to \cite[Chapter 11]{SimsonB} and \cite{dpSimson} for the relation between transformation $\Fp$ and Auslander-Reiten translate for poset.% (cf. \cite[Chapter 11]{SimsonB} and \cite{dpSimson}).
\end{rem}

Using formulas \eqref{Coxeterfactorization} one checks that the transformations $\Fp$ and $\Fm$ act on dimensions of representations (in coordinate cases) as follows:
$$
\Fp(\dimv \V)=\dimv \V\cdot \Cox{\Po},\qquad \Fm(\dimv \V)=\dimv \V\cdot (\Cox{\Po})^{-1}.
$$

\subsection{Stability behaviour of reflections}

First we show that $\Cs$ maps (semi)stable representation into (semi)stable ones.

\begin{lem}
Let $\thv=(\theta_0;\theta_s)_{s\in \Po}$ be a weigth. A representation $\V=(V_0;V_s)_{s\in \Po}$ is $\thv=(\theta_0;\theta_s)_{s\in \Po}$-(semi)stable iff the representation $\Cs(\V)=(V_0^{*};V_s^{\bot})_{s\in \Po}$ is $\Cs(\thv)$-(semi)stable, where
$$
\Cs(\thv):=\thv\cdot \Refom{\Po} =\Big(-\sum_{s\in \Po}\theta_s-\theta_0;\theta_s \Big )_{s\in \Po}. 
$$
\end{lem}

\begin{proof}
Notice that
\begin{align*}
	\thv(\dimv \V)&=\sum_{s\in \Poe} \theta_s\dim V_s,\\
	\Cs(\thv)(\dimv \Cs(\V))&=(-\sum_{s\in \Po}\theta_s-\theta_0)\dim V_0+\sum_{s\in \Po} \theta_s(\dim V_0-\dim V_s)\\
	&=-\thv(\dim \V).
\end{align*}
Therefore $\thv(\dimv \V)=0$ iff $\Cs(\thv)(\dimv \Cs(\V))=0$.
Now if $\Cs(\V)=(V_0^{*};V_s^{\perp})_{s\in \Po}$ is not $\Cs(\thv)$-stable, then there exists a subspace $K^\perp$ such that 
\begin{equation} \label{perpNonStab}
\sum_{s\in \Po} \theta_s \dim (V^{\perp}_s \cap K^\perp)-(\sum_{s \in \Po} \theta_s + \theta_0) \dim K^\perp \ge 0. 
\end{equation}
As $\dim (V^{\perp}_s \cap K^\perp)=\dim V_0-\dim V_s-\dim K + \dim (V_s\cap K)$, from \eqref{perpNonStab} we have 
\[
	\sum_{s\in \Po} \theta_s \big(\dim V_0-\dim V_s-\dim K + \dim (V_s\cap K)\big)-\Big(\sum_{s \in \Po} \theta_s + \theta_0\Big) (\dim V_0-\dim K) \ge 0. 
\]
Or, equivalently 
\[
	\sum_{s\in \Po} \theta_s \dim (V_s\cap K) + \theta_0 \dim K \geq 0. 
\]
 Hence $(V_0;V_s)_{s\in \Po}$ is not $\thv$-stable which is a contradiction.
\end{proof}

Assuming that $\Po$ is primitive we prove that the reflection $\Ct$ also maps (semi)stable representation into (semi)stable ones.  

\begin{lem}
A representation $\V$ is $\thv$-(semi)stable iff any representation $\W \in \Split{\V}$ is 
$\thv \cdot \Mir{\Po^{op}}$-(semi)stable.
\end{lem}
\begin{proof}
The proof is similar to the proof of Proposition \ref{coordStable}.
\end{proof}
Similarly we have that if a system of subspaces $(V_0;V_s)_{s\in \Po}$ indexed by a poset $\Po$ (not necessarily a representation) is $\thv$-stable then  $\SSum{\Po}{(V_0;V_s)}$ is $\thv \cdot (\Mir{\Po^{op}})^{-1}$-stable.

\begin{lem}
Let $\thv=(\theta_0;\theta_s)_{s\in \Po}$ be a weight, 
$\V=(V_0;V_s)_{s\in \Po}$  a system of subspaces which does not contain $E_0$ as a direct summand. Then $\V$ is $\thv=(\theta_0;\theta_s)_{s\in \Po}$-(semi)stable iff the system  $\Gale{\V}$ is $\Gale{\thv}$-(semi)stable, where
$$
\Gale{\thv}:=\thv\cdot \Reflm{\Po}=(\theta_0;-\theta_0-\theta_s)_{s\in \Po}.
$$
\end{lem}

\begin{proof}
It is enough to show that if $\V$ is not $\thv$-stable then $\Gale{\V}$ is not $\Gale{\thv}$-stable. One checks that $\thv(\dimv \V)=0$ iff $\Gale{\thv}(\dimv \Gale{\V})=0$. 
If $\V$ is not $\thv$-stable then there exists a proper subspace $K$ in $V_0$ such that
\begin{equation}\label{nonstabGale}
\sum_{s\in \Po}\theta_s\dim(V_s \cap K)+\theta_0 \dim K\geq 0. 
\end{equation}
Consider the following commutative diagram
\begin{equation*}
\begin{tikzcd}
                &   0                 & 0     & 0  &   \\
    0 \arrow{r} &   M \arrow{u} \arrow{r}               & \bigoplus_{s\in \Po}A_s \arrow{u}\arrow{r}    & B \arrow{u}\arrow{r} & 0 \\
    0 \arrow{r} &   \ker \varphi \arrow{u} \arrow{r}               & \bigoplus_{s\in \Po}V_s \arrow{u}\ar[r,"\varphi"]    & V_0 \arrow{u}\arrow{r} & 0 \\
    0 \arrow{r} &   \ker \psi \arrow{u} \arrow{r}                & \bigoplus_{s\in \Po} V_s\cap K \ar[u,hook]\ar[r,"\psi"]    & K \ar[u,hook]\arrow{r}  & 0 \\
                &   0 \arrow{u}                & 0  \arrow{u}   & 0 \arrow{u} &   \\
\end{tikzcd}
\end{equation*}
Dualizing the diagram we get that $M^*$ is a subspace of $\ker \varphi^{*}$. We will show that
\begin{equation} \label{nonstabGale3}
	\sum_{s\in \Po}(-\theta_0-\theta_s)\dim (V_s^*\cap M^*)+\theta_0 \dim M^*\geq 0.
\end{equation}
We have
\begin{equation*}
\begin{split} 
\dim M^*=\dim M&=\dim \ker \varphi- \dim \ker \psi\\
      &=\sum_{s\in \Po} \dim V_s-\dim V_0-(\sum_{s\in \Po} \dim(V_s \cap K)-\dim K).
\end{split}
\end{equation*}
On the other hand $\dim V_s -\dim(V_s \cap K)=\dim A_s=\dim A_s^*$. Using these identities one shows that the inequality \eqref{nonstabGale} is equivalent to 
\begin{equation}\label{nonstabGale2}
\sum_{s\in \Po}(-\theta_0-\theta_s)\dim A_s+\theta_0 \dim M^*\geq 0. 
\end{equation}
As each $A_s^*$ is a subspace in $V_s^*$ and $M^*$, therefore $\dim(V_s^* \cap M^*)\geq \dim A_s^*$. Thus inequality \eqref{nonstabGale2} implies (as each $-\theta_0-\theta_s$ is positive) inequality \eqref{nonstabGale3}. Hence the system of subspaces $\Gale{\V}$ is not $\Gale{\thv}$-stable.

\end{proof}

\begin{prop}
Assume that a representation $\V$ does not contain $E_0$ as a direct summand.
Then $\V$ is $\thv$-(semi)stable iff a representation $\Ct(\V)$ is $\Ct(\thv)$-(semi)stable,
 where
$$
	\Ct(\thv)=\thv\cdot \Mir{\Po^{op}}\cdot \Reflm{\Po} \cdot (\Mir{\Po})^{-1}.
$$
\end{prop}

Combining this proposition with the previous two lemmas  we get
\begin{thm} \label{CoxThm}
Assume that $\Po$ is primitive poset, $\alv=(\alpha_0;\alpha_s)_{s\in \Po}$ the dimension vector and $\thv$  a weight. A representation $\V\neq E_0$ is $\thv$-stable (respectively semistable) iff $\Fp(\V)$ is $\Fp(\thv)$-stable (respectively semistable), where
$$
	\Fp(\thv)=\thv\cdot \Coxs{\Po}.
$$
\end{thm}

In a subsequent work  we will establish similar statements for reflection $\Ct$ in the case of non-primitive posets.

\section{Stability and posets of finite type}  \label{secFinite}

Recall that M.Kleiner in \cite{Kleiner2} (see also \cite[Theorem 10.1]{SimsonB}) showed that a poset $\mathcal S$
has only a finite number of non equivalent indecomposable  representations (that is, the category $\cats{\Po}$ is of finite representation type) if and only if it does not contain a full poset whose Hasse diagram is one of the following
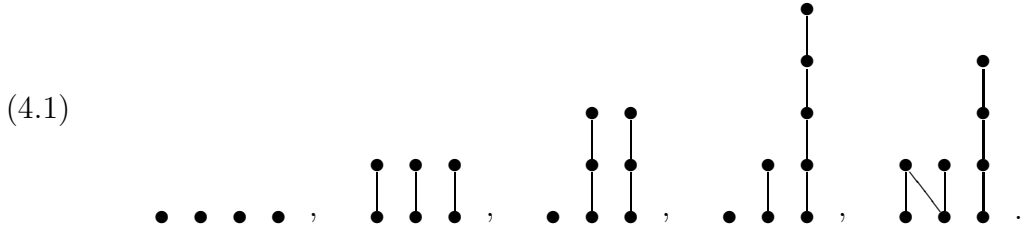
\begin{equation} \label{critial_posets_}
\begin{split}
\entrymodifiers={[o]} \xymatrix @C=0.3cm @R=0.5cm
{
  & & & &  & &  &  &  &  &  & &  &  &  &    \bullet \ar@{-}[d] & & & &
\\
  & & & &  & &  &  &  &  & &  &  &  & & \bullet \ar@{-}[d] &&   &  & \bullet \ar@{-}[d]
\\
  & & & & &  &  &  &  & & \bullet \ar@{-}[d] & \bullet \ar@{-}[d] &  & & & \bullet \ar@{-}[d] & & &  & \bullet \ar@{-}[d]
  \\
  & & & & & \bullet \ar@{-}[d] & \bullet \ar@{-}[d] & \bullet \ar@{-}[d] & &  & \bullet \ar@{-}[d] & \bullet \ar@{-}[d] & & & \bullet \ar@{-}[d] &  \bullet \ar@{-}[d] & & \bullet \ar@{-}[d] \ar@{-}[rd] & \bullet \ar@{-}[d] & \bullet \ar@{-}[d]\\
 \bullet & \bullet & \bullet & \bullet &, \resizebox{0.3cm}{!}{ } & \bullet  & \bullet & \bullet & , \resizebox{0.3cm}{!}{ }  & \bullet &\bullet & \bullet & , \resizebox{0.3cm}{!}{ }  & \bullet & \bullet & \bullet & , \resizebox{0.3cm}{!}{ }  & \bullet & \bullet & \bullet&. }
\end{split}
\end{equation}
Tha posets in list \eqref{critial_posets_} we call \textit{critical}. Note that $\cats{\Po}$ is of tame representation type for each critical poset (see, \cite[Chapter 15] {SimsonB} for  details).

In this section we prove the following

\begin{thm} Let $\Po$ be a finite poset. The following statements are equivalent. \label{TheoremSection3} 
\begin{itemize}
	\item[(a)] The category $\cats{\Po}$ is of finite representation type.
	\item[(b)] Any indecomposable representation of $\Po$ is positively costable.
	\item[(c)] Any indecomposable representation of $\Po$ is positively stable.\end{itemize}
\end{thm}

The implication $(c)\Rightarrow (a)$ follows from Proposition \ref{PropSchur}. Indeed, if $\Poe$ has the infinite representation type then there exist indecomposable representations whose endomorphism ring is not a division algebra. Therefore, they cannot be stable  by Corollary \ref{CorSchur}.

The implication $(b)\Rightarrow (a)$ follows from  Corollary \ref{NonSchurCostab} and the fact that any poset of infinite type has indecomposable coordinate representations whose endomorphism ring is not a division algebra. 

\subsection{Exact posets of finite type}

Recall that a poset $\Po$ is called \textit{exact} if it admits an exact representation.   A complete list of exact posets of finite type and their sincere representations was obtained in \cite{Kleiner2} (see also \cite[Chapter 10.7]{SimsonB}, for corrected list of exact representations). Namely, a non-primitive poset of finite type is exact if and only if it has one of the following forms:
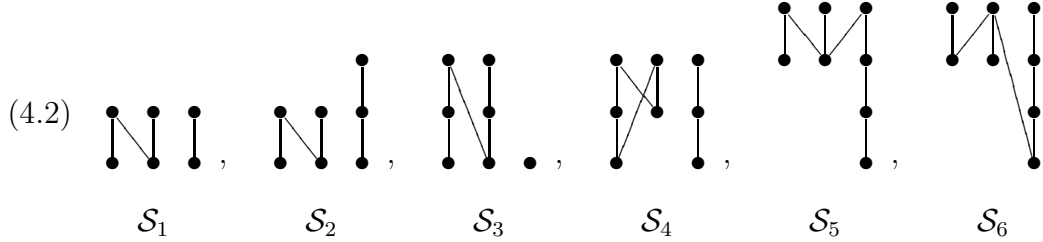
\begin{equation} \label{critial_posets}
\begin{split}
\entrymodifiers={[o]} 
\xymatrix @C=0.22cm @R=0.5cm
{
  & &  & \resizebox{0.3cm}{!}{ } & & &  &  \resizebox{0.3cm}{!}{ }  & & &  &  \resizebox{0.3cm}{!}{ }  &  &  &  &  \resizebox{0.3cm}{!}{ }  & \bullet\ar@{-}[d]\ar@{-}[dr] & \bullet\ar@{-}[d] & \bullet \ar@{-}[d]\ar@{-}[dl]& \resizebox{0.3cm}{!}{ }  & \bullet\ar@{-}[d] & \bullet \ar@{-}[d]\ar@{-}[dddr]\ar@{-}[dl]& \bullet\ar@{-}[d]& \\
  & &  & \resizebox{0.3cm}{!}{ } &   & & \bullet\ar@{-}[d] &  \resizebox{0.3cm}{!}{ }  & \bullet\ar@{-}[d]\ar@{-}[rdd] &\bullet\ar@{-}[d] &  &  \resizebox{0.3cm}{!}{ }  & \bullet\ar@{-}[d]\ar@{-}[dr] & \bullet\ar@{-}[d]\ar@{-}[ddl] & \bullet\ar@{-}[d] &  \resizebox{0.3cm}{!}{ }  & \bullet & \bullet & \bullet\ar@{-}[d] & \resizebox{0.3cm}{!}{ }  & \bullet & \bullet& \bullet\ar@{-}[d]& \\
 \bullet\ar@{-}[d]\ar@{-}[rd] & \bullet\ar@{-}[d] & \bullet\ar@{-}[d] & \resizebox{0.3cm}{!}{ } & \bullet\ar@{-}[d]\ar@{-}[rd]  & \bullet\ar@{-}[d] & \bullet\ar@{-}[d] &  \resizebox{0.3cm}{!}{ }  & \bullet\ar@{-}[d] &\bullet\ar@{-}[d] &  &  \resizebox{0.3cm}{!}{ }  & \bullet\ar@{-}[d] & \bullet & \bullet\ar@{-}[d] &  \resizebox{0.3cm}{!}{ }  &  & & \bullet\ar@{-}[d] & \resizebox{0.3cm}{!}{ }  &  & & \bullet\ar@{-}[d] & \\
 \bullet & \bullet & \bullet &, \resizebox{0.3cm}{!}{ } & \bullet  & \bullet & \bullet & , \resizebox{0.3cm}{!}{ }  & \bullet &\bullet & \bullet & , \resizebox{0.3cm}{!}{ }  & \bullet &  & \bullet & , \resizebox{0.3cm}{!}{ }  &  &  & \bullet &, \resizebox{0.3cm}{!}{ }  &  &  &  \bullet & \\
  & \mathcal S_1 & & &  & \mathcal S_2 & & &
  & \mathcal S_3 & & &  & \mathcal S_4 & & &
  & \mathcal S_5 & & &  & \mathcal S_6 & & &
  }
\end{split}
\end{equation}
For each non-primitive sincere poset $\mathcal S_1,\dots,\mathcal S_6$  we list all its exact representations in the following table. 

%\begin{table}
%\caption{Exact representations of non-primitive posets of finite type}
\def\arraystretch{1.2}
\begin{longtable}{|c|p{10.5cm}|}
\hline
\textit{Poset} & \textit{Exact representations} 
 \\
\hline
 $\mathcal S_1$ & $(K^3;\Su{123},\Su{1,2,3};\Su{1},\Su{1,2};\Su{3},\Su{2,3})$ \\
 \hline
$\mathcal S_2$ & 1) $(K^3;\Su{3},\Su{1,2,3};\Su{123},\Su{13,2};\Su{1},\Su{1,2},\Su{1,2,3})$\\
& 2) $(K^4;\Su{14},\Su{1,2,4};\Su{4},\Su{123,4};\Su{3},\Su{2,3},\Su{1,2,3})$\\
& 3) $(K^4;\Su{14},\Su{1,2,4};\Su{4},\Su{12,23,4};\Su{3},\Su{2,3},\Su{1,2,3})$\\
& 4) $(K^4;\Su{1,24},\Su{1,2,3,4};\Su{4},\Su{123,4};\Su{3},\Su{2,3},\Su{1,2,3})$\\
& 5) $(K^4;\Su{1,24},\Su{1,2,3,4};\Su{4},\Su{12,13,4};\Su{3},\Su{2,3},\Su{1,2,3})$\\
& 6) $(K^5;\Su{15,4},\Su{1,2,4,5};\Su{5},\Su{123,24,5};\Su{3},\Su{2,3},\Su{1,2,3})$\\
& 7) $(K^5;\Su{3,5},\Su{2,3,4,5};\Su{45},\Su{134,24,45};\Su{1},\Su{1,2},\Su{1,2,3,4})$\\
& 8) $(K^5;\Su{1,25},\Su{1,2,3,5};\Su{5},\Su{13,234,5};\Su{4},\Su{2,3,4},\Su{1,2,3,4})$\\
& 9) $(K^5;\Su{1,25},\Su{1,2,3,5};\Su{5},\Su{123,24,5};\Su{3,4},\Su{2,3,4},\Su{1,2,3,4})^*$\\
\hline
$\mathcal S_3$ & $(K^4;\Su{4},\Su{1,4},\Su{1,2,3,4};\Su{3},\Su{2,3},\Su{1,2,3};\Su{123,24})$\\
\hline
$\mathcal S_4$ & $(K^4;\Su{4},\Su{3,4},\Su{1,2,3,4};\Su{234},\Su{12,23,4};\Su{1},\Su{1,2},\Su{1,2,3})$\\
\hline
$\mathcal S_5$ & $(K^5;\Su{125,13},\Su{1,2,3,5};\Su{5},\Su{1,24,5};\Su{4},\Su{3,4},\Su{2,3,4},\Su{1,2,3,4,5})$\\
\hline
$\mathcal S_6$ & $(K^5;\Su{1,25},\Su{1,3,25};\Su{5},\Su{1,2,3,4,5};\Su{4},\Su{3,4},\Su{2,3,4},\Su{1,2,3,4,5})$\\
\hline
\end{longtable}
We used the following notation: $K^n$ denotes the  vector space over $\F$ with the canonical basis $e_1,\dots, e_n$ and 
$K_{i_1\dots i_k,\dots,j_1\dots j_m}$ denotes the subspace of $K^n$ generated by the vectors 
$e_{i_1\dots i_k}, \dots, e_{j_1\dots j_m}$ where
$$
	e_{i_1\dots i_k}=e_{i_1}+\dots+e_{i_k}, \dots, e_{j_1\dots j_m}=e_{j_1}+\dots+e_{j_m}.
$$

%\end{table}
\subsection{Proof of the implication $(a)\Rightarrow (b)$} 

\begin{prop} \label{thmDScho}
Suppose that $\Po$ has a finite representation type, and  $\V$ is a Schurian representation of $\V$. Then
	\begin{equation} \label{eqDStab}
		\Df{\Po}(\cdn \W,\cdn \V)-\Df{\Po}(\cdn \V,\cdn \W)>0,
	\end{equation}	
for any proper subrepresentation $\W$ of $\V$.
\end{prop}

\begin{proof} 
It is clear that it is enough to check the statement  in case $\Po$ is exact. 
If $\Po$ is primitive the claim  follows immediately from Corollary \ref{corEqualityofForms} and Schofield's characterization of Schurian roots for acyclic quivers \cite[Theorem 6.1]{sch}. 
Indeed, in this case any representation of $\Po$ corresponds to a representation of an unbound Hasse quiver $\Q{\Poe}$ of $\Poe$. Also, any representation of $\Po$ is coordinate. Hence, using  Corollary \ref{corEqualityofForms} and the fact that in this case 
$\Df{\Po}(\cdn \W,\cdn \V)=\Bf{\Po}(\dimv \W,\dimv \V)$ coincides with the usual Tits form of $\Q{\Poe}$, we apply \cite[Theorem 5]{sch} to prove that \eqref{eqDStab} holds.

Now assume that $\Po$ is non-primitive exact. For each representation we completely describe all maximal subcoordinate dimensions (see Appendix B in ArXiv version of the manuscript for the details). The statement now follows by direct verification of conditions \eqref{eqDStab}.
\end{proof}

By Proposition \ref{thmDScho} we have that any indecomposable $\V\in \cats{\Po}$ with coordinate dimension $\alv=(\alpha_0;\alpha_s)_{s\in \Po}$ of a poset of finite type is costable with a form $\thv\in \Zn^{\Poe}$ given by
\begin{equation*}
	\thv(\bev)=\Df{\Po}(\alv,\bev)-\Df{\Po}(\bev,\alv). 
\end{equation*}
It is straightforward to check that the components of this form are 
\begin{equation}\label{formulasCanCstab}
	\theta_0=-\sum_{s\in \Po} \alpha_s, \qquad \theta_s=\sum_{s\prec t\in \Poe} \alpha_t-\sum_{t\prec s\in \Poe} \alpha_t.
\end{equation}
For instance, a unique exact representation of a poset $\mathcal S_1$ is costable with a form $(-6;4,1;5,2;4,2)$.

Now observe that if a representation is exact  then for a fixed $s\in \Po$ we have 
$$
	\sum_{s\prec t\in \Poe} \alpha_t > \alpha_0 > \sum_{t\prec s\in \Poe} \alpha_t.
$$
Therefore each $\theta_s>0, s\in \Po$ and any exact representation is positively costable. Now  the implication $(a)\Rightarrow (b)$ follows from Proposition \ref{cstLift}.

\subsection{Proof of the implication $(a)\Rightarrow (c)$}

To prove the implication $(a)\Rightarrow (c)$ we show the analogue of 
Proposition \ref{thmDScho} for so-called sincere representations and their dimension vectors.

We call a representation $\V=(V_0;V_s)_{s\in \Po}$ \textit{sincere}
if it is indecomposable, $\dimv \V$ is sincere and $V_s\neq V_t$ if $s\prec t$ in $\Poe$. Respectively, $\Poe$ is called \textit{sincere} if it
has at least one sincere representation. The following proposition describes all sincere posets of finite
type. 

\begin{prop} The set of sincere posets of finite type consists of four primitive posets $(1,1,1)$, $(1,2,2)$, $(1,2,3)$, $(1,2,4)$ and non-primitive posets $\mathcal S_1,\dots, \mathcal S_6$.
\end{prop}

\begin{proof}
Let $\Po$ be a poset of finite type,  $\V\in \cats{\Po}$  its sincere representation. Precisely one of the following cases occurs:
\begin{itemize}
\item[i)] $\V$ is exact representation of $\Po$ with $V_s\neq V_0$ for all $s\in \Po$.
\item[ii)] $\V$ is non-exact representation at some $s\in\Po$, therefore $V_s=\sum_{t\prec s}V_t$.
\end{itemize}
In the first case $\Po$ is in list \eqref{critial_posets} of exact posets. In the second case $\V$ generates an indecomposable 
representation (denoted by $\V^1$) of the reduced poset $\Po_s=\Po \setminus {s}$. Obviously, $\V^1$ is a sincere representation
of $\Po_s$ and therefore it satisfies either 1) or 2) above. 
Proceeding in this way we eventually obtain an exact 
representation of some poset $\Po_{s_1,\dots,s_k}$ with 
$V_s\neq V_0,\ s\in \Po_{s_1,\dots,s_k}$ from the table \eqref{critial_posets}.

Summing up we have the following procedure to describe all
sincere posets and their sincere representations:
\begin{enumerate}
\item All exact posets which admit exact and at the same time sincere representation $\V$ (that is, $V_i \neq V_0$)  are precisely $(1,1,1)$, $(1,2,2)$, $(1,2,3)$,
$(1,2,4)$ and $\mathcal{S}_2$;
\item Let $\Po$ be a sincere poset and $\V$ its sincere representation. 
Let $\mathcal I$ be a subset of $\Po$  such that $\sum_{s\in \mathcal I}V_s \neq V_0$. Define an extended poset  $\Po^{\mathcal I}=(\Po \cup \{\widetilde s\},\prec_{\mathcal I})$ with a partial order defined in such a way that its restriction to $\Po$ coincides with $\prec$ and $s\prec_{\mathcal I} \widetilde s$, for all $s\in \mathcal I$.
Let $V^{\mathcal I}$ be a representation of $
\Po^{\mathcal I}$ given by $V^{\mathcal I}_{s}=V_s$ for all $s\in\Po$ and
$V^{\mathcal I}_{\tilde s}=\sum_{s\in I} V_s$. Evidently, $V^{\mathcal I}$
is a sincere representation and therefore $\Po^{\mathcal I}$ is a sincere poset.
\end{enumerate}
The above procedure clearly terminates as the dimensions of $V_0$ are bounded.
Hence inductively we obtain all sincere posets and
all their sincere representations.
\end{proof}

Proceeding as in the proof of the previous proposition we obtain the following list of all 
sincere representations of sincere posets $\Po_1,\dots,\Po_6$:

\def\arraystretch{1.3}
\begin{longtable}{|c|p{10.5cm}|}
\hline
\textit{Poset} & \textit{Sincere representations} 
 \\
\hline
 $\mathcal S_1$ & $(K^3;\Su{123},\Su{1,23};\Su{1},\Su{1,2};\Su{3},\Su{2,3})$ \\
 \hline
$\mathcal S_2$ & $(K^4;\Su{123,24},\Su{13,2,4};\Su{4},\Su{1,4};\Su{3},\Su{2,3},\Su{1,2,3})$\\
& $(K^4;\Su{124,13},\Su{12,13,4};\Su{4},\Su{1,2,4};\Su{3},\Su{2,3},\Su{1,2,3})$\\
\hline
$\mathcal S_3$ & $(K^4;\Su{4},\Su{1,4},\Su{1,3,4};\Su{3},\Su{2,3},\Su{1,2,3};\Su{123,24})$\\
\hline
$\mathcal S_4$ & $(K^4;\Su{4},\Su{123,4},\Su{1,23,4};\Su{14},\Su{1,2,4};\Su{3},\Su{2,3},\Su{1,2,3})$\\
\hline
$\mathcal S_5$ & $(K^5;\Su{15,4},\Su{1,2,4,5};\Su{5},\Su{123,24,5};\Su{3},\Su{2,3},\Su{1,2,3};\Su{1,2,3,5})$\\
\hline
$\mathcal S_6$ & $(K^5;\Su{5},\Su{1,2,5};\Su{134,235},\Su{13,23,4,5};\Su{4},\Su{3,4},\Su{2,3,4};\Su{1,2,3,4})$\\
\hline
\end{longtable}

Similarly to Proposition \ref{thmDScho} one proves the following:

\begin{prop} \label{propPosit}
 Let $\Po$ be a sincere poset and $\V$ its sincere representation. Then $\V$ is 
 stable with a form 
 \begin{equation} \label{ShSinc}
 	\thv(\W)=\Bf{\Po}(\dimv \V,\dimv \W)-\Bf{\Po}(\dimv \W,\dimv \V).
\end{equation}
\end{prop}

To prove this proposition we describe the set of maximal subdimensions for all sincere representations of posets of finite type and check the stability conditions  
$\eqref{ShSinc}$. The details are given in Appendix C of ArXiv version of the manuscript.

If $\Po$ is sincere than it is straighforward to see (see also, \cite[Proposition 4.2]{simson}) that 
$$
	\Bf{\Po}(\alv,\bev)=\alv \cdot \Mi{\Poe}^{-1} \cdot \bev^{tr}=\sum_{s\in \Poe} \alpha_s\beta_s-\sum_{s\to t \in \Poe}\alpha_s\beta_t+\sum_{s,t \in \Poe}r(s,t)\alpha_s\beta_t,
$$ 
in which $r(s,t)$ is the maximal number of $\F$-linear independent minimal commutativity relations with the source $s$ and the terminus $t$. By Proposition \ref{propPosit} we have that any sincere $\V\in \cats{\Po}$ with dimension $\alv=(\alpha_0;\alpha_s)_{s\in \Po}$ of a poset of finite type is stable with a form $\thv\in \Zn^{\Poe}$ given by
\begin{equation} \label{canocinalStab}
	\thv(\bev)=\Bf{\Po}(\alv,\bev)-\Bf{\Po}(\bev,\alv). 
\end{equation}
One checks that the components of this form are:
\begin{equation}\label{formulaCanStab}
	\theta_0=-\sum_{s\to 0 \in \Poe} \alpha_s+\sum_{s\in \Po} r(s,0)\alpha_s, \qquad 
	\theta_s=\sum_{s \to t\in \Poe} \alpha_t-\sum_{t \to s\in \Poe} \alpha_t-\sum_{t\in \Poe} r(s,t)\alpha_t.
\end{equation}
For instance, a unique sincere representation of a poset $\mathcal S_1$ is stable with a form 
$(-6;2,1;1,2;2,2)$. By examining each sincere poset we check that the components $\theta_i$ are positive. Therefore, each sincere representation of a poset of finite type is positively stable with the form defined by \eqref{formulaCanStab}.

Now let $\V$ be an indecomposable representation of $\Poe$ of finite type. Hence, there is a sincere subposet $\widetilde {\mathcal I}$ of $\Poe$ such that the restriction $\V_{\widetilde {\mathcal I}}$ of $\V$ to $\tilde I$ is a sincere representation. The representation $\V_{\widetilde {\mathcal I}}$ is positively stable by considerations above. Then  the representation $\V$ is positively stable by Proposition \ref{stLift}. The implication $(a)\Rightarrow (c)$ follows.

\section{Geometric stability} \label{secGeometric}

In this section we assume that $\F$ is algebraically closed. Fix the poset $\Po$ and the admissible dimension vector $\alv=(\alpha_0;\alpha_s)_{s\in \Po}$. As we mentioned above the variety $\R{\alv}{\Po}$ is projective and the group $\GL(\alv_0)$ acts on $\R{\alv}{\Po}$ diagonally. Our goal is to understand the quotient space $\R{\alv}{\Po}/\GL(\alv_0)$. As usual the main problem is that the quotient space is rarely a projective variety. One possible approach is to construct a ``good'' quotient is via Geometric Invariant Theory (GIT). 
We briefly recall this approach, for details we refer to \cite{Dolgachev} (for general approach), to \cite{King} (where the author constructed the good quotients for representations of quivers) and to \cite{reineke} (where the author motivated the geometric approach to the  classification problem of quiver representations and  discussed topological, arithmetic and algebraic methods for the study of moduli spaces).

\subsection{Brief review of GIT quotients}\label{GITreview}

Let $G$ be a reductive group acting on a projective algebraic variety $X$. The GIT approach  consists of the following steps. First one chooses a linearization of the action, that is, a $G$-equivariant embedding of $X$ into a projective space $\mathbb P^n$ with a linear action of $G$ (via representation of $G$ in $\GL(n+1)$). An embedding of $X$ to $\mathbb P^n$ is defined by choosing a line bundle $L$ over $X$ (which is ample iff the emdedding is closed) and the set of its sections $f_0,\dots,f_n$ (which form a basis in the space of sections $\Gamma(X,L)$).  Then one specifies (with respect to $L$) the sets $X^{ss}(L)$, $X^{s}(L)$, $X^{us}(L)$ of semi-stable, stable and regular points respectively on $X$, where
\begin{itemize}
\item[(i)] $x\in X$ is called \textit{semi-stable} if there exist $m>0$ and $f\in \Gamma(X,L^m)^G$ such that $X_f=\{y\in X \ | \ f(y)\neq 0\}$ is affine and contains $x$;
\item[(ii)] $x\in X$ is called \textit{stable} if it is semi-stable, stabilizer of $x$ is finite and $G$-action on $X_f$ is closed; 
\item[(iii)] $x\in X$ is called \textit{unstable} if it is not semi-stable.
\end{itemize}
The central point GIT is that there exists an algebraic quotient of $X$ by $G$, denoted by $X//G$, which can be described as the quotient of the open set of $X^{ss}(L)$ of semistable points by the equivalence relation: $x\sim y$ if and only if the orbit closures $\overline{G\cdot x}$ and $\overline{G\cdot y}$ intersects (in $X^{ss}(L)$). Therefore the points of $X//G$ are in one-one correspondence with the closed orbits in $X^{ss}(L)$. Note that in case $L$ is ample then  (see \cite[Proposition 8.1]{Dolgachev})
\begin{equation} \label{projModuli}
	X^{ss}(L)//G \cong \textrm{Proj} \Big ( \oplus_{n\geq 0} \Gamma(X,L^{\otimes n})^G \Big ),
\end{equation}
and $X^{ss}(L)//G$ is a projective variety. The variety $X^{s}(L)/G$ is a \textit{geometric quotient}, which parametrizes the stable orbits. 

A powerful tool to describe stable points is the Hilbert-Mumford numerical criterion of stability, which is stated in terms of the action to one-parameter subgroups of $G$. Let $x^{*} \in \F^{n+1}$ be a representative of $x\in X\subset \mathbb P^{n}$ and $\lambda: \F^{*}\to G$ (regular morphism) be a one-parameter subgroup of $G$. Then (in appropriate coordinates) it acts by:
$$
	\lambda(t)\cdot x^{*}=(t^{m_0}x_0,\dots,t^{m_n}x_n).
$$
Set 
\begin{align*} 
\upmu^L(x,\lambda)=\min_{t}\{m_i:x_i\neq 0\}.
\end{align*} The Hilbert-Mumform numerical criterion claims (see \cite[Theorem 9.1]{Dolgachev} for  details) that 
\begin{equation} \label{HMcriterion}
\begin{split}
	x\in X^{ss}(L) & \Leftrightarrow \upmu^L(x,\lambda) \leq 0,	\\
	x\in X^{s}(L) & \Leftrightarrow \upmu^L(x,\lambda) < 0,
\end{split}
\end{equation}
for all one-parameter subgroup of $G$.

\subsection{Linearization of $\SL(\alv_0)$-action.} \label{secLinearization}
First note that the orbits of $\GL(\alpha_0)$-action on $\R{\alv}{\Po}$ are in one-one correspondence with the orbits of $\SL(\alpha_0)$, so we study the action of $\SL(\alpha_0)$.
Fix a form $\thv=(\theta_s)_{s\in \Po}\in \mathbb Z^{\Po}$
with $\theta_s \geq 0$ for all $s\in \Po$.
  As shown in Proposition \ref{posvar} the variety $\R{\alv}{\Po}$ is closed in the product of Grassmanians $\prod_{s\in \Po} \Gr(\alpha_s,\alpha_0)$. Our first aim is to embed the variety $\R{\alv}{\Po}$ into some larger projective space corresponding to linearizing action of $\SL(\alpha_0)$. We use a slightly modified standard construction (see, for example, \cite[Chapter 11]{Dolgachev} and \cite{Knu}).

A standard way to embed $\Gr(\alpha_s,\alpha_0)$ into a projective space 
is via Plucker embedding, that is, for an element
$V_i \in \textrm{Gr}(\alpha_s,\alpha_0)$ we take its basis vectors $a_j$ and
wedge them together $a_1 \wedge\cdots\wedge a_{\alpha_s}$ obtaining an element of
$\mathbb P(\wedge^{\alpha_s}\F^{\alpha_0})$. Then using the Veronese map we embed the projective space $\mathbb P (V)$ into the space $\mathbb P({\textrm{Sym}^d(V)})$. Respectively, for the product of Grassmanians
  $\prod_{s \in \Po} \textrm{Gr}(\alpha_s,\alpha_0)$ we have the embedding
  \begin{eqnarray*}
        \prod_{s \in \Po} \Gr(\alpha_s,\alpha_0)
        \hookrightarrow \prod_{s \in \Po}
        \mathbb P(\textrm{Sym}^{\theta_s}(\wedge^{\alpha_s}\F^{\alpha_0})).
  \end{eqnarray*}
  Using the Segre map $\mathbb P^n\times \mathbb P^m \hookrightarrow \mathbb
  P^{(n+1)(m+1)-1}$ we embed the last product into
  \begin{eqnarray*}
     \mathbb P\left( \prod_{s \in \mathcal P} \textrm{Sym}^{\theta_s}(\wedge^{d_s}\F^{d_0})\right).
 \end{eqnarray*}
Hence, we have the following sequence of inclusions:
\begin{eqnarray*}
        \Gr(\alpha_s,\alpha_0) \hookrightarrow \mathbb P(\wedge^{\alpha_s}\F
        ^{\alpha_0}) \hookrightarrow \mathbb P ({\textrm{Sym}^{\theta_s}(\wedge^{\alpha_s}\F^{\alpha_0})}).
\end{eqnarray*}
And, therefore we get the following closed embedding of $\R{\alv}{\Po}$:
\begin{eqnarray*}
        \R{\alv}{\Po} \hookrightarrow \mathbb P(\wedge^{\alpha_s}\F
        ^{\alpha_0}) \hookrightarrow \mathbb P ({\textrm{Sym}^{\theta_s}(\wedge^{\alpha_s}\F^{\alpha_0})}).
\end{eqnarray*}
As embedding above is closed, the corresponding line bundle $L_\thv$ is ample. Note that $L_\thv$ has exactly one $\SL(\alpha_0)$ linearization, since the center of $\SL(\alpha_0)$ is $0$-dimensional. Our aim is to describe the set of semistable $\R{\alv}{\Po}^{\thv-ss}$ and stable $\R{\alv}{\Po}^{\thv-s}$ points with respect to $L_\thv$ (we adopt the arguments from \cite[Theorem 11.1]{Dolgachev}, \cite[Theorem 2.2]{Hu} and \cite{mfk}).

\begin{thm}
	Let $\thv=(\theta_s)_{s\in \Po}\in \Zn^{\Po}_{+}$. Then $\V=(V_0;V_s)_{s\in \Po} \in \R{\alv}{\Po}^{\thv-ss}$ (resp. $\in \R{\alv}{\Po}^{\thv-s}$) if and only if for any proper subrepresentation $\W\subset \V$ we have $\mu_{\thv}(\W)\leq \mu_{\thv}(\V)$ (resp. the strict inequality holds);
	that is, if and only if $\V$ is $\mu_\thv$-semistable (resp. $\mu_\thv$-stable).
\end{thm}

\begin{proof}
Let $n=\alpha_0=\dim V_0$ and $T$ be the maximal torus in $\SL(n)$. Each one-parameter subgroup $\lambda: \F^*\rightarrow T$ is conjugated to a diagonal one. Therefore, we assume that
$$
	\lambda(t)=\diag\{t^{q_1},\dots,t^{q_n}\},
$$
where $q_1+\dots+q_n=0$. Without loss the generality we can assume that $q_1\geq\dots\geq q_n$. Also, it is a standard fact that all such groups form a convex set with extreme points $\lambda_r: \F^*\rightarrow T$ given by 
$$
	\lambda_r(t)=\diag\{t^{q_1},\dots,t^{q_n}\},
$$ 
such that $q_1=\dots=q_r=n-r$, $q_{r+1}=\dots=q_{n}=-r.$

Suppose that $\V=(V_0;V_s)_{s\in \Po}$ is a semistable point.  Choose a basis $v_1,\dots,v_n$ of $V_0$. 
Set $H_i=\spa\{v_1,\dots,v_i\}$, $i=1,\dots,n$ (in particular, we have $H_n=V_0$ and $H_r=W$). Let $K$ be any subspace of $V_0$. Then for any integer $j$, $1\leq j \leq s=\dim K$, there is a unique integer $m_j$ such that
$$
	\dim(K \cap H_{m_j})=j, \quad \dim(K \cap H_{m_j-1})=j-1.
$$
Therefore we can represent $K$ (in the basis $e_1,\dots,e_n$) by the matrix $A_K$ of the form
$$
	A_K=\left[
	\begin{matrix}
		a_{11} & \dots & a_{1m_1} & 0 & \dots & 0 & 0 & \dots & 0\\
		a_{21} & \dots & \dots & a_{2m_2} & \dots & 0 & 0 & \dots & 0\\
		\vdots & \vdots & \vdots & \vdots & \vdots & \vdots & \vdots & \vdots & \vdots\\
		a_{k1} & \dots & \dots & \dots & \dots & \dots & a_{km_k} & \dots & 0\\
	\end{matrix}
	\right]^{T},
$$
with $a_{jmj}\neq 0$ for all $j$. Considering the maximal minors of $A_K$ we have that in the Plucker embedding $p_{i_1\dots i_k}(K)=0$ if $i_j>m_j$ and $p_{m_0\dots m_k}(K)\neq 0$. Also from the matrix representation of $K$ we get 
$$
	p_{i_1\dots i_k}(\lambda(t) K)=t^{q_{i_1}+\dots+q_{i_k}} p_{i_1\dots i_k}(K).
$$
Applying the procedure above to all subspaces $V_s$ in $\V$ we get the numbers 
$m_1^{(s)},\dots,m_{\alpha_s}^{(s)}$ for all $s\in \Po$ and we have (by thw minimality of the numerical function) that
$$
	\upmu^{L_\thv}(\V,\lambda)=\sum_{s \in \Po} \theta_s 
	\sum_{i=1}^{\alpha_i} q_{m_{j}^{(s)}}. 
$$
Now, since $\dim(V_s \cap H_j)-\dim(V_s \cap H_{j-1})=0$ if $j\neq m_j^{(s)}$, we rewrite the previous sum as follows:
\begin{equation*}
\begin{split}
	\upmu^{L_\thv}(\V,\lambda)&=\sum_{s \in \Po} \theta_s 
	\sum_{i=1}^{n} q_i \big(\dim(V_s \cap H_i)-\dim(V_s \cap H_{i-1})\big)\\
	&=\sum_{s \in \Po} \theta_s \big(\alpha_s q_n+\sum_{i=1}^{n-1}(\dim(V_s \cap H_j)(q_i-q_{i+1})\big)  \\
	&=q_n \sum_{s\in \Po} \theta_s\alpha_s +\sum_{j=1}^{n-1}\Big(\sum_{s\in \Po} \theta_s \dim (V_s\cap H_j) (q_j-q_{j_1}) \Big ).
\end{split}
\end{equation*}
Note that $\upmu^{L_\thv}(\V,\lambda)$ is linear in $(q_1,\dots,q_n)$. Therefore replacing t $\lambda$ by a subgroup $\lambda_s$ we get
$$
	\upmu^{L_\thv}(\V,\lambda_r)=-r\sum_{s\in \Po} \theta_s\alpha_s + n \sum_{s\in \Po} \theta_s\dim (V_s\cap H_r).
$$
By the Hilbert-Mumford numerical criteria \eqref{HMcriterion} we have that if $\V$ is semistable (resp. stable) then  $\upmu^{L_\thv}(\V,\lambda_r)\leq 0$ (resp. $\upmu^{L_\thv}(\V,\lambda_r)< 0$), which is the same as 
$$
	\mu_{\thv}(\W)\leq \mu_{\thv}(\V), \quad \mbox{(resp } <),
$$
 where $\W=(H_r;V_s\cap H_r)_{s\in \Po}$ is a proper subrepresentation 
of $\V$. Hence $\V$ is $\mu_\thv$-semistable (resp. stable). 

Conversely, let $\V$ is $\thv$-semistable but not semistable with respest to $L_\thv$. Then there exist a one-parameter subgroup $\lambda$ such that $\upmu^{L_\thv}(\V,\lambda)>0$. Hence, there must exist $1\leq r\leq n-1$ such that $\upmu^{L_\thv}(\V,\lambda_r)>0$, which is equivalent to  $$\mu_{\thv}((H;V_s\cap H)_{s\in \Po})>\mu_{\thv}(\V)$$ for some $r$-dimensional subspace $H$ of $V_0$. Therefore $\V$ is not $\thv$-semistable. Contradiction.  
Similarly one proves the sufficiency of conditions for the strict inequality.
\end{proof}

\begin{cor}
	If the dimension vector $\alv$ satisfies $\thv(\bev)\neq 0$ for all $0\neq\bev < \alv$, then
	\begin{equation} \label{eqSSandST}
	 \R{\alv}{\Po}^{\thv-ss}=\R{\alv}{\Po}^{\thv-s}.
	\end{equation}
\end{cor}
\begin{proof}
	Indeed, in this case each semistable representation is already stable.
\end{proof}
Note that  if $\alv$ is coprime (that is, $\textrm{gcd}(\alpha_s:s\in \Poe)=1$) then the equality \eqref{eqSSandST} holds for the generic choice of $\thv$.

\subsection{Polystable representations}

We start by the following proposition (see also \cite[Proposition 3.1]{Hu}).

\begin{prop}\label{directpolytable}
Let $\thv=(\theta_s)_{s\in \Po}$. Assume that $\V\in \cats{\Po}$ is $\mu_\thv$-semistable and $\V=\oplus_{i=1}^{l} \W_i$ is a direct sum of subrepresentations. Then $\mu_\thv(\W_i)=\mu_\thv(\V)$ and $\W_i$ are $\mu_\thv$-semistable.
\end{prop}

\begin{proof}
	Assume that $\V=\W_1\oplus \W_2$. Then
	$$
		\ses{\W_1}{\V}{\W_2},		
	$$
	and 
	$$
		\ses{\W_2}{\V}{\W_1}.		
	$$
	As $\V$ is semistable then $\mu_\thv(\W_1)\leq \mu_\thv(\V)$ and $\mu_\thv(\W_2)\leq \mu_\thv(\V)$. By Proposition \ref{seesaw} we have that $\mu_\thv(\W_1)\geq \mu_\thv(\V)$ and $\mu_\thv(\W_2)\geq \mu_\thv(\V)$.  The statement follows.
\end{proof}

A $\mu_\thv$-semistable representation $\V$ will be called $\mu_\thv$-\textit{polystable} if it decomposes into a direct sum of finitely many $\mu_\thv$-stable subrepresentations.  Similarly to \cite[Proposition 3.3]{Hu} one proves the following:

\begin{prop} \label{polystable-closed}
	$\V$ is $\mu_\thv$-polystable if and only if the orbit of $\V$ in $\R{\alv}{\Po}^{\thv-ss}$ is closed.
\end{prop}

As a conseguence of this proposition we have that $\R{\alv}{\Po}^{\thv-ss}//\SL(\alpha_0)$ parametrizes $\mu_\thv$-polystable representations. Denote by $\cats{\Po}^{\thv-ps}$ the additive subcategory of $\cats{\Po}^{\thv-ps}$ consisting of $\mu_\thv$-polystable representations. Then $\cats{\Po}^{\thv-ps}$ is semisimple, where $\mu_\thv$-stable representations are precisely the simple objects.

\subsection{Moduli space of representations of posets} 
Let $\thv=(\theta_s)_{s\in \Po}\in \Zn^{\Po}$ and fix an admissible dimension vector $\alv\in \Zn^{\Poe}$. We will make the following identification: 
$$
	\M{\Po}{\alv}^{\thv-ss}=\R{\alv}{\Po}^{\thv-ss}//\SL(\alpha_0), \qquad \M{\Po}{\alv}^{\thv-s}=\R{\alv}{\Po}^{\thv-s}/\SL(\alpha_0).
$$

\begin{cor}
By \eqref{projModuli}, the variety $\M{\Po}{\alv}^{\thv-ss}$ is projective and by Proposition \ref{polystable-closed} it parametrizes the isomorphisms classes of $\mu_\thv$-polystable representations of $\Po$ of dimension vector $\alv$.
\end{cor}  

\begin{cor}
	The variety $\M{\Po}{\alv}^{\thv-s}$ is open in $\M{\Po}{\alv}^{\thv-ss}$ and parametrizes the isomorphisms classes of $\mu_\thv$-stable representations of $\Po$ of dimension vector $\alv$.
\end{cor}  

By Theorem \ref{dimVariety} we have that 
$\dim \R{\alv}{\Po}=\alpha_0^2-\Bf{\Po}(\alv,\alv)$, therefore if $\M{\Po}{\alv}^{\thv-s}$ is non-empty we have that 
\begin{equation} \label{dimFormula}
\begin{split}
	\dim \M{\Po}{\alv}^{\thv-s}&=\dim \R{\alv}{\Po} - \dim \SL(\alpha_0)\\
								  &=\dim \R{\alv}{\Po}-\alpha_0^2+1 \\
&=1-\Bf{\Po}(\alv,\alv).                           
\end{split}
\end{equation}
Note that this dimension formula is a direct analogue of the dimension formula for moduli space of $\mu_\thv$-stable representations of quiver $Q$ which given in terms of the quadratic form associated with $Q$ (see, for example, \cite[Section 3.5]{reineke}).

\subsection{Moduli spaces and Coxeter functors}

Assume that $\Po$ is primitive and $\alv\neq (1;0,\dots,0)_{s\in \Po}$. Due to Theorem \ref{CoxThm}  Coxeter transformation $\Fp$ (defined in Section \ref{secCoxeter}) 
gives rise to a map between moduli spaces:
$$
\Fp:\M{\Po}{\alv}^{\thv-ss}\to \M{\Po}{\Fp(\alv)}^{\Fp(\thv)-ss}.
$$
Applying $(\Fp)^n$, in certain cases (for instance when $\Po$ is of finite type, or when $\alv$ is preprojective) we are able to obtain the information about $\M{\Po}{\alv}^{\thv-ss}$ knowing it in simpler cases (e.g., one-dimensional cases). We believe that a more careful study of these maps deserves further attention.

\subsection{Examples.} 

\begin{ex} Assume that $\Po$ is a poset of finite type. By Theorem \ref{TheoremSection3} we have that if $\alv$ is an admissible indecomposable dimension then both sets 
$\R{\alv}{\Po}^{\thv-ss}$ and $\R{\alv}{\Po}^{\thv-s}$ are non-empty. Therefore 
$\M{\Po}{\alv}^{\thv-ss}$ and $\M{\Po}{\alv}^{\thv-s}$ are non-empty as well. As $\mathcal S$ is of finite type, the orbit of indecomposable $\V$ with dimension $\alv$ is dense in $\R{\alv}{\Po}$ therefore $\M{\Po}{\alv}^{\thv-s}$ consists of one point. 
\end{ex}

\begin{ex}
Now assume that the poset $\Po$ is one of the critical poset from list \eqref{critial_posets_}. Consider dimension vector $\alv_\Po$ which a minimal imaginary root of form $\Bf{\Po}$ (that is, minimal $\alv_\Po$ so that $\Bf{\Po}(\alv_\Po,\alv_\Po)=0$):
\begin{equation*}
\begin{split}
	\alv_{(1,1,1,1)}&=(2;1,1,1,1);\\
	\alv_{(2,2,2)}&=(3;1,2,1,2,1,2);\\	
	\alv_{(1,3,3)}&=(4;2,1,2,3,1,2,3);\\	
	\alv_{(1,2,5)}&=(6;3,2,4,1,2,3,4,5);\\	
	\alv_{(N,4)}&=(5;2,4,1,3,1,2,3,4).\\	
\end{split}
\end{equation*}
For instance, for unique non-primitive critical poset $(N,4)$ we have the following Hasse diagram (where the components of dimension vectors we place at corresponding vertices of the diagram)
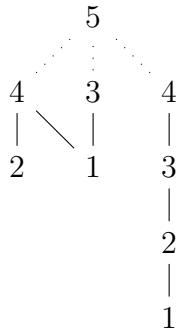
\begin{center}
\begin{tikzpicture}[scale=1,auto,swap,shift=(current page.center)]
  \node (1) at (1,2) {$2$};
  \node (3) at (2,2) {$1$};
  \node (2) at (1,3) {$4$};
  \node (4) at (2,3) {$3$};
  \node (5) at (3,0) {$1$};
  \node (6) at (3,1) {$2$};
  \node (7) at (3,2) {$3$};
  \node (8) at (3,3) {$4$};
  \node (9) at (2,4) {$5$};
  \draw[-] (1) -- (2); \draw[-] (3) -- (2);
  \draw[-] (3) -- (4); \draw[-] (5) -- (6) -- (7) -- (8);
  \draw[loosely dotted] (2) -- (9);
  \draw[loosely dotted] (4) -- (9);
  \draw[loosely dotted] (8) -- (9);
\end{tikzpicture}
\end{center}
It is straightforward to see that both $\R{\alv}{\Po}^{\thv-ss}$ and $\R{\alv}{\Po}^{\thv-s}$ are non-empty for the choice of $\thv$ given by  formulas \eqref{canocinalStab}. Therefore the moduli spaces
$\M{\Po}{\alv}^{\thv-ss}$ and $\M{\Po}{\alv}^{\thv-s}$ are non-empty as well.  
By \eqref{dimFormula} we have $\dim \M{\Po}{\alv}^{\thv-s}=1$, as $\Bf{\Po}(\alv,\alv)=0$ in these cases. 
\end{ex}

\begin{ex}
Consider the poset $\Po$, such that the Hasse diagramm of $\Poe$ (with the components of admissible dimension vector $\alv$  placed in corresponding vertices) is given by: 
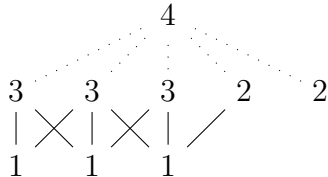
\begin{center}
\begin{tikzpicture}[scale=1,auto,swap,shift=(current page.center)]
  \node (1) at (0,2) {$1$};
  \node (2) at (1,2) {$1$};
  \node (3) at (2,2) {$1$};
  \node (5) at (0,3) {$3$};
  \node (6) at (1,3) {$3$};
  \node (7) at (2,3) {$3$};
  \node (8) at (3,3) {$2$};
  \node (9) at (4,3) {$2$};
  \node (10) at (2,4) {$4$};
  \draw[-] (1) -- (5); \draw[-] (1) -- (6);
  \draw[-] (2) -- (6); \draw[-] (2) -- (5);
  \draw[-] (2) -- (7); \draw[-] (3) -- (6);
  \draw[-] (3) -- (7); \draw[-] (3) -- (8); 
  \draw[loosely dotted] (5) -- (10);
  \draw[loosely dotted] (6) -- (10);
  \draw[loosely dotted] (7) -- (10);
  \draw[loosely dotted] (8) -- (10);
  \draw[loosely dotted] (9) -- (10);
\end{tikzpicture}
\end{center}

\iffalse
$$ \xymatrix @C=0.2cm @R=0.8cm{&&&&4&&\\
3\ar@{-}[urrrr]&&3\ar@{-}[urr]&&3\ar@{-}[u]&&2\ar@{-}[ull]&&2\ar@{-}[ullll]\\
1\ar@{-}[urr]\ar@{-}[u]&&1\ar@{-}[ull]\ar@{-}[urr]\ar@{-}[u]&&1\ar@{-}[ull]\ar@{-}[urr]\ar@{-}[u]&&}
$$
\fi
Again one easily construct the stable representation in $\R{\alv}{\Po}^{\thv-s}$ with respect to the choice of $\thv$ given by  formulas \eqref{canocinalStab}. For instance, the following representation is $\thv$-stable (we use the same notation as in Section 4)
$$
(K^4;K_1,K_2,K_3,K_{1,2,34},K_{1,2,3},K_{2,3,14},K_{3,14},K_{12,34}).
$$
Therefore the moduli spaces $\M{\Po}{\alv}^{\thv-s}$ is non-empty and by dimension formula \eqref{dimFormula} we have $\dim \M{\Po}{\alv}^{\thv-s}=2$, as $\Bf{\Po}(\alv,\alv)=-1$.
\end{ex}

\section{Moment map and unitary representation of posets} \label{secMoment}

\subsection{Unitary representation of posets}

We assume that $\F=\Cn$. 
By a \textit{unitary representation} of $\Po$ we mean a subspace representation $\U=(U_0;U_s)_{s\in \Po}$ in which the ambient $U_0$ is a unitary space. Two unitary representations $\U=(U_0;U_s)$ and $\U'=(U'_0;U'_s)$ of $\Po$ are unitarily equivalent if there exists a unitary bijection $\varphi: U_0\to U'_0$ such that $\varphi(U_s)=U'_s$ for all $s\in \Po$.  Result of \cite{Halmos} gives a complete classification of indecomposable systems of two unitary subspaces (which is already a non-finite problem). In \cite{BFKSY} the authors classified the posets which have finite, tame and wild unitary type. Note that the problem of classifying of unitary representations is wild even for the poset $\Po=\{s_1,s_2,s_3 \ | \ s_1\prec s_2 \}$. It turned out that the classification becomes possible for a broader class of posets if one imposes additional conditions on unitary representations (cf. \cite{KruglyakNazarovaRoiter, KruglyakRoiter, SamYus}).

We say that a unitary representation 
$\U=(U_0; U_s)_{s\in \Po}$ is a representation of weight $\chv=(\chi_s)_{s\in \Po} \in  \mathbb Z^{\Poe}_{+}$ (or $\chv$-\textit{representation}) if
\begin{equation}\label{chirelation}
 \sum_{s\in \Po} \chi_s P_{U_s}= \chi_0 I,
\end{equation}
where $P_M$ denote the orthogonal projection of $U_0$ onto subspace $M$, and $\chi_0\in \mathbb Q$ is determined by the trace identity of \eqref{chirelation}. 
All $\chv$-representations of $\mathcal P$ form an additive category denoted by 
$\ucats{\Po}{\chv}$.

There is an obvious (forgetfull) functor $\Fu:\ucats{\Po}{\chv}\to \cats{\Po}$ which relates to 
$\chv$-representation $\U=(U_0;U_s)_{s\in \Po}$ the underlying system of vector spaces (forgetting the inner product). We prove the following (see also \cite[Lemma 5]{SamYus})

\begin{prop} \label{functorF}
Let $\U=(U_0;U_s)_{s\in \Po}\in \ucats{\Po}{\chv}$ be $\chv$-representation. Then $\Fu(\U)$ is $\mu_\thv$-polystable with $\thv=(\chi_s)_{s\in \Po}$.
\end{prop}

\begin{proof}
First suppose that $\U$ is indecomposable. Equating the traces of both sides in \eqref{chirelation}, we get $\sum_{s \in \mathcal S} \chi_s \dim U_s=\chi_0 \dim U.$
If $M$ is any proper subspace of $U$ then $
    \sum_{s \in \mathcal S}\chi_s P_{U_s}P_M=\chi_0 P_M$. Therefore $\chi_0=\mu_\thv(\U)$.
Equating the traces of both sides in the last equality we get
\[
    \sum_{s \in \mathcal S}\chi_s {\rm{tr}}(P_{U_s}P_M)=\chi_0 \dim M.
\]
It follows from \cite{Halmos} that
${\rm{tr}}(P_{M_1\cap M_2}) \leq {\rm{tr}}(P_{M_1}P_{M_2})$  for each two subspaces $M_1$ and $M_2$, and so
\[
  \sum_{s \in \mathcal S}\chi_s {\rm{tr}}(P_{U_s\cap M})\leq\sum_{s \in \mathcal S}\chi_s {\rm{tr}}(P_{U_s}P_M)=\chi_0 \dim M.
\]
It remains to prove that the last inequality is strict. Indeed, assuming
that  ${\rm{tr}}(P_{U_s\cap M})={\rm{tr}}(P_{U_s}P_M)$ for all $s$, we obtain that each $P_{U_s}$ commutes with $P_M$. Hence, the subspace
$M$ is invariant with respect to the projections $P_{U_i}$ and the representation $\U$ is decomposable. This contradicts the assumption.	
Therefore, $\mu_\thv(\W)<\mu_\thv(\U)$ for any proper subrepresentation $\W$, and $\U$ is $\mu_\thv$-stable.

Now, if $\U$ is a decomposable $\chv$-representation, we get that $\U$ is $\mu_\thv$-semistable. Then, proceeding as in Proposition \ref{directpolytable} one proves that $\U$ is $\mu_\thv$-polystable.
\end{proof}

Having a subspace representation $\V=(V_0;V_s)_{s\in \Po}$ we say that it is \textit{$\chv$-unita\-ri\-zable} if there is an inner product on $V_0$ such that $\V$ is a $\chv$-representation with respect to this product.

 It was shown in  \cite{SamYus} that $\ucats{\Po}{\chv}$ has a finite number of unitarily non equivalent indecomposable representations for each weight $\chv$ if and only if $\Po$ has a finite type; that is, if and only if $\Po$ contains one of the Kleiner's critical posets. There are other similarities between $\chv$-representations and usual representations of poset (see, for example, \cite{WY13} and the references therein). In this section we explain these similarities via the Kempf-Ness theorem (which establishes the homeomorphism between GIT and symplectic quotients) and by constructing the functorial connection between the categories $\ucats{\Po}{\chv}$ and $\cats{\Po}$.

\subsection{Moment map, symplectic reduction and the Kempf-Ness theorem}
We briefly recall the idea behind the symplectic quotients and the Kempf-Ness theorem.
Suppose again that $G$ is a complex reductive group acting linearly on a smooth complex projective variety $X\subset \mathbb P^n$. Apart from taking GIT quotient (as in Section  \ref{GITreview}) one can alternatively consider the so-called symplectic quotient. 
As $G$ is a complex reductive group, it is equal to the complexification of its maximal compact subgroup $K$ (by $\mathfrak k$ we denote the corresponding Lie algebra of $K$). Complex projective space $\mathbb P^n$ has a natural K\"{a}hler structure given by the Fubini-Study form, therefore $X$ is symplectic with symplectic form $\omega$. Assuming that $K$ acts unitarily, there is a moment map for this action
$$
	\Phi: X \to \mathfrak k^*,
$$  
which satisfies:
\begin{enumerate}
    \item $\Phi$ is $K$-equivariant with respect to the action of
    $K$ on $X$ and to the coaction of $K$ on $\mathfrak k^*$; that is, the
    following holds
    \begin{eqnarray*}\Phi(g\cdot p)=g\Phi(p)g^{-1},\quad p\in M,\ g\in K;\end{eqnarray*}
    \item
    $\Phi$ lifts the infinitesimal action, in the sense that, for all $A\in \mathfrak k^{*}$ we have 
    $$
    	d\Phi^A=\omega(A_X,--)
    $$
    where $\Phi^A: X \to \mathbb R$ is the map given by $x\mapsto \Phi(x)\cdot A$, and  the
    infinitesimal action $\mathfrak k \to Vect(X)$ is given by $A\to A_X$ with
    $$
    		A_{X,x}=\dfrac{d}{dt}\exp(tA)\cdot x \big |_{t=0}.
    $$ 
\end{enumerate}

\begin{thm}[Kempf-Ness theorem, \cite{KempfNess}]
There is an inclusion $\Phi^{-1}(0)\subset X^{ss}(L)$ which induces a homeomorphism between the symplectic reduction and the GIT quotient
$$
	\Phi^{-1}(0)/K\cong X^{ss}(L)//G.
$$
\end{thm}

\subsection{$\chv$-unitarizable representations via moment map}
Let $\thv=(\theta_s)_{s\in \Po}\in \mathbb Z^{\Po}$ be a weight with positive components,
$\alv=(\alpha_0;\alpha_s)_{s\in \Po}\in \mathbb Z^{\Poe}$  an admissible dimension, and 
$\V=(V_0;V_s)\in\cats{\Po}$  a representation with $\dimv\V=\alv$. We regard $\V$ as a point in $\R{\alv}{\Po}(L_{\thv})$ after the embedding of $\R{\alv}{\Po}$ into the projective space as in Section \ref{secLinearization}. It is easy to check that the moment map of
$\SL(\alpha_0)$-action on $\R{\alv}{\Po}(L)$  has a form
\begin{equation*}
\begin{split}
	\Phi: \R{\alv}{\Po}(L)&\to \mathfrak {su}(\alpha_0)^{*},\\
		  \V=(V_0;V_s)_{s\in \Po}&\mapsto \sum_{s\in \Po} \theta_s A_s A_s^{*}-\mu_\thv(\V) I,
\end{split}	
\end{equation*}
 where $A_i$ is an isometry which embeds $V_i$ into $\mathbb C^{\alpha_0}$. 
Considering $\Phi^{-1}(0)$ we get
\begin{eqnarray*}
  \Phi^{-1}(0)= \left \{ (P_s)_{s \in \mathcal P}\in (M_{\alpha_0}(\mathbb C))_{s \in \Po}\,  \scalebox{1.5}{\Big |}\!
    \begin{array}{c}
        P_s=P_s^*=P_s^2,\ \textrm{rank}(P_s)=\alpha_s, \\
        P_sP_t=P_sP_t=P_s, \quad s\prec t, \\
        \sum_{s\in \Po}\theta_s P_s=\mu_\thv(\alv) I   
     \end{array}  \right\},
\end{eqnarray*}
therefore $\Phi^{-1}(0)$ is a set of objects $\U$ in $\ucats{\Po}{\thv}$ with dimension $\alv$. 
If $\U$ is $\chv$-representation then $\Fu(\U)$ is $\mu_\chv$-polystable by Proposition \ref{functorF}. Therefore, the functor $\Fu(\cdot)$ yields a natural map $\phi:\Phi^{-1}(0)\rightarrow  \R{\alv}{\Po}^{\thv-ss}$. As a consequence of the Kempf-Ness theorem we have

\begin{thm}
Let $\Po$ be a poset,  $\alv=(\alpha_0;\alpha_s)_{s \in \Po}$  an admissible dimension vector and $\thv=(\theta_0;\theta_s)_{s\in \Po}$  a form such that $\thv(\alv)=0$. The map $\phi:\Phi^{-1}(0)\rightarrow  \R{\alv}{\Po}^{\thv-ss}$ induces a bijection:
$$
	\Phi^{-1}(0)/U(\alpha_0) \simeq \M{\Po}{\alv}^{\thv-ss}.
$$
\end{thm}

We immediately have

\begin{cor} \label{unitCorr}
A representation $\V=(V_0;V_s)_{s\in \Po}$ is $\chv=(\chi_s)_{s\in \Po}$-unitarizable if and only if $\V$ is $\mu_\chv$-polystable.
\end{cor}

Applying Theorem \ref{TheoremSection3} we have

\begin{cor}
 Any indecomposable representation  $\V$ of a poset of finite type is $\chv$-unitarizable, where $\chv$ is constructed by formulas \eqref{formulaCanStab} with respect to the dimension of $\V$.
\end{cor}

\subsection{Relation between the categories $\ucats{\Po}{\chv}$ and $\cats{\Po}$}

Recall that $\core{\cats{\Po}}$ of $\cats{\Po}$ is a maximal subgroupoid of $\cats{\Po}$: the subcategory consisting of the same objects as in $\cats{\Po}$, in which morphisms are the isomorphisms in $\cats{\Po}$. Note that $\cats{\Po}$ and $\core{\cats{\Po}}$ are the ``same'' from the classification point of view.  

As we mentioned above there is a forgetful functor $\Fu:\ucats{\Po}{\chv}\to \cats{\Po}$.   It follows from Corollary \ref{unitCorr} that the image of $\Fu$ (on objects) coincides with the objects of $\core{\cats{\Po}^{\chv-ps}}$.  Now we construct the functor in opposite direction. 
Let $\V=(V_0;V_s)\in \cats{\Po}^{\chv-ps}$.  There is unique inner product in $\V_0$ which 
makes $\V$ into a $\chv$-representation. Denote the resulting $\chv$-representation by $\Ft(\V)\in \ucats{\Po}{\chv}$. Given an invertible morphism $g:\V \to \W$ define $\Ft(g)=\varphi$, where $g=\varphi A$ is a right polar decomposition of $g$ ($\varphi$ is a unitary map and $A$ is positive definite). As $g$ is invertible, the right polar decomposition is unique and hence $\Ft(g)$ is well-defined. Also, one checks that it is a morphism between $\Ft(\V)$ and $\Ft(\W)$ (see \cite[Theorem 3]{SamYus}). One can easily see that $\Ft$ preserves the composition of morphisms and therefore yields a functor. 

Consider the following relation on morphisms in $\core{\cats{\Po}}$. Given two morphisms $g_1, g_2:\V \to \W$ we say that $g_1 \sim g_2$ if $\varphi_1=\varphi_2$ in right polar decompositions $g_1=\varphi_1 A_1$ and $g_2=\varphi_2 A_2$ with respect to some inner product in $V_0$ and $W_0$. One can show (the proof is left to the reader) that the relation $\sim$ does not depend on the choice of inner product and in fact is an equivalence relation on morphisms in  $\core{\cats{\Po}}$. By  $\core{\cats{\Po}}/\sim$ we denote the corresponding quotient category and by $\Pi:\core{\cats{\Po}} \to \core{\cats{\Po}}/\sim$ the quotient functor. 
By construction if follows that $\Ft$ factors as $\Ft'\circ \Pi$ (as the unitary parts in polar decomposition of morphisms $g_1\sim g_2$ are the same).

\begin{prop}
Functors $\Pi \circ \Fu$ and $\Ft'$ establish an isomorphism between the categories 
$\ucats{\Po}{\chv}$ and $\core{\cats{\Po}^{\chv-ps}}/\sim$.
\end{prop} 

Summing up the constructions above we have the following
\begin{equation*}
\begin{tikzcd}[column sep = large, row sep = large]
	\ucats{\Po}{\chv} \ar[r,bend left,"\Fu"] & \core{\cats{\Po}^{\chv-ps}}\ar[l,bend left,"\Ft"]\ar[d,"\Pi"]\ar[r,hook] & \cats{\Po} \\ 
	 & \core{\cats{\Po}^{\chv-ps}} \ar[lu,bend left,"\Ft'"]/\sim 
\end{tikzcd}
\end{equation*}

\

In particular one shows that $\ucats{\Po}{\chv}$ has finitely many indecomposable objects for any $\chv$ iff $\cats{\Po}$ has finitely many indecomposable representations (up to an isomorphism), which reproves \cite[Theorem 1]{SamYus}.

\appendix \section{Some additional statements} \label{appA}

\begin{prop} \label{quotientincats}
Let $\V=(V_0;V_s)_{s\in \Po}$ and $K\subset V_0$ then for the induced representation $\V_K=(K;V_s\cap K)_{s\in \Po}$ we have  $\V/\V_K \in \cats{\Po}$.
\end{prop}

\begin{proof}
	Let $s,t\in \Po$ and $s\prec t$. It is enough to show that $V_s/(V_s\cap K)\subset V_t/(V_t\cap K)$. Define the following map 
	$$
		\rho:V_s/(V_s\cap K)\rightarrow V_t/(V_t\cap K),
	$$
by $\rho(x+V_s\cap K)=x+V_t\cap K$, for any $x\in V_s$. Clearly, $\rho$ is linear. If $x_1+V_s\cap K=x_2+V_s\cap K$ then $x_1-x_2 \in V_s\cap K$ and $x_1-x_2 \in V_t\cap K$, therefore $\rho(x_1+V_s\cap K)=\rho(x_2+V_s\cap K)$ and $\rho$ is well-defined. 
We have  $\ker{\rho}=0$, therefore $\rho$ is an inclusion.
\end{proof}

\begin{prop} \label{posvar}
	Let $\alv=(\alpha_0;\alpha_s)_{s\in \Po}$ be a dimension vector. Variety $\R{\alv}{\Po}$ is a Zariski closed, irreducible subset of $\prod_{s\in \Po}\Gr(\alpha_s,\alpha_0)$. 
\end{prop}

\begin{proof}
Given the elements $s_1,\dots,s_m\in \Po$ denote by $\Po(s_1,\dots,s_m)$ a subposet of $\Po$ consisting of these elements. We will use the same letter $\alv$ (abusing the notation) to denote the restriction of the dimension vector on the subposet of $\Po$. 
Clearly, $\R{\alv}{\Po(s)}=\Gr(\alpha_s,\alpha_0)$. Given two incomparable points $s_1$ and $s_2$ we have that  $\R{\alv}{\Po(s_1,s_2)}=\Gr(\alpha_{s_1},\alpha_0)\times \Gr(\alpha_{s_2},\alpha_0)$. If $s_1\prec s_2$ then $\R{\alv}{\Po(s_1,s_2)}$ is a flag of two subspaces and therefore it is a Zariski closed in  
$\Gr(\alpha_{s_1},\alpha_0)\times \Gr(\alpha_{s_2},\alpha_0)$.

Now, for any two $s_1,s_2\in \Po$ let $\pi_{s_1,s_2}$ be the restriction to $\R{\alv}{\Po}$ of the projection
$$
	\prod_{s\in \Po} \Gr(\alpha_s,\alpha_0) \twoheadrightarrow \Gr(\alpha_{s_1},\alpha_0)\times \Gr(\alpha_{s_2},\alpha_0).
$$
Then we have 
$$
	\R{\alv}{\Po}=\bigcap_{(s_1,s_2)\in \Po\times \Po} \pi_{s_1,s_2}^{-1}\big(\R{\alv}{\Po(s_1,s_2)}\big).
$$
Hence $\R{\alv}{\Po}$ is Zariski closed. 

Now we prove that it is irreducible. We proceed by induction on the cardinality of $\Po$. If $\Po$ has one element then $\R{\alv}{\Po}=\Gr(\alpha_{s_1},\alpha_0)$ and hence irreducible. Let $x \in \Po$ be a maximal point in $\Po$. Consider the natural projection 
$$
f:\R{\alv}{\Po} \twoheadrightarrow \R{\widetilde \alv}{\widetilde \Po},$$
where $\widetilde \Po=\Po\setminus \{x\}$ and $\widetilde \alv$ is the restriction of dimension vector $\alv$ onto subposet $\widetilde \Po$. One checks that the generic fibers of this map has the form $\Gr(\alpha_s-X,\alpha_0-X)$ in which $X$ is the dimension of the sum $\sum_{t\to x} V_t$ for the generic point $(V_s)_{s\in \Po} \in \R{\widetilde \alv}{\widetilde \Po}$. Therefore the induction pass follows by \cite[Theorem 11.14]{Harris}, as the fibers of $f$ are irreducibles and each $\R{\widetilde \alv}{\widetilde \Po}$ is irreducible by induction assumption.
\end{proof}

The following Theorem (proved in \cite{CI}) relates the dimension of variety $\R{\alv}{\Po}$ and the Euler quadratic form associated with $\Poe$. 

\begin{thm} \label{dimVariety}
Let $\Po$ be a poset and 
$\alv=(\alpha_0;\alpha_s)_{s\in \Po}$ be an admissible dimension vector. We have that 
$$\dim \R{\alv}{\Po}=\alpha_0^2-\Bf{\Po}(\alv,\alv).$$
\end{thm}

\newpage
\section*{Appendix B. Exact representations of finite representation type, their maximal subcoordinate vectors and costability condition}

For each non-primitive posets $\Po_1,\dots,\Po_6$ of finite representation type (given in Section 3.1) we list its exact representations, costability condition (calculated by formulas  \eqref{formulasCanCstab}) and maximal subcoordinate vectors. 
By $K^n$ we denote the $\F$ vector space with canonical basis $e_1,\dots,e_n$ and 
$K_{i_1\dots i_k,\dots,j_1\dots j_m}$ denotes the subspace of $K^n$ generated by the vectors 
$e_{i_1\dots i_k}, \dots, e_{j_1\dots j_m}$ in which
$$
	e_{i_1\dots i_k}=e_{i_1}+\dots+e_{i_k}, \dots e_{j_1\dots j_m}=e_{j_1}+\dots+e_{j_m}.
$$

Poset $\Po_1$.

\begin{longtable}{|c|}
\hline 
\begin{tabular}{p{6cm}|p{6cm}}
\textbf{Representation}
\begin{tikzpicture}
  \node[draw, circle] (9) at (2,2) {$K^3$};
  \node (2) at (0,1) {$K_{1,2,3}$};
  \node (1) at (0,0) {$K_{123}$};
  \node (4) at (2,1) {$K_{1,2}$};
  \node (3) at (2,0) {$K_{1}$};
  \node (6) at (4,1) {$K_{2,3}$};
  \node (5) at (4,0) {$K_{3}$};
  \draw (1) -- (2) -- (3) -- (4);
  \draw (5) -- (6);
\end{tikzpicture}
&
\textbf{Costability}
\begin{tikzpicture}
  \node[draw, circle] (9) at (2,2) {$-6$};
  \node (2) at (0,1) {$1$};
  \node (1) at (0,0) {$4$};
  \node (4) at (2,1) {$2$};
  \node (3) at (2,0) {$5$};
  \node (6) at (4,1) {$2$};
  \node (5) at (4,0) {$4$};
  \draw (1) -- (2) -- (3) -- (4);
  \draw (5) -- (6);
\end{tikzpicture}
\end{tabular}
\\ \hline
\begin{tabular}{p{12.5cm}}
\textbf{Maximal subcoordinate vectors}\newline \newline
$(1;0,0;1,0;0,0)$,\ $(1;0,1;0,0;1,0)$,\  
$(1;0,1;0,1;0,1)$,\ $(1;1,0;0,0;0,0)$,\  \newline
$(2;0,1;1,0;1,0)$,\ $(2;0,1;1,1;0,1)$,\  
$(2;0,2;0,1;1,1)$,\ $(2;1,0;1,0;0,1)$,\  \newline
$(2;1,1;0,1;0,1)$,\ $(2;1,1;0,1;1,0)$,\  
\end{tabular}
\\ \hline
\end{longtable}

Poset $\Po_2$.

\begin{longtable}{|c|}
\hline 
\begin{tabular}{p{6cm}|p{6cm}}
\textbf{Representation}
\begin{tikzpicture}
  \node[draw, circle] (9) at (2,2) {$K^3$};
  \node (2) at (0,1) {$K_{1,2,3}$};
  \node (1) at (0,0) {$K_{3}$};
  \node (4) at (2,1) {$K_{13,2}$};
  \node (3) at (2,0) {$K_{123}$};
  \node (7) at (4,2) {$K_{1,2,3}$};
  \node (6) at (4,1) {$K_{2,3}$};
  \node (5) at (4,0) {$K_{1}$};
  \draw (1) -- (2) -- (3) -- (4);
  \draw (5) -- (6) -- (7);
\end{tikzpicture}
&
\textbf{Costability}
\begin{tikzpicture}
  \node[draw, circle] (9) at (2,2) {$-7$};
  \node (2) at (0,1) {$1$};
  \node (1) at (0,0) {$4$};
  \node (4) at (2,1) {$2$};
  \node (3) at (2,0) {$5$};
  \node (7) at (4,2) {$1$};
  \node (6) at (4,1) {$3$};
  \node (5) at (4,0) {$5$};
  \draw (1) -- (2) -- (3) -- (4);
  \draw (5) -- (6) -- (7);
\end{tikzpicture}
\end{tabular}
\\ \hline
\begin{tabular}{p{14cm}}
\textbf{Maximal subcoordinate vectors}\newline \newline
$(1;0,0;1,0;0,0,1)$,\ $(1;0,1;0,0;1,0,0)$,\ $(1;0,1;0,1;0,0,1)$,\  
$(1;0,1;0,1;0,1,0)$,\ \newline $(1;1,0;0,0;0,0,1)$,\ $(2;0,1;1,0;1,0,1)$,\  
$(2;0,1;1,1;0,1,1)$,\ $(2;0,2;0,1;1,1,0)$,\ \newline $(2;1,0;1,0;0,1,1)$,\  
$(2;1,1;0,1;0,1,1)$,\ $(2;1,1;0,1;1,0,1)$. \end{tabular}
\\ \hline
\end{longtable}

\newpage

\begin{longtable}{|c|}
\hline 
\begin{tabular}{p{6cm}|p{6cm}}
\textbf{Representation}
\begin{tikzpicture}
  \node[draw, circle] (9) at (2,2) {$K^4$};
  \node (2) at (0,1) {$K_{1,2,4}$};
  \node (1) at (0,0) {$K_{14}$};
  \node (4) at (2,1) {$K_{4,123}$};
  \node (3) at (2,0) {$K_{4}$};
  \node (7) at (4,2) {$K_{1,2,3}$};
  \node (6) at (4,1) {$K_{2,3}$};
  \node (5) at (4,0) {$K_{3}$};
  \draw (1) -- (2) -- (3) -- (4);
  \draw (5) -- (6) -- (7);
\end{tikzpicture}
&
\textbf{Costability}
\begin{tikzpicture}
  \node[draw, circle] (9) at (2,2) {$-7$};
  \node (2) at (0,1) {$2$};
  \node (1) at (0,0) {$5$};
  \node (4) at (2,1) {$3$};
  \node (3) at (2,0) {$6$};
  \node (7) at (4,2) {$2$};
  \node (6) at (4,1) {$4$};
  \node (5) at (4,0) {$6$};
  \draw (1) -- (2) -- (3) -- (4);
  \draw (5) -- (6) -- (7);
\end{tikzpicture}
\end{tabular}
\\ \hline
\begin{tabular}{p{14cm}}
\textbf{Maximal subcoordinate vectors}\newline \newline
$(1;0,0;0,0;1,0,0)$,\ $(1;0,0;0,1;0,0,1)$,\ $(1;0,0;1,0;0,0,0)$,\  
$(1;0,1;0,0;0,0,1)$,\ \newline $(1;0,1;0,0;0,1,0)$,\ $(1;1,0;0,0;0,0,0)$,\  
$(2;0,0;1,0;1,0,0)$,\ $(2;0,0;1,1;0,0,1)$,\ \newline $(2;0,1;0,0;1,1,0)$,\ 
$(2;0,1;0,1;0,1,1)$,\ $(2;0,1;0,1;1,0,1)$,\ $(2;0,1;1,0;0,1,0)$,\ \newline
$(2;0,2;0,0;0,1,1)$,\ $(2;1,0;0,0;1,0,0)$,\ $(2;1,0;0,1;0,0,1)$,\  
$(2;1,0;1,0;0,0,1)$,\ \newline $(2;1,1;0,0;0,1,0)$,\ $(3;0,1;1,0;1,1,0)$,\  
$(3;0,1;1,1;0,1,1)$,\ $(3;0,1;1,1;1,0,1)$,\ \newline $(3;0,2;0,1;1,1,1)$,\  
$(3;1,0;1,0;1,0,1)$,\ $(3;1,0;1,1;0,1,1)$,\ $(3;1,1;0,1;0,1,1)$,\  \newline
$(3;1,1;0,1;1,0,1)$,\ $(3;1,1;0,1;1,1,0)$,\ $(3;1,1;1,0;0,1,1)$.
 \end{tabular}
\\ \hline
\end{longtable}

\begin{longtable}{|c|}
\hline 
\begin{tabular}{p{6cm}|p{6cm}}
\textbf{Representation}
\begin{tikzpicture}
  \node[draw, circle] (9) at (2,2) {$K^4$};
  \node (2) at (0,1) {$K_{1,2,4}$};
  \node (1) at (0,0) {$K_{14}$};
  \node (4) at (2,1) {$K_{4,12,23}$};
  \node (3) at (2,0) {$K_{4}$};
  \node (7) at (4,2) {$K_{1,2,3}$};
  \node (6) at (4,1) {$K_{2,3}$};
  \node (5) at (4,0) {$K_{3}$};
  \draw (1) -- (2) -- (3) -- (4);
  \draw (5) -- (6) -- (7);
\end{tikzpicture}
&
\textbf{Costability}
\begin{tikzpicture}
  \node[draw, circle] (9) at (2,2) {$-8$};
  \node (2) at (0,1) {$2$};
  \node (1) at (0,0) {$5$};
  \node (4) at (2,1) {$3$};
  \node (3) at (2,0) {$7$};
  \node (7) at (4,2) {$2$};
  \node (6) at (4,1) {$4$};
  \node (5) at (4,0) {$6$};
  \draw (1) -- (2) -- (3) -- (4);
  \draw (5) -- (6) -- (7);
\end{tikzpicture}
\end{tabular}
\\ \hline
\begin{tabular}{p{14cm}}
\textbf{Maximal subcoordinate vectors}\newline \newline
$(1;0,0;0,0;1,0,0)$,\ $(1;0,0;0,1;0,1,0)$,\ $(1;0,0;1,0;0,0,0)$,\  
$(1;0,1;0,0;0,1,0)$,\ \newline $(1;0,1;0,1;0,0,1)$,\ $(1;1,0;0,0;0,0,0)$,\  
$(2;0,0;1,0;1,0,0)$,\ $(2;0,0;1,1;0,1,0)$,\ \newline $(2;0,1;0,1;1,0,1)$,\ 
$(2;0,1;0,1;1,1,0)$,\ $(2;0,1;0,2;0,1,1)$,\ $(2;0,1;1,0;0,1,0)$,\  \newline
$(2;0,1;1,1;0,0,1)$,\ $(2;0,2;0,1;0,1,1)$,\ $(2;1,0;0,1;1,0,0)$,\  
$(2;1,0;1,0;0,0,1)$,\ \newline $(2;1,1;0,1;0,0,1)$,\ $(2;1,1;0,1;0,1,0)$,\  
$(3;0,1;1,1;1,0,1)$,\ $(3;0,1;1,1;1,1,0)$,\ \newline $(3;0,1;1,2;0,1,1)$,\  
$(3;0,2;0,2;1,1,1)$,\ $(3;1,0;1,1;1,0,1)$,\ $(3;1,1;0,2;0,1,1)$,\  \newline
$(3;1,1;0,2;1,0,1)$,\ $(3;1,1;0,2;1,1,0)$,\ $(3;1,1;1,1;0,1,1)$.
\end{tabular}
\\ \hline
\end{longtable}

\newpage

\begin{longtable}{|c|}
\hline 
\begin{tabular}{p{6cm}|p{6cm}}
\textbf{Representation}
\begin{tikzpicture}
  \node[draw, circle] (9) at (2,2) {$K^4$};
  \node (2) at (0,1) {$K_{1,2,3,4}$};
  \node (1) at (0,0) {$K_{1,24}$};
  \node (4) at (2,1) {$K_{4,123}$};
  \node (3) at (2,0) {$K_{4}$};
  \node (7) at (4,2) {$K_{1,2,3}$};
  \node (6) at (4,1) {$K_{2,3}$};
  \node (5) at (4,0) {$K_{3}$};
  \draw (1) -- (2) -- (3) -- (4);
  \draw (5) -- (6) -- (7);
\end{tikzpicture}
&
\textbf{Costability}
\begin{tikzpicture}
  \node[draw, circle] (9) at (2,2) {$-8$};
  \node (2) at (0,1) {$1$};
  \node (1) at (0,0) {$5$};
  \node (4) at (2,1) {$3$};
  \node (3) at (2,0) {$6$};
  \node (7) at (4,2) {$2$};
  \node (6) at (4,1) {$4$};
  \node (5) at (4,0) {$6$};
  \draw (1) -- (2) -- (3) -- (4);
  \draw (5) -- (6) -- (7);
\end{tikzpicture}
\end{tabular}
\\ \hline
\begin{tabular}{p{14cm}}
\textbf{Maximal subcoordinate vectors}\newline \newline
$(1;0,0;1,0;0,0,0)$,\ $(1;0,1;0,0;0,1,0)$,\ $(1;0,1;0,0;1,0,0)$,\ $(1;0,1;0,1;0,0,1)$,\  \newline
$(1;1,0;0,0;0,0,1)$,\ $(2;0,1;1,0;1,0,0)$,\ $(2;0,1;1,1;0,0,1)$,\ $(2;0,2;0,0;1,1,0)$,\  \newline
$(2;0,2;0,1;0,1,1)$,\ $(2;0,2;0,1;1,0,1)$,\ $(2;1,0;1,0;0,0,1)$,\ $(2;1,0;1,0;0,1,0)$,\  \newline
$(2;1,1;0,0;1,0,1)$,\ $(2;1,1;0,1;0,1,1)$,\ $(2;2,0;0,0;0,0,1)$,\ $(3;1,1;1,0;1,1,0)$,\  \newline
$(3;1,1;1,1;0,1,1)$,\ $(3;1,1;1,1;1,0,1)$,\ $(3;1,2;0,1;1,1,1)$,\ $(3;2,0;1,0;0,1,1)$,\  \newline
$(3;2,1;0,1;0,1,1)$,\ $(3;2,1;0,1;1,0,1)$. 
\end{tabular}
\\ \hline
\end{longtable}

\begin{longtable}{|c|}
\hline 
\begin{tabular}{p{6cm}|p{6cm}}
\textbf{Representation}
\begin{tikzpicture}
  \node[draw, circle] (9) at (2,2) {$K^4$};
  \node (2) at (0,1) {$K_{1,2,3,4}$};
  \node (1) at (0,0) {$K_{1,24}$};
  \node (4) at (2,1) {$K_{4,12,23}$};
  \node (3) at (2,0) {$K_{4}$};
  \node (7) at (4,2) {$K_{1,2,3}$};
  \node (6) at (4,1) {$K_{2,3}$};
  \node (5) at (4,0) {$K_{3}$};
  \draw (1) -- (2) -- (3) -- (4);
  \draw (5) -- (6) -- (7);
\end{tikzpicture}
&
\textbf{Costability}
\begin{tikzpicture}
  \node[draw, circle] (9) at (2,2) {$-9$};
  \node (2) at (0,1) {$1$};
  \node (1) at (0,0) {$5$};
  \node (4) at (2,1) {$3$};
  \node (3) at (2,0) {$7$};
  \node (7) at (4,2) {$2$};
  \node (6) at (4,1) {$4$};
  \node (5) at (4,0) {$6$};
  \draw (1) -- (2) -- (3) -- (4);
  \draw (5) -- (6) -- (7);
\end{tikzpicture}
\end{tabular}
\\ \hline
\begin{tabular}{p{14cm}}
\textbf{Maximal subcoordinate vectors}\newline \newline
$(1;0,0;1,0;0,0,0)$,\ $(1;0,1;0,0;1,0,0)$,\ $(1;0,1;0,1;0,0,1)$,\ $(1;0,1;0,1;0,1,0)$,\  \newline
$(1;1,0;0,0;0,0,1)$,\ $(2;0,1;1,0;1,0,0)$,\ $(2;0,1;1,1;0,1,0)$,\ $(2;0,2;0,1;1,0,1)$,\  \newline
$(2;0,2;0,1;1,1,0)$,\ $(2;0,2;0,2;0,1,1)$,\ $(2;1,0;1,0;0,1,0)$,\ $(2;1,0;1,1;0,0,1)$,\  \newline
$(2;1,1;0,1;0,1,1)$,\ $(2;1,1;0,1;1,0,1)$,\ $(2;2,0;0,1;0,0,1)$,\ $(3;1,1;1,1;1,0,1)$,\  \newline
$(3;1,1;1,1;1,1,0)$,\ $(3;1,1;1,2;0,1,1)$,\ $(3;1,2;0,2;1,1,1)$,\ $(3;2,0;1,1;0,1,1)$,\  \newline
$(3;2,1;0,2;0,1,1)$,\ $(3;2,1;0,2;1,0,1)$.
\end{tabular}
\\ \hline
\end{longtable}

\newpage

\begin{longtable}{|c|}
\hline 
\begin{tabular}{p{6cm}|p{6cm}}
\textbf{Representation}
\begin{tikzpicture}
  \node[draw, circle] (9) at (2,2) {$K^5$};
  \node (2) at (0,1) {$K_{1,2,4,5}$};
  \node (1) at (0,0) {$K_{4,15}$};
  \node (4) at (2,1) {$K_{123,24,5}$};
  \node (3) at (2,0) {$K_{5}$};
  \node (7) at (4,2) {$K_{1,2,3}$};
  \node (6) at (4,1) {$K_{2,3}$};
  \node (5) at (4,0) {$K_{3}$};
  \draw (1) -- (2) -- (3) -- (4);
  \draw (5) -- (6) -- (7);
\end{tikzpicture}
&
\textbf{Costability}
\begin{tikzpicture}
  \node[draw, circle] (9) at (2,2) {$-9$};
  \node (2) at (0,1) {$2$};
  \node (1) at (0,0) {$6$};
  \node (4) at (2,1) {$4$};
  \node (3) at (2,0) {$8$};
  \node (7) at (4,2) {$3$};
  \node (6) at (4,1) {$5$};
  \node (5) at (4,0) {$7$};
  \draw (1) -- (2) -- (3) -- (4);
  \draw (5) -- (6) -- (7);
\end{tikzpicture}
\end{tabular}
\\ \hline
\begin{tabular}{p{14cm}}
\textbf{Maximal subcoordinate vectors}\newline \newline
$(1;0,0;0,0;1,0,0)$,\ $(1;0,0;0,1;0,0,1)$,\ $(1;0,0;1,0;0,0,0)$,\ $(1;0,1;0,0;0,0,1)$,\  \newline
$(1;0,1;0,0;0,1,0)$,\ $(1;0,1;0,1;0,0,0)$,\ $(1;1,0;0,0;0,0,0)$,\ $(2;0,0;1,0;1,0,0)$,\  \newline
$(2;0,0;1,1;0,0,1)$,\ $(2;0,1;0,0;1,1,0)$,\ $(2;0,1;0,1;0,1,1)$,\ $(2;0,1;0,1;1,0,1)$,\  \newline
$(2;0,1;0,2;0,0,1)$,\ $(2;0,1;1,0;0,1,0)$,\ $(2;0,1;1,1;0,0,0)$,\ $(2;0,2;0,0;0,1,1)$,\  \newline
$(2;0,2;0,1;0,0,1)$,\ $(2;1,0;0,0;1,0,0)$,\ $(2;1,0;0,1;0,0,1)$,\ $(2;1,0;1,0;0,0,1)$,\  \newline
$(2;1,1;0,0;0,0,1)$,\ $(2;1,1;0,1;0,1,0)$,\ $(2;2,0;0,0;0,0,0)$,\ $(3;0,1;1,0;1,1,0)$,\  \newline
$(3;0,1;1,1;0,1,1)$,\ $(3;0,1;1,1;1,0,1)$,\ $(3;0,1;1,2;0,0,1)$,\ $(3;0,2;0,1;1,1,1)$,\  \newline
$(3;0,2;0,2;0,1,1)$,\ $(3;0,2;0,2;1,0,1)$,\ $(3;1,0;1,0;1,0,1)$,\ $(3;1,0;1,1;0,1,1)$,\  \newline
$(3;1,1;0,1;1,0,1)$,\ $(3;1,1;0,1;1,1,0)$,\ $(3;1,1;0,2;0,1,1)$,\ $(3;1,1;1,0;0,1,1)$,\  \newline
$(3;1,1;1,1;0,0,1)$,\ $(3;1,1;1,1;0,1,0)$,\ $(3;1,2;0,1;0,1,1)$,\ $(3;2,0;0,1;0,0,1)$,\  \newline
$(3;2,0;0,1;1,0,0)$,\ $(3;2,0;1,0;0,0,1)$,\ $(3;2,1;0,1;0,1,0)$,\ $(4;1,1;1,1;1,1,1)$,\  \newline
$(4;1,1;1,2;0,1,1)$,\ $(4;1,1;1,2;1,0,1)$,\ $(4;1,2;0,2;1,1,1)$,\ $(4;2,0;1,1;1,0,1)$,\  \newline
$(4;2,1;0,2;0,1,1)$,\ $(4;2,1;0,2;1,0,1)$,\ $(4;2,1;0,2;1,1,0)$,\ $(4;2,1;1,1;0,1,1)$.\end{tabular}
\\ \hline
\end{longtable}

\begin{longtable}{|c|}
\hline 
\begin{tabular}{p{6cm}|p{6cm}}
\textbf{Representation}
\begin{tikzpicture}
  \node[draw, circle] (9) at (2,2) {$K^5$};
  \node (2) at (0,1) {$K_{2,3,4,5}$};
  \node (1) at (0,0) {$K_{3,5}$};
  \node (4) at (2,1) {$K_{134,24,45}$};
  \node (3) at (2,0) {$K_{45}$};
  \node (7) at (4,2) {$K_{1,2,3,4}$};
  \node (6) at (4,1) {$K_{1,2}$};
  \node (5) at (4,0) {$K_{1}$};
  \draw (1) -- (2) -- (3) -- (4);
  \draw (5) -- (6) -- (7);
\end{tikzpicture}
&
\textbf{Costability}
\begin{tikzpicture}
  \node[draw, circle] (9) at (2,2) {$-10$};
  \node (2) at (0,1) {$2$};
  \node (1) at (0,0) {$6$};
  \node (4) at (2,1) {$4$};
  \node (3) at (2,0) {$8$};
  \node (7) at (4,2) {$3$};
  \node (6) at (4,1) {$6$};
  \node (5) at (4,0) {$8$};
  \draw (1) -- (2) -- (3) -- (4);
  \draw (5) -- (6) -- (7);
\end{tikzpicture}
\end{tabular}
\\ \hline
\begin{tabular}{p{14cm}}
\textbf{Maximal subcoordinate vectors}\newline \newline
$(1;0,0;0,0;1,0,0)$,\ $(1;0,0;1,0;0,0,0)$,\ $(1;0,1;0,0;0,1,0)$,\ $(1;0,1;0,1;0,0,1)$,\  \newline
$(1;1,0;0,0;0,0,1)$,\ $(2;0,0;1,0;1,0,0)$,\ $(2;0,1;0,0;1,1,0)$,\ $(2;0,1;0,1;1,0,1)$,\  \newline
$(2;0,1;0,2;0,0,2)$,\ $(2;0,1;1,0;0,1,0)$,\ $(2;0,1;1,1;0,0,1)$,\ $(2;0,2;0,1;0,1,1)$,\  \newline
$(2;1,0;0,0;1,0,1)$,\ $(2;1,0;1,0;0,0,1)$,\ $(2;1,1;0,0;0,1,1)$,\ $(2;1,1;0,1;0,0,2)$,\  \newline
$(2;1,1;0,1;0,1,0)$,\ $(2;2,0;0,0;0,0,1)$,\ $(3;0,1;1,0;1,1,0)$,\ $(3;0,1;1,1;1,0,1)$,\  \newline
$(3;0,1;1,2;0,0,2)$,\ $(3;0,2;0,1;1,1,1)$,\ $(3;0,2;0,2;0,1,2)$,\ $(3;0,2;0,2;1,0,2)$,\  \newline
$(3;1,0;1,1;1,0,1)$,\ $(3;1,1;0,1;1,0,2)$,\ $(3;1,1;0,1;1,1,1)$,\ $(3;1,1;0,2;0,1,2)$,\  \newline
$(3;1,1;1,1;0,0,2)$,\ $(3;1,1;1,1;0,1,1)$,\ $(3;1,2;0,1;0,1,2)$,\ $(3;2,0;0,1;1,0,1)$,\  \newline
$(3;2,0;1,0;0,0,2)$,\ $(3;2,1;0,1;0,0,2)$,\ $(3;2,1;0,1;0,1,1)$,\ $(4;1,1;1,1;1,1,1)$,\  \newline
$(4;1,1;1,2;0,1,2)$,\ $(4;1,1;1,2;1,0,2)$,\ $(4;1,2;0,2;1,1,2)$,\ $(4;2,0;1,1;1,0,2)$,\  \newline
$(4;2,1;0,2;0,1,2)$,\ $(4;2,1;0,2;1,0,2)$,\ $(4;2,1;0,2;1,1,1)$,\ $(4;2,1;1,1;0,1,2)$.\end{tabular}
\\ \hline
\end{longtable}

\begin{longtable}{|c|}
\hline 
\begin{tabular}{p{6cm}|p{6cm}}
\textbf{Representation}
\begin{tikzpicture}
  \node[draw, circle] (9) at (2,2) {$K^5$};
  \node (2) at (0,1) {$K_{1,2,3,5}$};
  \node (1) at (0,0) {$K_{1,25}$};
  \node (4) at (2,1) {$K_{13,234,5}$};
  \node (3) at (2,0) {$K_{5}$};
  \node (7) at (4,2) {$K_{1,2,3,4}$};
  \node (6) at (4,1) {$K_{2,3,4}$};
  \node (5) at (4,0) {$K_{4}$};
  \draw (1) -- (2) -- (3) -- (4);
  \draw (5) -- (6) -- (7);
\end{tikzpicture}
&
\textbf{Costability}
\begin{tikzpicture}
  \node[draw, circle] (9) at (2,2) {$-10$};
  \node (2) at (0,1) {$2$};
  \node (1) at (0,0) {$6$};
  \node (4) at (2,1) {$4$};
  \node (3) at (2,0) {$8$};
  \node (7) at (4,2) {$2$};
  \node (6) at (4,1) {$5$};
  \node (5) at (4,0) {$8$};
  \draw (1) -- (2) -- (3) -- (4);
  \draw (5) -- (6) -- (7);
\end{tikzpicture}
\end{tabular}
\\ \hline
\begin{tabular}{p{14cm}}
\textbf{Maximal subcoordinate vectors}\newline \newline
$(1;0,0;0,0;1,0,0)$,\ $(1;0,0;0,1;0,1,0)$,\ $(1;0,0;1,0;0,0,0)$,\ $(1;0,1;0,0;0,1,0)$,\  \newline
$(1;0,1;0,1;0,0,1)$,\ $(1;1,0;0,0;0,0,1)$,\ $(2;0,0;1,0;1,0,0)$,\ $(2;0,0;1,1;0,1,0)$,\  \newline
$(2;0,1;0,1;0,2,0)$,\ $(2;0,1;0,1;1,0,1)$,\ $(2;0,1;0,1;1,1,0)$,\ $(2;0,1;0,2;0,1,1)$,\  \newline
$(2;0,1;1,0;0,1,0)$,\ $(2;0,1;1,1;0,0,1)$,\ $(2;0,2;0,0;0,2,0)$,\ $(2;0,2;0,1;0,1,1)$,\  \newline
$(2;1,0;0,0;1,0,1)$,\ $(2;1,0;1,0;0,0,1)$,\ $(2;1,0;1,0;0,1,0)$,\ $(2;1,1;0,1;0,1,1)$,\  \newline
$(2;2,0;0,0;0,0,1)$,\ $(3;0,1;1,1;0,2,0)$,\ $(3;0,1;1,1;1,0,1)$,\ $(3;0,1;1,1;1,1,0)$,\  \newline
$(3;0,1;1,2;0,1,1)$,\ $(3;0,2;0,1;1,2,0)$,\ $(3;0,2;0,2;0,2,1)$,\ $(3;0,2;0,2;1,1,1)$,\  \newline
$(3;1,0;1,0;1,0,1)$,\ $(3;1,0;1,0;1,1,0)$,\ $(3;1,0;1,1;0,2,0)$,\ $(3;1,1;0,1;1,1,1)$,\  \newline
$(3;1,1;0,2;0,2,1)$,\ $(3;1,1;1,0;0,2,0)$,\ $(3;1,1;1,1;0,1,1)$,\ $(3;1,2;0,1;0,2,1)$,\  \newline
$(3;2,0;0,1;1,0,1)$,\ $(3;2,0;1,0;0,1,1)$,\ $(3;2,1;0,1;0,1,1)$,\ $(4;1,1;1,1;1,2,0)$,\  \newline
$(4;1,1;1,2;0,2,1)$,\ $(4;1,1;1,2;1,1,1)$,\ $(4;1,2;0,2;1,2,1)$,\ $(4;2,0;1,1;1,1,1)$,\  \newline
$(4;2,1;0,2;0,2,1)$,\ $(4;2,1;0,2;1,1,1)$,\ $(4;2,1;1,1;0,2,1)$.
\end{tabular}
\\ \hline
\end{longtable}

\newpage
\begin{longtable}{|c|}
\hline 
\begin{tabular}{p{6cm}|p{6cm}}
\textbf{Representation}
\begin{tikzpicture}
  \node[draw, circle] (9) at (2,2) {$K^5$};
  \node (2) at (0,1) {$K_{1,2,3,5}$};
  \node (1) at (0,0) {$K_{1,25}$};
  \node (4) at (2,1) {$K_{123,24,5}$};
  \node (3) at (2,0) {$K_{5}$};
  \node (7) at (4,2) {$K_{1,2,3,4}$};
  \node (6) at (4,1) {$K_{2,3,4}$};
  \node (5) at (4,0) {$K_{3,4}$};
  \draw (1) -- (2) -- (3) -- (4);
  \draw (5) -- (6) -- (7);
\end{tikzpicture}
&
\textbf{Costability}
\begin{tikzpicture}
  \node[draw, circle] (9) at (2,2) {$-10$};
  \node (2) at (0,1) {$2$};
  \node (1) at (0,0) {$6$};
  \node (4) at (2,1) {$4$};
  \node (3) at (2,0) {$8$};
  \node (7) at (4,2) {$2$};
  \node (6) at (4,1) {$4$};
  \node (5) at (4,0) {$7$};
  \draw (1) -- (2) -- (3) -- (4);
  \draw (5) -- (6) -- (7);
\end{tikzpicture}
\end{tabular}
\\ \hline
\begin{tabular}{p{14cm}}
\textbf{Maximal subcoordinate vectors}\newline \newline
$(1;0,0;0,1;0,1,0)$,\ $(1;0,0;1,0;0,0,0)$,\ $(1;0,1;0,0;0,1,0)$,\ $(1;0,1;0,0;1,0,0)$,\  \newline
$(1;0,1;0,1;0,0,1)$,\ $(1;1,0;0,0;0,0,1)$,\ $(2;0,0;1,1;0,1,0)$,\ $(2;0,1;0,0;2,0,0)$,\  \newline
$(2;0,1;0,1;1,1,0)$,\ $(2;0,1;0,2;0,1,1)$,\ $(2;0,1;1,0;1,0,0)$,\ $(2;0,1;1,1;0,0,1)$,\  \newline
$(2;0,2;0,0;1,1,0)$,\ $(2;0,2;0,1;0,1,1)$,\ $(2;0,2;0,1;1,0,1)$,\ $(2;1,0;0,1;1,0,0)$,\  \newline
$(2;1,0;1,0;0,0,1)$,\ $(2;1,0;1,0;0,1,0)$,\ $(2;1,1;0,0;1,0,1)$,\ $(2;1,1;0,1;0,1,1)$,\  \newline
$(2;2,0;0,0;0,0,1)$,\ $(3;0,1;1,0;2,0,0)$,\ $(3;0,1;1,1;1,1,0)$,\ $(3;0,1;1,2;0,1,1)$,\  \newline
$(3;0,2;0,1;2,0,1)$,\ $(3;0,2;0,1;2,1,0)$,\ $(3;0,2;0,2;1,1,1)$,\ $(3;1,0;1,1;1,1,0)$,\  \newline
$(3;1,1;0,1;2,0,1)$,\ $(3;1,1;0,2;1,1,1)$,\ $(3;1,1;1,0;1,1,0)$,\ $(3;1,1;1,1;0,1,1)$,\  \newline
$(3;1,1;1,1;1,0,1)$,\ $(3;1,2;0,1;1,1,1)$,\ $(3;2,0;1,0;0,1,1)$,\ $(3;2,1;0,1;0,1,1)$,\  \newline
$(3;2,1;0,1;1,0,1)$,\ $(4;1,1;1,1;2,0,1)$,\ $(4;1,1;1,1;2,1,0)$,\ $(4;1,1;1,2;1,1,1)$,\  \newline
$(4;1,2;0,2;2,1,1)$,\ $(4;2,1;0,2;1,1,1)$,\ $(4;2,1;0,2;2,0,1)$,\ $(4;2,1;1,1;1,1,1)$.\end{tabular}
\\ \hline
\end{longtable}

Poset $\Po_3$

\begin{longtable}{|@{}c@{}|}
\hline 
\begin{tabular}{p{6cm}|p{6cm}}
\textbf{Representation}
\vspace*{0cm}%
\begin{tikzpicture}
  \node[draw, circle] (9) at (3,3) {$K^4$};
  \node (3) at (0,2) {$K_{1,2,3,4}$};
  \node (2) at (0,1) {$K_{3,4}$};
  \node (1) at (0,0) {$K_{4}$};
  \node (6) at (2,2) {$K_{1}$};
  \node (5) at (2,1) {$K_{12,234,4}$};
  \node (4) at (2,0) {$K_{234}$};
  \node (7) at (4,0) {$K_{1,2}$};
  \draw (1) -- (2) -- (3) -- (4) -- (5) -- (6);
\end{tikzpicture}
&
\textbf{Costability}
\begin{tikzpicture}
  \node[draw, circle] (9) at (3,3) {$-8$};
  \node (3) at (0,2) {$1$};
  \node (2) at (0,1) {$4$};
  \node (1) at (0,0) {$6$};
  \node (6) at (2,2) {$2$};
  \node (5) at (2,1) {$4$};
  \node (4) at (2,0) {$7$};
  \node (7) at (4,0) {$4$};
  \draw (1) -- (2) -- (3) -- (4) -- (5) -- (6);
\end{tikzpicture}
\end{tabular}
\\ \hline
\begin{tabular}{p{14cm}}
\textbf{Maximal subcoordinate vectors}\newline \newline
$(1;0,0,0;1,0,0;0)$,\ $(1;0,0,1;0,0,1;1)$,\ $(1;0,0,1;0,1,0;0)$,\  
$(1;0,1,0;0,0,1;0)$,\ \newline $(1;1,0,0;0,0,0;0)$,\ $(2;0,0,1;1,0,1;1)$,\  
$(2;0,0,1;1,1,0;0)$,\ $(2;0,0,2;0,0,1;2)$,\ \newline  $(2;0,0,2;0,1,1;1)$,\  
$(2;0,1,0;1,0,1;0)$,\ $(2;0,1,1;0,1,1;1)$,\ $(2;1,0,0;1,0,0;0)$,\  \newline 
$(2;1,0,1;0,0,1;1)$,\ $(2;1,0,1;0,1,0;1)$,\ $(2;1,1,0;0,0,1;0)$,\  
$(3;0,1,1;1,0,1;2)$,\ \newline  $(3;0,1,1;1,1,1;1)$,\ $(3;0,1,2;0,1,1;2)$,\  
$(3;1,0,1;1,0,1;1)$,\ $(3;1,0,1;1,1,0;1)$,\ \newline  $(3;1,0,2;0,1,1;2)$,\  
$(3;1,1,0;1,0,1;1)$,\ $(3;1,1,1;0,1,1;1)$.
\end{tabular}
\\ \hline
\end{longtable}

Poset $\Po_4$
\begin{longtable}{|@{}c@{}|}
\hline
\begin{tabular}{|p{6cm}|p{6cm}|}
\textbf{Representation}
\vspace*{0cm}%
\begin{tikzpicture}
  \node[draw, circle] (9) at (2,2.5) {$K^4$};
  \node (3) at (0,2) {$K_{1,2,3,4}$};
  \node (2) at (0,1) {$K_{3,4}$};
  \node (1) at (0,0) {$K_{4}$};
  \node (5) at (2,1) {$K_{12,234,4}$};
  \node (4) at (2,0) {$K_{234}$};
  \node (8) at (4,2) {$K_{1,2,3}$};
  \node (7) at (4,1) {$K_{1,2}$};
  \node (6) at (4,0) {$K_{1}$};
  \draw (1) -- (2) -- (3) -- (4) -- (5) -- (1);
  \draw (6) -- (7) -- (8);
\end{tikzpicture}
&
\textbf{Costability}
\begin{tikzpicture}
  \node[draw, circle] (9) at (2,2.5) {$-8$};
  \node (3) at (0,2) {$1$};
  \node (2) at (0,1) {$4$};
  \node (1) at (0,0) {$7$};
  \node (5) at (2,1) {$2$};
  \node (4) at (2,0) {$6$};
  \node (8) at (4,2) {$2$};
  \node (7) at (4,1) {$4$};
  \node (6) at (4,0) {$6$};
  \draw (1) -- (2) -- (3) -- (4) -- (5) -- (1);
  \draw (6) -- (7) -- (8);
\end{tikzpicture}
\end{tabular}
\\ \hline
\begin{tabular}{p{12cm}}
\textbf{Maximal subcoordinate vectors}\newline \newline
$(1;0,0,0;1,0;0,0,0)$,\ $(1;0,0,1;0,0;1,0,0)$,\ $(1;0,0,1;0,1;0,1,0)$,\  \newline
$(1;0,1,0;0,0;0,0,1)$,\ $(1;1,0,0;0,0;0,0,0)$,\ $(2;0,0,1;1,0;1,0,0)$,\  \newline
$(2;0,0,1;1,1;0,1,0)$,\ $(2;0,0,2;0,1;1,1,0)$,\ $(2;0,1,0;1,0;0,0,1)$,\  \newline
$(2;0,1,0;1,0;0,1,0)$,\ $(2;0,1,1;0,1;0,1,1)$,\ $(2;0,1,1;0,1;1,0,1)$,\  \newline
$(2;1,0,0;1,0;0,0,1)$,\ $(2;1,0,1;0,0;1,0,0)$,\ $(2;1,0,1;0,1;0,1,0)$,\  \newline
$(2;1,1,0;0,0;0,0,1)$,\ $(3;0,1,1;1,1;0,1,1)$,\ $(3;0,1,1;1,1;1,0,1)$,\  \newline
$(3;0,1,1;1,1;1,1,0)$,\ $(3;0,1,2;0,2;1,1,1)$,\ $(3;1,0,1;1,0;1,0,1)$,\  \newline
$(3;1,0,1;1,1;0,1,1)$,\ $(3;1,0,2;0,1;1,1,0)$,\ $(3;1,1,0;1,0;0,1,1)$,\  \newline
$(3;1,1,1;0,1;0,1,1)$,\ $(3;1,1,1;0,1;1,0,1)$.
\end{tabular}
\\ \hline
\end{longtable}

\newpage
Poset $\Po_5$.

\begin{longtable}{|@{}c@{}|}
\hline 
\begin{tabular}{p{6cm}|p{6cm}}
\textbf{Representation}
\begin{tikzpicture}
  \node[draw, circle] (9) at (0.8,3) {$K^5$};
  \node (2) at (0,1) {$K_{1,2,3,5}$};
  \node (1) at (0,0) {$K_{13,125}$};
  \node (4) at (2,1) {$K_{1,24,5}$};
  \node (3) at (2,0) {$K_{5}$};
  \node (8) at (4,3) {$K_{1,2,3,4,5}$};
  \node (7) at (4,2) {$K_{2,3,4}$};
  \node (6) at (4,1) {$K_{3,4}$};
  \node (5) at (4,0) {$K_4$};
  \draw (5) -- (6) -- (7) -- (8);
  \draw (4) -- (3) -- (8);
  \draw (1) -- (2) -- (3);
\end{tikzpicture}
&
\textbf{Costability}
\begin{tikzpicture}
  \node[draw, circle] (9) at (0.8,3) {$-10$};
  \node (2) at (0,1) {$2$};
  \node (1) at (0,0) {$6$};
  \node (4) at (2,1) {$4$};
  \node (3) at (2,0) {$9$};
  \node (8) at (4,3) {$1$};
  \node (7) at (4,2) {$4$};
  \node (6) at (4,1) {$6$};
  \node (5) at (4,0) {$8$};
  \draw (5) -- (6) -- (7) -- (8);
  \draw (4) -- (3) -- (8);
  \draw (1) -- (2) -- (3);
\end{tikzpicture}
\end{tabular}
\\ \hline
\begin{tabular}{p{12.5cm}}
\textbf{Maximal subcoordinate vectors}\newline \newline
$(1;0,0;0,0;1,0,0,0)$,\ $(1;0,0;0,1;0,0,1,0)$,\ $(1;0,0;1,0;0,0,0,0)$,\  \newline
$(1;0,1;0,0;0,0,1,0)$,\ $(1;0,1;0,0;0,1,0,0)$,\ $(1;0,1;0,1;0,0,0,1)$,\  \newline
$(1;1,0;0,0;0,0,0,1)$,\ $(2;0,0;1,0;1,0,0,0)$,\ $(2;0,0;1,1;0,0,1,0)$,\  \newline
$(2;0,1;0,0;1,1,0,0)$,\ $(2;0,1;0,1;0,1,1,0)$,\ $(2;0,1;0,1;1,0,0,1)$,\  \newline
$(2;0,1;0,1;1,0,1,0)$,\ $(2;0,1;0,2;0,0,1,1)$,\ $(2;0,1;1,0;0,0,1,0)$,\  \newline
$(2;0,1;1,0;0,1,0,0)$,\ $(2;0,1;1,1;0,0,0,1)$,\ $(2;0,2;0,0;0,1,1,0)$,\  \newline
$(2;0,2;0,1;0,0,1,1)$,\ $(2;1,0;0,1;1,0,0,1)$,\ $(2;1,0;1,0;0,0,0,1)$,\  \newline
$(2;1,1;0,1;0,0,0,2)$,\ $(2;1,1;0,1;0,0,1,1)$,\ $(2;1,1;0,1;0,1,0,1)$,\  \newline
$(2;2,0;0,0;0,0,0,2)$,\ $(3;0,1;1,0;1,1,0,0)$,\ $(3;0,1;1,1;0,1,1,0)$,\  \newline
$(3;0,1;1,1;1,0,0,1)$,\ $(3;0,1;1,1;1,0,1,0)$,\ $(3;0,1;1,2;0,0,1,1)$,\  \newline
$(3;0,2;0,1;1,1,1,0)$,\ $(3;0,2;0,2;1,0,1,1)$,\ $(3;1,0;1,1;1,0,0,1)$,\  \newline
$(3;1,1;0,1;1,1,0,1)$,\ $(3;1,1;0,2;0,0,1,2)$,\ $(3;1,1;0,2;0,1,1,1)$,\  \newline
$(3;1,1;0,2;1,0,0,2)$,\ $(3;1,1;0,2;1,0,1,1)$,\ $(3;1,1;1,0;0,1,1,0)$,\  \newline
$(3;1,1;1,1;0,0,1,1)$,\ $(3;1,1;1,1;0,1,0,1)$,\ $(3;1,2;0,1;0,1,1,1)$,\  \newline
$(3;2,0;0,1;1,0,0,2)$,\ $(3;2,0;1,0;0,0,1,1)$,\ $(3;2,1;0,1;0,0,1,2)$,\  \newline
$(3;2,1;0,1;0,1,0,2)$,\ $(4;1,1;1,1;1,1,0,1)$,\ $(4;1,1;1,1;1,1,1,0)$,\  \newline
$(4;1,1;1,2;0,1,1,1)$,\ $(4;1,1;1,2;1,0,1,1)$,\ $(4;1,2;0,2;1,1,1,1)$,\  \newline
$(4;2,0;1,1;1,0,1,1)$,\ $(4;2,1;0,2;0,1,1,2)$,\ $(4;2,1;0,2;1,0,1,2)$,\  \newline
$(4;2,1;0,2;1,1,0,2)$,\ $(4;2,1;1,1;0,1,1,1)$.
\end{tabular}
\\ \hline
\end{longtable}

\newpage

Poset $\Po_6$. 

\begin{longtable}{|@{}c@{}|}
\hline 
\begin{tabular}{p{6cm}|p{6cm}}
\textbf{Representation}
\begin{tikzpicture}
  \node[draw, circle] (9) at (0.8,2.6) {$K^5$};
  \node (2) at (0,1) {$K_{13,234,5}$};
  \node (1) at (0,0) {$K_{5}$};
  \node (4) at (2,1) {$K_{1,2,3,4,5}$};
  \node (3) at (2,0) {$K_{1,25}$};
  \node (8) at (4,3) {$K_{1,2,3,4}$};
  \node (7) at (4,2) {$K_{2,3,4}$};
  \node (6) at (4,1) {$K_{3,4}$};
  \node (5) at (4,0) {$K_4$};
  \draw (5) -- (6) -- (7) -- (8);
  \draw (5) -- (4) -- (3);
  \draw (4) -- (1) -- (2);
\end{tikzpicture}
&
\textbf{Costability}

\begin{tikzpicture}
  \node[draw, circle] (9) at (0.8,2.6) {$-10$};
  \node (2) at (0,1) {$4$};
  \node (1) at (0,0) {$8$};
  \node (4) at (2,1) {$1$};
  \node (3) at (2,0) {$6$};
  \node (8) at (4,3) {$2$};
  \node (7) at (4,2) {$4$};
  \node (6) at (4,1) {$6$};
  \node (5) at (4,0) {$9$};
  \draw (5) -- (6) -- (7) -- (8);
  \draw (5) -- (4) -- (3);
  \draw (4) -- (1) -- (2);
\end{tikzpicture}
\end{tabular}
\\ \hline
\begin{tabular}{p{12.5cm}}
\textbf{Maximal subcoordinate vectors}\newline \newline
$(1;0,0;0,0;1,0,0,0)$,\ $(1;0,0;0,1;0,1,0,0)$,\ $(1;0,0;1,0;0,0,0,1)$,\  \newline
$(1;0,1;0,1;0,0,0,1)$,\ $(1;0,1;0,1;0,0,1,0)$,\ $(1;1,0;0,0;0,0,0,0)$,\  \newline
$(2;0,0;0,1;1,1,0,0)$,\ $(2;0,0;1,0;1,0,0,1)$,\ $(2;0,0;2,0;0,0,0,1)$,\  \newline
$(2;0,1;0,1;1,0,0,1)$,\ $(2;0,1;0,1;1,0,1,0)$,\ $(2;0,1;0,2;0,1,1,0)$,\  \newline
$(2;0,1;1,1;0,0,1,1)$,\ $(2;0,1;1,1;0,1,0,1)$,\ $(2;0,2;0,2;0,0,1,1)$,\  \newline
$(2;1,0;0,0;1,0,0,0)$,\ $(2;1,0;0,1;0,1,0,0)$,\ $(2;1,0;1,0;0,0,0,1)$,\  \newline
$(2;1,0;1,0;0,0,1,0)$,\ $(2;1,1;0,1;0,0,0,1)$,\ $(2;1,1;0,1;0,0,1,0)$,\  \newline
$(3;0,1;0,2;1,1,1,0)$,\ $(3;0,1;1,1;1,0,1,1)$,\ $(3;0,1;1,1;1,1,0,1)$,\  \newline
$(3;0,1;2,0;1,0,0,1)$,\ $(3;0,1;2,1;0,0,1,1)$,\ $(3;0,1;2,1;0,1,0,1)$,\  \newline
$(3;0,2;0,2;1,0,1,1)$,\ $(3;0,2;0,3;0,1,1,1)$,\ $(3;0,2;1,2;0,1,1,1)$,\  \newline
$(3;1,0;0,1;1,1,0,0)$,\ $(3;1,0;1,0;1,0,0,1)$,\ $(3;1,0;1,0;1,0,1,0)$,\  \newline
$(3;1,0;2,0;0,0,1,1)$,\ $(3;1,1;0,1;1,0,0,1)$,\ $(3;1,1;0,1;1,0,1,0)$,\  \newline
$(3;1,1;0,2;0,1,1,0)$,\ $(3;1,1;1,1;0,0,1,1)$,\ $(3;1,1;1,1;0,1,0,1)$,\  \newline
$(3;1,1;1,1;0,1,1,0)$,\ $(3;1,2;0,2;0,0,1,1)$,\ $(4;0,2;1,2;1,1,1,1)$,\  \newline
$(4;0,2;2,1;1,0,1,1)$,\ $(4;0,2;2,1;1,1,0,1)$,\ $(4;0,2;2,2;0,1,1,1)$,\  \newline
$(4;1,1;1,1;1,1,0,1)$,\ $(4;1,1;1,1;1,1,1,0)$,\ $(4;1,1;2,0;1,0,1,1)$,\  \newline
$(4;1,1;2,1;0,1,1,1)$,\ $(4;1,2;1,1;1,0,1,1)$,\ $(4;1,2;1,2;0,1,1,1)$.
\end{tabular}
\\ \hline
\end{longtable}

\newpage

\section*{Appendix C. Sincere representations of finite representation type, their maximal subdimension vectors and stability conditions.}

For each non-primitive posets $\Po_1,\dots,\Po_6$ of finite representation type (given in Section 3.1) we list its sincere representations, stability condition (calculated by formulas \eqref{formulaCanStab}) and maximal subdimension vectors. We use the same notation as in Appendix B.

\

Poset $\Po_1$.

\begin{longtable}{|c|}
\hline 
\begin{tabular}{p{6cm}|p{6cm}}
\textbf{Representation}
\begin{tikzpicture}
  \node[draw, circle] (9) at (2,2) {$K^3$};
  \node (2) at (0,1) {$K_{1,23}$};
  \node (1) at (0,0) {$K_{123}$};
  \node (4) at (2,1) {$K_{1,2}$};
  \node (3) at (2,0) {$K_{1}$};
  \node (6) at (4,1) {$K_{2,3}$};
  \node (5) at (4,0) {$K_{3}$};
  \draw (1) -- (2) -- (3) -- (4);
  \draw (5) -- (6);
\end{tikzpicture}
&
\textbf{Stability}
\begin{tikzpicture}
  \node[draw, circle] (9) at (2,2) {$-6$};
  \node (2) at (0,1) {$1$};
  \node (1) at (0,0) {$2$};
  \node (4) at (2,1) {$2$};
  \node (3) at (2,0) {$1$};
  \node (6) at (4,1) {$2$};
  \node (5) at (4,0) {$2$};
  \draw (1) -- (2) -- (3) -- (4);
  \draw (5) -- (6);
\end{tikzpicture}
\end{tabular}
\\ \hline
\begin{tabular}{p{12.5cm}}
\textbf{Maximal subdimension vectors}\newline \newline
$(1;0,0;0,0;1,1)$,\ $(1;0,0;0,1;0,1)$,\ $(1;0,1;0,0;0,1)$,\ $(1;0,1;1,1;0,0)$,\  \newline
$(1;1,1;0,0;0,0)$,\ $(2;0,1;0,1;1,2)$,\ $(2;0,1;1,1;1,1)$,\ $(2;0,1;1,2;0,1)$,\  \newline
$(2;1,1;0,1;1,1)$,\ $(2;1,2;1,1;0,1)$
\end{tabular}
\\ \hline
\end{longtable}

Poset $\Po_2$.

\begin{longtable}{|c|}
\hline 
\begin{tabular}{p{6cm}|p{6cm}}
\textbf{Representation}
\begin{tikzpicture}
  \node[draw, circle] (9) at (2,2) {$K^3$};
  \node (2) at (0,1) {$K_{13,2,4}$};
  \node (1) at (0,0) {$K_{123,24}$};
  \node (4) at (2,1) {$K_{1,4}$};
  \node (3) at (2,0) {$K_{4}$};
  \node (7) at (4,2) {$K_{1,2,3}$};
  \node (6) at (4,1) {$K_{2,3}$};
  \node (5) at (4,0) {$K_{3}$};
  \draw (1) -- (2) -- (3) -- (4);
  \draw (5) -- (6) -- (7);
\end{tikzpicture}
&
\textbf{Stability}
\begin{tikzpicture}
  \node[draw, circle] (9) at (2,2) {$-7$};
  \node (2) at (0,1) {$1$};
  \node (1) at (0,0) {$3$};
  \node (4) at (2,1) {$3$};
  \node (3) at (2,0) {$1$};
  \node (7) at (4,2) {$2$};
  \node (6) at (4,1) {$2$};
  \node (5) at (4,0) {$2$};
  \draw (1) -- (2) -- (3) -- (4);
  \draw (5) -- (6) -- (7);
\end{tikzpicture}
\end{tabular}
\\ \hline
\begin{tabular}{p{14cm}}
\textbf{Maximal subdimension vectors}\newline \newline
$(1;0,0;0,0;1,1,1)$,\ $(1;0,0;0,1;0,0,1)$,\ $(1;0,1;0,0;0,1,1)$,\ $(1;0,1;1,1;0,0,0)$,\  \newline
$(1;1,1;0,0;0,0,1)$,\ $(2;0,1;0,0;1,2,2)$,\ $(2;0,1;0,1;1,1,2)$,\ $(2;0,1;1,1;1,1,1)$,\  \newline
$(2;0,1;1,2;0,0,1)$,\ $(2;1,1;0,0;1,1,2)$,\ $(2;1,1;0,1;0,1,2)$,\ $(2;1,2;0,0;0,1,2)$,\  \newline
$(2;1,2;1,1;0,1,1)$,\ $(2;2,2;0,0;0,0,1)$,\ $(3;1,2;0,1;1,2,3)$,\ $(3;1,2;1,1;1,2,2)$,\  \newline
$(3;1,2;1,2;1,1,2)$,\ $(3;2,2;0,1;1,1,2)$,\ $(3;2,3;1,1;0,1,2)$. 
\end{tabular}
\\ \hline
\end{longtable}

\newpage

\begin{longtable}{|c|}
\hline 
\begin{tabular}{p{6cm}|p{6cm}}
\textbf{Representation}
\begin{tikzpicture}
  \node[draw, circle] (9) at (2,2) {$K^3$};
  \node (2) at (0,1) {$K_{12,13,4}$};
  \node (1) at (0,0) {$K_{124,13}$};
  \node (4) at (2,1) {$K_{1,2,4}$};
  \node (3) at (2,0) {$K_{4}$};
  \node (7) at (4,2) {$K_{1,2,3}$};
  \node (6) at (4,1) {$K_{2,3}$};
  \node (5) at (4,0) {$K_{3}$};
  \draw (1) -- (2) -- (3) -- (4);
  \draw (5) -- (6) -- (7);
\end{tikzpicture}
&
\textbf{Stability}
\begin{tikzpicture}
  \node[draw, circle] (9) at (2,2) {$-8$};
  \node (2) at (0,1) {$1$};
  \node (1) at (0,0) {$3$};
  \node (4) at (2,1) {$3$};
  \node (3) at (2,0) {$2$};
  \node (7) at (4,2) {$2$};
  \node (6) at (4,1) {$2$};
  \node (5) at (4,0) {$2$};
  \draw (1) -- (2) -- (3) -- (4);
  \draw (5) -- (6) -- (7);
\end{tikzpicture}
\end{tabular}
\\ \hline
\begin{tabular}{p{14cm}}
\textbf{Maximal subdimension vectors}\newline \newline
$(1;0,0;0,0;1,1,1)$,\ $(1;0,0;0,1;0,1,1)$,\ $(1;0,1;0,1;0,0,1)$,\ $(1;0,1;1,1;0,0,0)$,\  \newline
$(1;1,1;0,0;0,0,1)$,\ $(1;1,1;0,1;0,0,0)$,\ $(2;0,1;0,1;1,2,2)$,\ $(2;0,1;0,2;0,1,2)$,\  \newline
$(2;0,1;1,1;1,1,1)$,\ $(2;0,1;1,2;0,1,1)$,\ $(2;1,1;0,1;1,1,2)$,\ $(2;1,1;0,2;0,1,1)$,\  \newline
$(2;1,2;0,1;0,1,2)$,\ $(2;1,2;1,2;0,0,1)$,\ $(2;2,2;0,1;0,0,1)$,\ $(3;1,2;0,2;1,2,3)$,\  \newline
$(3;1,2;1,2;1,2,2)$,\ $(3;1,2;1,3;0,1,2)$,\ $(3;2,2;0,2;1,1,2)$,\ $(3;2,3;1,2;0,1,2)$. 
\end{tabular}
\\ \hline
\end{longtable}

Poset $\Po_3$

\begin{longtable}{|@{}c@{}|}
\hline 
\begin{tabular}{p{6cm}|p{6cm}}
\textbf{Representation}
\vspace*{0cm}%
\begin{tikzpicture}
  \node[draw, circle] (9) at (3,3) {$K^4$};
  \node (3) at (0,2) {$K_{1,3,4}$};
  \node (2) at (0,1) {$K_{1,4}$};
  \node (1) at (0,0) {$K_{4}$};
  \node (6) at (2,2) {$K_{1,2,3}$};
  \node (5) at (2,1) {$K_{2,3}$};
  \node (4) at (2,0) {$K_{3}$};
  \node (7) at (4,0) {$K_{123,24}$};
  \draw (1) -- (2) -- (3) -- (4) -- (5) -- (6);
\end{tikzpicture}
&
\textbf{Stability}
\begin{tikzpicture}
  \node[draw, circle] (9) at (3,3) {$-8$};
  \node (3) at (0,2) {$1$};
  \node (2) at (0,1) {$2$};
  \node (1) at (0,0) {$2$};
  \node (6) at (2,2) {$2$};
  \node (5) at (2,1) {$2$};
  \node (4) at (2,0) {$1$};
  \node (7) at (4,0) {$4$};
  \draw (1) -- (2) -- (3) -- (4) -- (5) -- (6);
\end{tikzpicture}
\end{tabular}
\\ \hline
\begin{tabular}{p{14cm}}
\textbf{Maximal subdimension vectors}\newline \newline
$(1;0,0,0;0,0,1;1)$,\ $(1;0,0,1;1,1,1;0)$,\ $(1;0,1,1;0,0,1;0)$,\ $(1;1,1,1;0,0,0;0)$,\  \newline
$(2;0,0,1;0,0,1;2)$,\ $(2;0,0,1;1,1,2;1)$,\ $(2;0,0,1;1,2,2;0)$,\ $(2;0,1,1;0,1,2;1)$,\  \newline
$(2;0,1,2;1,1,2;0)$,\ $(2;1,1,1;0,1,1;1)$,\ $(2;1,1,2;1,1,1;0)$,\ $(2;1,2,2;0,0,1;0)$,\  \newline
$(3;0,1,2;1,1,2;2)$,\ $(3;0,1,2;1,2,3;1)$,\ $(3;1,1,2;0,1,2;2)$,\ $(3;1,1,2;1,2,2;1)$,\  \newline
$(3;1,2,3;1,1,2;1)$.
\end{tabular}
\\ \hline
\end{longtable}

\newpage
Poset $\Po_4$
\begin{longtable}{|@{}c@{}|}
\hline
\begin{tabular}{|p{6cm}|p{6cm}|}
\textbf{Representation}
\vspace*{0cm}%
\begin{tikzpicture}
  \node[draw, circle] (9) at (2,2.5) {$K^4$};
  \node (3) at (0,2) {$K_{1,23,4}$};
  \node (2) at (0,1) {$K_{123,4}$};
  \node (1) at (0,0) {$K_{4}$};
  \node (5) at (2,1) {$K_{1,2,4}$};
  \node (4) at (2,0) {$K_{14}$};
  \node (8) at (4,2) {$K_{1,2,3}$};
  \node (7) at (4,1) {$K_{2,3}$};
  \node (6) at (4,0) {$K_{3}$};
  \draw (1) -- (2) -- (3) -- (4) -- (5) -- (1);
  \draw (6) -- (7) -- (8);
\end{tikzpicture}
&
\textbf{Stability}
\begin{tikzpicture}
  \node[draw, circle] (9) at (2,2.5) {$-7$};
  \node (3) at (0,2) {$1$};
  \node (2) at (0,1) {$2$};
  \node (1) at (0,0) {$1$};
  \node (5) at (2,1) {$2$};
  \node (4) at (2,0) {$2$};
  \node (8) at (4,2) {$2$};
  \node (7) at (4,1) {$2$};
  \node (6) at (4,0) {$2$};
  \draw (1) -- (2) -- (3) -- (4) -- (5) -- (1);
  \draw (6) -- (7) -- (8);
\end{tikzpicture}
\end{tabular}
\\ \hline
\begin{tabular}{p{12cm}}
\textbf{Maximal subdimension vectors}\newline \newline
$(1;0,0,0;0,0;1,1,1)$,\ $(1;0,0,0;0,1;0,1,1)$,\ $(1;0,0,1;0,0;0,1,1)$,\  \newline
$(1;0,0,1;0,1;0,0,1)$,\ $(1;0,0,1;1,1;0,0,0)$,\ $(1;0,1,1;0,0;0,0,1)$,\  \newline
$(1;1,1,1;0,1;0,0,0)$,\ $(2;0,0,1;0,1;1,2,2)$,\ $(2;0,0,1;0,2;0,1,2)$,\  \newline
$(2;0,0,1;1,1;1,1,1)$,\ $(2;0,0,1;1,2;0,1,1)$,\ $(2;0,1,1;0,1;1,1,2)$,\  \newline
$(2;0,1,2;0,1;0,1,2)$,\ $(2;0,1,2;1,1;0,1,1)$,\ $(2;1,1,1;0,1;1,1,1)$,\  \newline
$(2;1,1,1;0,2;0,1,1)$,\ $(2;1,1,2;0,1;0,1,1)$,\ $(2;1,1,2;1,2;0,0,1)$,\  \newline
$(2;1,2,2;0,1;0,0,1)$,\ $(3;0,1,2;0,2;1,2,3)$,\ $(3;0,1,2;1,2;1,2,2)$,\  \newline
$(3;1,1,2;0,2;1,2,2)$,\ $(3;1,1,2;1,2;1,1,2)$,\ $(3;1,1,2;1,3;0,1,2)$,\  \newline
$(3;1,2,2;0,2;1,1,2)$,\ $(3;1,2,3;1,2;0,1,2)$. 
\end{tabular}
\\ \hline
\end{longtable}

Poset $\Po_5$. 

\begin{longtable}{|@{}c@{}|}
\hline 
\begin{tabular}{p{6cm}|p{6cm}}
\textbf{Representation}
\begin{tikzpicture}
  \node[draw, circle] (9) at (0.8,3) {$K^5$};
  \node (2) at (0,1) {$K_{1,2,4,5}$};
  \node (1) at (0,0) {$K_{15,4}$};
  \node (4) at (2,1) {$K_{123,24,5}$};
  \node (3) at (2,0) {$K_{5}$};
  \node (8) at (4,3) {$K_{1,2,3,5}$};
  \node (7) at (4,2) {$K_{1,2,3}$};
  \node (6) at (4,1) {$K_{2,3}$};
  \node (5) at (4,0) {$K_3$};
  \draw (5) -- (6) -- (7) -- (8);
  \draw (4) -- (3) -- (8);
  \draw (1) -- (2) -- (3);
\end{tikzpicture}
&
\textbf{Stability}
\begin{tikzpicture}
  \node[draw, circle] (9) at (0.8,3) {$-9$};
  \node (2) at (0,1) {$4$};
  \node (1) at (0,0) {$2$};
  \node (4) at (2,1) {$3$};
  \node (3) at (2,0) {$1$};
  \node (8) at (4,3) {$1$};
  \node (7) at (4,2) {$2$};
  \node (6) at (4,1) {$2$};
  \node (5) at (4,0) {$2$};
  \draw (5) -- (6) -- (7) -- (8);
  \draw (4) -- (3) -- (8);
  \draw (1) -- (2) -- (3);
\end{tikzpicture}
\end{tabular}
\\ \hline
\begin{tabular}{p{12.5cm}}
\textbf{Maximal subdimension vectors}\newline \newline
$(1;0,0;0,0;1,1,1,1)$,\ $(1;0,0;0,1;0,0,1,1)$,\ $(1;0,1;0,0;0,1,1,1)$,\  \newline
$(1;0,1;1,1;0,0,0,1)$,\ $(1;1,1;0,0;0,0,0,1)$,\ $(2;0,1;0,0;1,2,2,2)$,\  \newline
$(2;0,1;0,1;1,1,2,2)$,\ $(2;0,1;1,1;1,1,1,2)$,\ $(2;0,1;1,2;0,0,1,2)$,\  \newline
$(2;0,2;0,0;0,1,2,2)$,\ $(2;0,2;1,1;0,1,1,2)$,\ $(2;0,2;1,2;0,0,0,1)$,\  \newline
$(2;1,1;0,0;1,1,1,2)$,\ $(2;1,2;0,0;0,1,1,2)$,\ $(2;1,2;0,1;0,1,1,1)$,\  \newline
$(2;1,2;1,1;0,0,1,2)$,\ $(2;2,2;0,0;0,0,0,1)$,\ $(3;0,2;0,1;1,2,3,3)$,\  \newline
$(3;0,2;1,1;1,2,2,3)$,\ $(3;0,2;1,2;1,1,2,3)$,\ $(3;0,2;1,3;0,0,1,2)$,\  \newline
$(3;1,2;0,1;1,2,2,3)$,\ $(3;1,2;1,1;1,1,2,3)$,\ $(3;1,2;1,2;0,1,2,3)$,\  \newline
$(3;1,3;1,1;0,1,2,3)$,\ $(3;1,3;1,2;0,1,1,2)$,\ $(3;2,2;0,1;1,1,1,2)$,\  \newline
$(3;2,3;0,1;0,1,1,2)$,\ $(3;2,3;1,1;0,0,1,2)$,\ $(4;1,3;1,2;1,2,3,4)$,\  \newline
$(4;1,3;1,3;1,1,2,3)$,\ $(4;2,3;0,2;1,2,2,3)$,\ $(4;2,3;1,2;1,1,2,3)$,\  \newline
$(4;2,4;1,2;0,1,2,3)$. 
\end{tabular}
\\ \hline
\end{longtable}

Poset $\Po_6$. 

\begin{longtable}{|@{}c@{}|}
\hline 
\begin{tabular}{p{6cm}|p{6cm}}
\textbf{Representation}
\begin{tikzpicture}
  \node[draw, circle] (9) at (0.8,2.6) {$K^5$};
  \node (2) at (0,1) {$K_{1,2,5}$};
  \node (1) at (0,0) {$K_{5}$};
  \node (4) at (2,1) {$K_{13,23,4,5}$};
  \node (3) at (2,0) {$K_{134,235}$};
  \node (8) at (4,3) {$K_{1,2,3,4}$};
  \node (7) at (4,2) {$K_{2,3,4}$};
  \node (6) at (4,1) {$K_{3,4}$};
  \node (5) at (4,0) {$K_4$};
  \draw (5) -- (6) -- (7) -- (8);
  \draw (5) -- (4) -- (3);
  \draw (4) -- (1) -- (2);
\end{tikzpicture}
&
\textbf{Stability}

\begin{tikzpicture}
  \node[draw, circle] (9) at (0.8,2.6) {$-9$};
  \node (2) at (0,1) {$4$};
  \node (1) at (0,0) {$2$};
  \node (4) at (2,1) {$1$};
  \node (3) at (2,0) {$4$};
  \node (8) at (4,3) {$2$};
  \node (7) at (4,2) {$2$};
  \node (6) at (4,1) {$2$};
  \node (5) at (4,0) {$1$};
  \draw (5) -- (6) -- (7) -- (8);
  \draw (5) -- (4) -- (3);
  \draw (4) -- (1) -- (2);
\end{tikzpicture}
\end{tabular}
\\ \hline
\begin{tabular}{p{12.5cm}}
\textbf{Maximal subdimension vectors}\newline \newline
$(1;0,0;0,1;1,1,1,1)$,\ $(1;0,0;1,1;0,0,0,1)$,\ $(1;0,1;0,0;0,0,1,1)$,\  \newline
$(1;1,1;0,1;0,0,0,0)$,\ $(2;0,0;0,1;1,2,2,2)$,\ $(2;0,0;0,2;1,1,2,2)$,\  \newline
$(2;0,0;1,2;1,1,1,2)$,\ $(2;0,0;2,2;0,0,0,1)$,\ $(2;0,1;0,1;1,1,2,2)$,\  \newline
$(2;0,1;0,2;0,0,1,2)$,\ $(2;0,1;1,1;0,1,1,2)$,\ $(2;0,2;0,1;0,0,1,2)$,\  \newline
$(2;1,1;0,2;1,1,1,1)$,\ $(2;1,1;1,2;0,0,1,1)$,\ $(2;1,2;0,1;0,0,1,1)$,\  \newline
$(3;0,1;0,2;1,2,3,3)$,\ $(3;0,1;1,2;1,2,2,3)$,\ $(3;0,1;1,3;1,1,2,3)$,\  \newline
$(3;0,1;2,3;1,1,1,2)$,\ $(3;0,2;0,2;1,1,2,3)$,\ $(3;0,2;1,2;0,1,2,3)$,\  \newline
$(3;1,1;0,2;1,2,2,2)$,\ $(3;1,1;1,3;1,1,2,2)$,\ $(3;1,1;2,3;0,0,1,2)$,\  \newline
$(3;1,2;0,2;1,1,2,2)$,\ $(3;1,2;1,2;0,1,2,2)$,\ $(3;1,2;1,3;0,0,1,2)$,\  \newline
$(3;1,3;0,2;0,0,1,2)$,\ $(4;0,2;1,3;1,2,3,4)$,\ $(4;0,2;2,3;1,2,2,3)$,\  \newline
$(4;1,2;1,3;1,2,3,3)$,\ $(4;1,2;2,4;1,1,2,3)$,\ $(4;1,3;1,3;1,1,2,3)$.
\end{tabular}
\\ \hline
\end{longtable}

\newcommand{\etalchar}[1]{$^{#1}$}


\begin{thebibliography}{DHKK14}

\bibitem[BFK{\etalchar{+}}13]{BFKSY}
V.~Bondarenko, V.~Futorny, T.~Klimchuk, V.~Sergeichuk, and K.~Yusenko.
\newblock Systems of subspaces of a unitary space.
\newblock {\em Linear Algebra Appl.}, 438(5):2561--2573, 2013.

\bibitem[BM11]{BM11}
A.~Bayer and E.~Macr\`\i.
\newblock The space of stability conditions on the local projective plane.
\newblock {\em Duke Math. J.}, 160(2):263--322, 2011.

\bibitem[BPP16]{BPP16}
N.~Broomhead, D.~Pauksztello, and D.~Ploog.
\newblock Discrete derived categories {II}: the silting pairs {CW} complex and
  the stability manifold.
\newblock {\em J. Lond. Math. Soc. (2)}, 93(2):273--300, 2016.

\bibitem[Bri07]{Bridgeland07}
T.~Bridgeland.
\newblock Stability conditions on triangulated categories.
\newblock {\em Ann. of Math. (2)}, 166(2):317--345, 2007.

\bibitem[Bri17]{BR17}
T.~Bridgeland.
\newblock Scattering diagrams, {H}all algebras and stability conditions.
\newblock {\em Algebr. Geom.}, 4(5):523--561, 2017.

\bibitem[BST17]{BST17}
T.~Br\"{u}stle, D.~Smith, and H.~Treffinger.
\newblock Stability conditions, $\tau$-tilting theory and maximal green
  sequences.
\newblock 2017.
\newblock URL: \url{arXiv:1705.08227}.

\bibitem[CI18]{CI}
C.~Cavalcante and K.~Iusenko.
\newblock On dimension of poset variety.
\newblock {\em Linear Algebra Appl. In press.}, 2018.

\bibitem[DHKK14]{DHKK}
G.~Dimitrov, F.~Haiden, L.~Katzarkov, and M.~Kontsevich.
\newblock Dynamical systems and categories.
\newblock In {\em The influence of {S}olomon {L}efschetz in geometry and
  topology}, volume 621 of {\em Contemp. Math.}, pages 133--170. Amer. Math.
  Soc., Providence, RI, 2014.

\bibitem[dlPnS92]{dpSimson}
J.~A. de~la Pe\~na and D.~Simson.
\newblock Prinjective modules, reflection functors, quadratic forms, and
  {A}uslander-{R}eiten sequences.
\newblock {\em Trans. Amer. Math. Soc.}, 329(2):733--753, 1992.

\bibitem[Dol03]{Dolgachev}
I.~Dolgachev.
\newblock {\em Lectures on invariant theory}, volume 296 of {\em London
  Mathematical Society Lecture Note Series}.
\newblock Cambridge University Press, Cambridge, 2003.

\bibitem[Dro74]{Drozd}
Yu.~A. Drozd.
\newblock Coxeter transformations and representations of partially ordered
  sets.
\newblock {\em Funct. Anal. Appl.}, 8(3):219–225, 1974.

\bibitem[FJM08]{FJM08}
V.~Futorny, M.~Jardim, and A.~Moura.
\newblock On moduli spaces for abelian categories.
\newblock {\em Comm. Algebra}, 36(6):2171--2185, 2008.

\bibitem[Hal69]{Halmos}
P.~R. Halmos.
\newblock Two subspaces.
\newblock {\em Trans. Amer. Math. Soc.}, 144:381--389, 1969.

\bibitem[Har95]{Harris}
J.~Harris.
\newblock {\em Algebraic geometry}, volume 133 of {\em Graduate Texts in
  Mathematics}.
\newblock Springer-Verlag, New York, 1995.

\bibitem[HdlP02]{Hille}
L.~Hille and J.A. de~la Pe{\~n}a.
\newblock Stable representations of quivers.
\newblock {\em J. Pure Appl. Algebra}, 172(2-3):205--224, 2002.

\bibitem[Hu05]{Hu}
Y.~Hu.
\newblock Stable configurations of linear subspaces and quotient coherent
  sheaves.
\newblock {\em Q. J. Pure Appl. Math.}, 1(1):127--164, 2005.

\bibitem[Kin94]{King}
A.~D. King.
\newblock Moduli of representations of finite-dimensional algebras.
\newblock {\em Quart. J. Math. Oxford Ser. (2)}, 45(180):515--530, 1994.

\bibitem[Kle72]{Kleiner2}
M.~M. Kle{\u\i}ner.
\newblock Partially ordered sets of finite type.
\newblock {\em Zap. Nau\v cn. Sem. Leningrad. Otdel. Mat. Inst. Steklov.
  (LOMI)}, 28:32--41, 1972.
\newblock Investigations on the theory of representations.

\bibitem[KN79]{KempfNess}
G.~Kempf and L.~Ness.
\newblock The length of vectors in representation spaces.
\newblock In {\em Algebraic geometry ({P}roc. {S}ummer {M}eeting, {U}niv.
  {C}openhagen, {C}openhagen, 1978)}, volume 732 of {\em Lecture Notes in
  Math.}, pages 233--243. Springer, Berlin, 1979.

\bibitem[KNR06]{KruglyakNazarovaRoiter}
S.~A. Kruglyak, L.~A. Nazarova, and A.~V. Roiter.
\newblock Matrix problems in hilbert space.
\newblock 2006.
\newblock URL: \url{arXiv:math/0605728}.

\bibitem[Knu00]{Knu}
A.~Knutson.
\newblock The symplectic and algebraic geometry of {H}orn's problem.
\newblock {\em Linear Algebra Appl.}, 319(1-3):61--81, 2000.
\newblock Special Issue: Workshop on Geometric and Combinatorial Methods in the
  Hermitian Sum Spectral Problem (Coimbra, 1999).

\bibitem[KR05]{KruglyakRoiter}
S.~A.~Kruglyak and A.~V.~Roiter.
\newblock Locally scalar representations of graphs in the category of {H}ilbert
  spaces.
\newblock {\em Funct. Anal. Appl.}, 39(2):91--105, 2005.

\bibitem[KS14]{KS14}
M.~Kontsevich and Y.~Soibelman.
\newblock Wall-crossing structures in {D}onaldson-{T}homas invariants,
  integrable systems and mirror symmetry.
\newblock In {\em Homological mirror symmetry and tropical geometry}, volume~15
  of {\em Lect. Notes Unione Mat. Ital.}, pages 197--308. Springer, Cham, 2014.

\bibitem[Lad08]{ladkani}
S.~Ladkani.
\newblock On the periodicity of {C}oxeter transformations and the
  non-negativity of their {E}uler forms.
\newblock {\em Linear Algebra Appl.}, 428(4):742--753, 2008.

\bibitem[Mac07]{Macri07}
E.~Macr\`\i.
\newblock Stability conditions on curves.
\newblock {\em Math. Res. Lett.}, 14(4):657--672, 2007.

\bibitem[MFK94]{mfk}
D.~Mumford, J.~Fogarty, and F.~Kirwan.
\newblock {\em Geometric invariant theory}, volume~34.
\newblock Springer-Verlag, Berlin, third edition, 1994.

\bibitem[MS17]{MacriShmidt}
E.~Macr\`\i and B.~Schmidt.
\newblock Lectures on {B}ridgeland stability.
\newblock In {\em Moduli of curves}, volume~21 of {\em Lect. Notes Unione Mat.
  Ital.}, pages 139--211. Springer, Cham, 2017.

\bibitem[Rei08]{reineke}
M.~Reineke.
\newblock Moduli of representations of quivers.
\newblock In {\em Trends in representation theory of algebras and related
  topics}, EMS Ser. Congr. Rep., pages 589--637. Eur. Math. Soc., Z\"urich,
  2008.

\bibitem[Rud97]{Rud}
A.~Rudakov.
\newblock Stability for an abelian category.
\newblock {\em J. Algebra}, 197(1):231--245, 1997.

\bibitem[Sch91]{SC91}
A.~Schofield.
\newblock Semi-invariants of quivers.
\newblock {\em J. London Math. Soc. (2)}, 43(3):385--395, 1991.

\bibitem[Sch92]{sch}
A.~Schofield.
\newblock General representations of quivers.
\newblock {\em Proc. London Math. Soc. (3)}, 65(1):46--64, 1992.

\bibitem[Sim92]{SimsonB}
D.~Simson.
\newblock {\em Linear representations of partially ordered sets and vector
  space categories}, volume~4 of {\em Algebra, Logic and Applications}.
\newblock Gordon and Breach Science Publishers, Montreux, 1992.

\bibitem[Sim10]{simson}
D.~Simson.
\newblock Integral bilinear forms, {C}oxeter transformations and {C}oxeter
  polynomials of finite posets.
\newblock {\em Linear Algebra Appl.}, 433(4):699--717, 2010.

\bibitem[SY12]{SamYus}
Yu.~Samoilenko and K.~Yusenko.
\newblock Kleiner's theorem for unitary representations of posets.
\newblock {\em Linear Algebra Appl.}, 437(2):581--588, 2012.

\bibitem[WY13]{WY13}
T.~Weist and K.~Yusenko.
\newblock Unitarizable representations of quivers.
\newblock {\em Algebr. Represent. Theory}, 16(5):1349--1383, 2013.

\end{thebibliography}
\end{document}